\documentclass[10pt]{amsart}
\usepackage{amssymb,amsthm,amsmath,latexsym}
\usepackage{mathtools} 
\usepackage{mathrsfs}
\usepackage{physics} 
\usepackage{esint} 
\usepackage{braket} 
\usepackage{enumitem}
\usepackage[toc,page]{appendix} 
\usepackage{hyperref}
\usepackage[x11names]{xcolor}
\newtheorem{theorem}{Theorem}[section]
\newtheorem{proposition}{Proposition}[section]

\newtheorem{lemma}{Lemma}[section]

\newtheorem{definition}{Definition}[section]
\theoremstyle{remark}
\newtheorem{remark}{Remark}[section]

\numberwithin{equation}{section} \numberwithin{theorem}{section}
\numberwithin{proposition}{section} \numberwithin{lemma}{section}
\numberwithin{corollary}{section}
\numberwithin{definition}{section} \numberwithin{remark}{section}

\newcommand{\R}{\mathbb{R}}
\newcommand{\spt}{\mathrm{spt}}
\newcommand{\LL}{\mathcal{L}} 
\newcommand{\HH}{\mathcal{H}} 
\newcommand{\N}{\mathbb{N}} 
\newcommand{\mres}{\mathbin{\vrule height 1.6ex depth 0pt width
0.13ex\vrule height 0.13ex depth 0pt width 1.3ex}}

\newcommand{\U}{U}
\newcommand{\W}{W}
\newcommand{\WW}{W_{10}}
\newcommand{\ZZ}{Z}
\newcommand{\eaf}{\varepsilon_A}

\author[C. Labourie]{Camille Labourie}
\author[A. Lemenant]{Antoine Lemenant}

\address[C. Labourie]{Department of Mathematics, Friedrich-Alexander Universität Erlangen-Nürnberg.
Cauerstr. 11, D-91058 Erlangen, Germany}
\email{camille.labourie@fau.de}

\address[A. Lemenant]{Universit\'e de Lorraine -- CNRS, UMR 7502 IECL, BP 70239
54506 Vandoeuvre-lès-Nancy,  France}
\email{antoine.lemenant@univ-lorraine.fr}

\date{\today}

\title[Epsilon-regularity for Griffith almost-minimizers under a separating condition]{Epsilon-regularity for Griffith almost-minimizers in any dimension under a separating condition}

\begin{document}

\begin{abstract}
    In this paper we prove that if $(u,K)$ is an almost-minimizer of the Griffith functional and $K$ is $\varepsilon$-close to a plane in some ball $B\subset \R^N$ while separating the ball $B$ in two big parts, then $K$ is $C^{1,\alpha}$ in a slightly smaller ball. Our result contains and generalizes the 2 dimensional result of \cite{bil}, with a different and more sophisticate approach inspired by \cite{l3,l2}, using also \cite{Lab} in order to adapt a part of the argument to Griffith minimizers.
\end{abstract}

\maketitle
\tableofcontents

\section{Introduction}

The variational model of crack propagation introduced by Francfort and Marigo \cite{FM} is based upon  the idea of Griffith from the early 20th century, saying that the needed energy to produce a crack in an elastic material, is proportional to the surface area of the crack. This is how one can hope to produce mathematically, a crack set depending on time $K(t)$, coming from the quasistatic limit of discrete sets $K_n:=K(t_n)$ which minimizes at each time the stationary Griffith energy
$$\int_{\Omega \setminus K} \mathbb{C} e(u) \cdot e(u) \dd{x} + \mathcal{H}^{N-1}(K),$$
where $e(u) = (Du + Du^T)/2$ stands for the symmetric gradient of the elastic displacement $u:\Omega \to \R^N$,  $\mathcal{H}^{N-1}$ is the Haudorff measure and $\mathbb{C}$ are elliptic coefficients.

Despite of the similarity with the classical  Mumford-Shah functional, for which a huge literature gave a thorough description of the minimizers from the 90s, the Griffith functional is  more delicate due to the configurations $u$ with vector values, much different from the scalar case. Another difficulty is coming from the fact that the energy with $e(u)$ controls only the symmetric part of the gradient, and not the full gradient $Du$ which implies, in absence of good Korn type inequalities in the non regular domain $\Omega \setminus K$, some major technical problems. This is why the development of the tools to treat Griffith minimizers appeared relatively recently only, comparing to the standard theory of free-discontinuity problems. To mention a few,  the weak existence in $GSBD$ by Dal Maso in \cite{dalM}, the strong existence of a minimizer by Chambolle and Crismale \cite{vito}, in addition to the Ahlfors-regularity  by Chambolle, Conti and Iurlano \cite{CCI} using also a previus  work by Conti, Focardi and Iurlano \cite{CFI}.

The main question that we address in this paper concerns the  $C^1$ regularity of minimizers for the Griffith functional, in any dimension. This question is related to the so called Mumford-Shah conjecture, about the minimizers of the famous Mumford-Shah functionnal
$$\int_{\Omega\setminus K}|\nabla u|^2 \dd{x} + \mathcal{H}^1(K).$$
The question of Mumford and Shah (1989) which is still open, is to know whether the minimizers $K$ of the Mumford-Shah functional are formed by a finite number of $C^1$ curves (see for instance the review paper about this problem \cite{lreview2}).  Bonnet in \cite{b} showed that the answer is true if $K$ is assumed to be a connected set. Some precise $C^1$ estimates have also been  obtained by  David \cite{d5},   and a famous  partial  $C^1$ result in higher dimensions  is given by Ambrosio, Fusco, Pallara in  \cite{afp}. Later, an independent  proof and more precise in the special dimension $3$ case appeared in \cite{l2}. See also \cite{rigot,DPF,AFH,DLF1} for further developments.

However, while trying to attack the Griffith functional with the standard Mumford-Shah tools, we rapidly see that most of the arguments do not work on this variant with $e(u)$. The energy controls only the symmetric part of the gradient which does not control $u$ in general. For instance, the monotonicity of Bonnet does not apply in the vectorial context. Most of the extension technics do not work neither, and the co-area formula cannot be used anymore. All the competitors obtained by truncation or composition are forbidden, which makes the analysis highly difficult. For instance, the Euler-Lagrange formula is not available, which prevents to derive a tilt-excess estimate, a starting point  to apply the standard approach of Amborsio, Fusco and Pallara \cite{afp}.

Despite of theses difficulties, the second author together with Babadjian and Iurlano have recently proved  in \cite{bil}   a partial $C^1$ result on the singular set of a Griffith minimizer, in dimension 2, with the additional assumption that $K$ is connected. The proof does not extend to higher dimensions for several reasons. Later in \cite{ll}, the result has been exploited to improve the dimension of the singular set and integrability of the symmetrized gradient in dimension 2.

In the present paper, we extend the results of \cite{bil} to any dimension $N\geq 2$, by using a completely different approach. Since connectedness of $K$ has no powerful meaning in higher dimensions, we replace it by a separating assumption: we say that $K$ separates $B(x,r)$ when, possibly after rotating, the north pole and south pole lie in different connected components of $B(x,r)\setminus K$ (see Definition \ref{defi_separation} for a more precise statement).

Our $\varepsilon$-regularity result uses a quantities  that is usually called the bilateral flatness of $K$ defined by 

\begin{equation*}
    \beta_K(x_0,r_0) := r^{-1} \inf_P \max \Set{\sup_{y \in K \cap B(x_0,r_0)}{\rm dist }(y,P), \sup_{y \in P \cap B(x_0,r_0)}{\rm dist }(y,K)},
\end{equation*}
where the infimum is over all hyperplanes $P$ passing trough $x_0$.

Now here is our main result (we refer to {\color{black} Definition \ref{def_locmin} for the precise definition} of an almost-minimizer with gauge $h$ and to Definition \ref{defi_separation} for the separating condition).

\begin{theorem}
    For each choice of $\alpha \in (0,1)$, there exists constants $\varepsilon_0 \in (0,1)$, $\gamma \in (0,\alpha)$ and $c \in (0,1)$ (all depending on $N$, $\alpha$) such that the following holds.
    Let $(u,K)$ be a Griffith almost-minimizer with gauge $h(t) = h(1) t^{\alpha}$ in an open set $\Omega$.
    Let $x_0 \in K$, $r_0 > 0$ be such that $B(x_0,r_0) \subset \Omega$, $K$ separates $B(x_0,r_0)$ and
    \begin{equation*}
        \beta_K(x_0,r_0) + h(r_0) \leq \varepsilon_0.
    \end{equation*}
    Then $K$ is a smooth $C^{1,\gamma}$ embedded surface in $B(x_0,c r_0)$.
\end{theorem}

Notice  that without control on the flatness $\beta_K$, the set $K$ could be a triple junction  (three curves meeting with an angle $2\pi/3)$.

It is worth mentioning that  in dimension 2 when $K$ is connected, then the separating condition holds $\mathcal{H}^1$-a.e. on $K$ so that our result contains the one of \cite{bil} and  generalizes it to almost-minimizers instead of minimizers.   As a matter of fact, we also extend the results of  \cite{ll} to almost-minimizers as well.

Now in higher dimensions, it is not clear how to prove that, for an almost-minimizer (or even a true minimizer), the separating condition  holds almost everywhere. Therefore, our result does not directly imply that the singular set of a Griffith minimizer in dimension $N$ is $C^{1,\alpha}$ $\mathcal{H}^{N-1}$-a.e.
Proving that the separating condition holds almost everywhere is a difficult open problem because of the lack of co-area formula for the symmetrized gradient.
However, we believe our result to be a step toward the full regularity of minimizer in any dimension.   

Next, let us say a few words about the proof. The strategy that we employ   is to follow the approach introduced in \cite{l3,l2} for the Mumford-Shah functional. The main idea, that was suggested  by Guy David to the second author during his thesis, is to use a stopping time argument on the flatness to identify a region where $K$ would be ``good'' (which means $\varepsilon$-flat at every scale) and  another region where $K$ would be ``bad'' (i.e. stops being flat at some scale), which is performed in Section \ref{section5}. The main point is then to estimate the size of the bad region, denoted by $m(r)$. This is done by use of a compactness argument (Lemma \ref{lem_flatness}) that says that if $K$ stops being flat at some scale, then one can win quantitatively some surface area in that scale,  which is one of the key ingredients in the construction of a competitor. In this compactness argument we had to use a different argument compared to the scalar scale  \cite{l3,l2} to avoid using the uniform concentration property, which is not known for our set $K$.

Furthermore, the general strategy works by use of another very important ingredient:  an extension tool for the function $u$.  Since we control approximatively the geometry of $K$ at every working scale thanks to our stopping time function, we can extend $u$ near $K$ by replacing it with averages and use a partition of unity in order to obtain a well defined function while $K$ has been modified as a competitor. We have written a general statement encoding this procedure in Lemma \ref{lem_extension}, that we use later several times for each competitor that we create.

Gathering together the stopping time argument, the compactness lemma and the extension tool, we are able to prove that the size of the bad set in a ball $B_r$, denoted by  $m(r)$, has a decaying property involving the normalized elastic energy $\omega(r)$ (see Proposition \ref{prop_badmass_decay})  defined by 

$$
\omega(x_0,r_0) = \frac{1}{r_0^{N-1}}\int_{B(x_0,r_0) \setminus K} \abs{e(u)}^2 \dd{x}.
$$

Then we need to prove that the normalized elastic energy has itself a decaying property to bootstrap the estimates. The decay of energy is proved by use of a compactness argument which is the purpose of Proposition \ref{prop_energy_decay}. In this proof we benefit from the separating property in order to use a jump-transfer technic. But we also use the sophisticated extension tool (Lemma \ref{lem_extension}) in a subtile manner in order to gain some closure property of a contradicting sequence. The proof is therefore totally different from the two dimensional argument of  \cite{bil}, which uses the Airy function (available only in dimension 2).

At the end we prove that all the quantities $\omega(r)$ and $m(r)$ are decaying like a power $r^\alpha$ whereas $\beta(r)$ stays small, which implies that $K$ is actually an almost-minimal set with gauge of order $r^\alpha$, which leads to the conclusion.

As the strategy is similar to the one of \cite{l3,l2}, it is quite probable that in dimension 3 an analogous $\varepsilon$-regularity result near minimal cones of type $\mathbb{Y}$ and $\mathbb{T}$ would be available  using a variant of our work. But for sake of simplicity we have written here only the case of flat cone $\mathbb{P}$ (hyperplane).

Finally let us mention some main differences with the work in \cite{l3,l2}.   Firstly let us say that  our paper is totally self-contained and no statement from \cite{l3,l2} has been directly used. Everything has been re-written with full details, sometimes quite differently and simplified, and we hope that the tools developed in here could be useful for other purposes. In addition, most of the time we had to adapt to the Griffith functional  in a nontrivial manner some of the arguments used in \cite{l3,l2}.

For instance, in the scalar case it is known that the Mumford-Shah minimizers have the ``uniform concentration'' property. Since the proof of that fact relies on the co-area formula, this is a difficult open question for Griffith minimizers, and prevents us to easily  obtain semi-continuity behavior of the surface area for a sequence of Griffith minimizers. Since this was one of the key ingredients in the compactness argument in Lemma \ref{lem_flatness}, we have used a different strategy following the results of the first author in \cite{Lab}. Another difference comes in the proof of Proposition  \ref{prop_energy_decay}, where we proceed differently in order to avoid competitors of the form $\chi(u_n)$, which are not admissible when dealing with the Griffith energy.

%
%
\section{Preliminaries}

Our working space is an open set $\Omega \subset \R^N$, where $N \geq 2$.
We say that a constant is \emph{universal} when it depends only on $N$.
We introduce a few definitions.

\medskip

\noindent{\bf (Coral) pairs.}
We define an \emph{admissible pair} as a pair $(u,K)$ such that $K$ is a relatively closed subset of $\Omega$ and $u \in W^{1,2}_{\mathrm{loc}}(\Omega \setminus K;\R^N)$.
We say that the pair is \emph{coral} if for all $x \in K$ and for all $r > 0$,
\begin{equation*}
    \HH^{N-1}(K \cap B(x,r)) > 0,
\end{equation*}
where $\HH^{N-1}$ is the Hausdorff measure of dimension $N-1$.
For any open set $V$, we define $LD(V)$ as the set of functions $u \in W^{1,2}_{\mathrm{loc}}(V)$ such that $\int_{V} \abs{e(u)}^2 \dd{x} < +\infty$.

\medskip

\noindent{\bf Competitors.}
Let $(u,K)$ be an admissible pair.
Let $x \in \Omega$ and $r > 0$ be such that $\overline{B}(x,r) \subset \Omega$.
A \emph{competitor} of $(u,K)$ in $B(x,r)$ is an admissible pair $(v,L)$ such that
\begin{equation*}
    L \setminus B(x,r) = K \setminus B(x,r) \quad \text{and} \quad v = u \quad \text{a.e. in} \quad \Omega \setminus \left(K \cup B(x,r)\right).
\end{equation*}

\medskip

\noindent{\bf Local minimizers and almost-minimizers.}
In general, a \emph{gauge} is a non-decreasing function $h : (0,+\infty) \to [0,+\infty]$ such that $\lim_{t \to 0^+} h(t) = 0$.
Our main theorem applies only with gauges of the form $h(t) = h(1) t^{\alpha}$, where $\alpha \in (0,1)$, but  it is only at the last section (Section 6) that this will be used.



\begin{definition}\label{def_locmin}
    A \emph{Griffith local almost-minimizer} with gauge $h$ in $\Omega$ is a coral admissible pair $(u,K)$, such that for all $x \in \Omega$, for all $r > 0$ with $\overline{B}(x,r) \subset \Omega$ and for all competitor $(v,L)$ of $(u,K)$ in $B(x,r)$, we have
    \begin{equation*}
        \int_{B(x,r) \setminus K} \abs{e(u)}^2 \dd{x} + \HH^{N-1}(K \cap B(x,r)) \leq \int_{B(x,r) \setminus L} \abs{e(v)}^2 \dd{x} + \HH^{N-1}(L \cap B(x,r)) + h(r) r^{N-1}, \label{eq_minimality0}
    \end{equation*}
    where $e(u)$ is the symmetrized gradient of $u$;
    \begin{equation*}
        e(u) := \frac{D u + D u^T}{2}.
    \end{equation*}
\end{definition}
Note that the definition is local since we only work in balls away from the boundary $\partial \Omega$. In the rest of the paper however, we will omit the word `local' for simplicity.

\begin{remark}
    Notice that adding to $K$ a negligible set does not affect the almost-minimality condition~\eqref{eq_minimality0}. This  explains why one needs to assume $K$ to be a coral set, as we did in Definition \ref{def_locmin}, in order to expect any  kind of regularity result.  
\end{remark}

\begin{remark}[Ahlfors-regularity]
    If $(u,K)$ is a Griffith almost-minimizer in $\Omega$, a standard comparison argument using the minimality condition (\ref{eq_minimality0}) shows that for all $x \in \Omega$ and $r > 0$ with $B(x,r) \subset \Omega$ and $h(r) \leq 1$, we have
    \begin{equation}\label{eq_AF0}
        \int_{B(x,r) \setminus K} \abs{e(u)}^2 \dd{x} + \HH^{N-1}(K \cap B(x,r)) \leq (\sigma_N + 1) r^{N-1},
    \end{equation}
    where $\sigma_N$ is the measure of the unit sphere.
    It also follows from a careful inspection of the proofs in \cite{CCI} (which are based on \cite{CFI}), that there exists universal constants $C \geq 1$ and $\varepsilon_A \in (0,1)$ such that the following holds. If $(u,K)$ is any Griffith almost-minimizer with any gauge $h$, then for all $x \in K$ and $r > 0$ such that $B(x,r) \subset \Omega$ and $h(r) \leq \varepsilon_A$, we have
    \begin{equation*}
        \HH^{N-1}(K \cap B(x,r)) \geq C^{-1} r^{N-1}.
    \end{equation*}
    In view of (\ref{eq_AF0}), we can assume $\varepsilon_A$ smaller and $C$ bigger so that for all $x \in K$ and $r > 0$ such that $B(x,r) \subset \Omega$ and $h(r) \leq \varepsilon_A$, we have
    \begin{equation}\label{eq_AF}
        C^{-1} r^{N-1} \leq \HH^{N-1}(K \cap B(x,r)) \leq C r^{N-1}.
    \end{equation}
    We will frequently need to assume $h(r) \leq \varepsilon_A$ in our statements to make sure that we can use the Ahlfors-regularity property (\ref{eq_AF}).
\end{remark}

\begin{remark}
    As proved in \cite{CCI}, an example of Griffith minimizer is given by $(u,\overline{J_u})$, where $u$ a GSBD-minimizer of the Griffith functional
    \begin{equation*}
        \int_{\Omega} \abs{e(u)}^2 \dd{x} + \HH^{N-1}(J_u)
    \end{equation*}
    with prescribed Dirichlet boundary condition.
\end{remark}


\begin{remark} Contrary to what commonly follows from the standard Mumford-Shah theory, and due to the absence of good $L^\infty$ estimates, it is not known whether  Griffith minimizers of a functional with an additional term of the form $\int_{\Omega} |u-g|^2 \dd{x}$ with $g\in L^\infty$ is an almost-minimizer. However, an example of almost-minimizer if for instance a minimizer of a functional of the form 
    $$\int_{\Omega \setminus K} |e(u)+A|^2 \dd x + \mathcal{H}^{N-1}(K)$$
    as studied recently in \cite{lempak}, for which the result of the present paper in full generality in dimension 2, is used.
\end{remark}

\begin{remark}
    Let $(u,K)$ be a Griffith almost-minimizer with gauge $h$ in $B(x_0,r_0)$.
    The rescaled pair of $(u,K)$ in $B(0,1)$ is given by $(v,L)$ such that
    \begin{equation*}
        v(x) = r_0^{-\frac{1}{2}} u(x_0 + r_0 x) ,\quad L = r_0^{-1}(K - x_0).
    \end{equation*}
    We observe that for all $x \in B(0,1)$ and $r > 0$ with $\overline{B}(x,r) \subset B(0,1)$ and for all competitor $(w,G)$ of $(v,L)$ in $B(x,r)$, we have
    \begin{equation*}
        \int_{B(x,r) \setminus L} \abs{e(v)}^2 \dd{x} + \HH^1(L \cap B(x,r)) \leq \int_{B(x,r) \setminus G} \abs{e(w)}^2 \dd{x} + \HH^1(G \cap B(x,r)) + h(r r_0) r.
    \end{equation*}
    Thus, $(v,L)$ an almost-minimizer with gauge $\tilde{h}(t) := h(r_0 t)$ in $B(0,1)$.
\end{remark}

We will introduce different quantities to study the Griffith almost-minimizers but all of them will be invariant under scaling (the flatness $\beta$, the normalized elastic energy $\omega$ and the bad mass $m$).

\medskip


\medskip

\noindent{\bf The normalized elastic energy.}
Let $(u,K)$ be an admissible pair.
For any $x_0 \in \Omega$ and $r_0 > 0$ such that $B(x_0,r_0) \subset \Omega$, we define the \emph{normalized elastic energy} of $u$ in $B(x_0,r_0)$ as
\begin{equation*}
    \omega_u(x_0,r_0) = \frac{1}{r_0^{N-1}}\int_{B(x_0,r_0) \setminus K} \abs{e(u)}^2 \dd{x}.
\end{equation*}
When there is no ambiguity, we write simply $\omega(x_0,r_0)$ instead of $\omega_u(x_0,r_0)$.

\begin{remark}
    We see that for all ball $B(x,r) \subset B(x_0,r_0)$, we have
    \begin{equation*}
        \omega(x,r) \leq \left(\frac{r_0}{r}\right)^{N-1} \omega(x_0,r_0).
    \end{equation*}
\end{remark}

\medskip

\noindent{\bf The bilateral flatness.}
Let $K$ be a relatively closed subset of $\Omega$.
For any $x_0 \in K$ and $r_0 > 0$ such that $B(x_0,r_0) \subset \Omega$, we define the {\it bilateral flatness} of $K$ in $B(x_0,r_0)$ by
\begin{equation}\label{eq_beta}
    \beta_K(x_0,r_0) := r^{-1} \inf_P \max \Set{\sup_{y \in K \cap B(x_0,r_0)}{\rm dist }(y,P), \sup_{y \in P \cap B(x_0,r_0)}{\rm dist }(y,K)},
\end{equation}
where the infimum is taken over all hyperplanes $P$ through $x_0$.
An equivalent definition is that $\beta_K(x_0,r_0)$ is the infimum of all $\varepsilon > 0$ for which there exists an hyperplane $P$ through $x_0$ such that
\begin{equation}\label{eq_beta2}
    \begin{gathered}
        K \cap B(x_0,r_0) \subset \set{y \in B(x_0,r_0) | \mathrm{dist}(y,P) \leq \varepsilon r_0}\\
        P \cap B(x_0,r_0) \subset \set{y \in B(x_0,r_0) | \mathrm{dist}(y,K) \leq \varepsilon r_0}.
    \end{gathered}
\end{equation}
It is easy to check that the infimum in (\ref{eq_beta}) and (\ref{eq_beta2}) is always attained.
When there is no ambiguity, we write simply $\beta(x_0,r_0)$ instead of $\beta_K(x_0,r_0)$.

\begin{remark}\label{rmk_beta}
    We see that for all $0 < r \leq r_0$,
    \begin{equation}\label{eq_beta_scaling}
        \beta(x_0,r) \leq \frac{r_0}{r} \beta(x_0,r_0)
    \end{equation}
    and for all ball $B(x,r) \subset B(x_0,r_0)$ centred on $K$, we also have
    \begin{equation*}
        \beta(x,r) \leq \frac{2r_0}{r} \beta(x_0,r_0).
    \end{equation*}
    Let $(K_i)_i$ and $K$ be relatively closed subsets of $\Omega$ containing $x_0$ such that $(K_i)_i$ converges to $K$ in local Hausdorff distance\footnote{It means that for all compact set $B \subset \Omega$, for all $\varepsilon > 0$, there exists $i_0$ such that for $i \geq i_0$, $K_i \cap B \subset \set{\mathrm{dist}(\cdot,K) \leq \varepsilon}$ and $K \cap B \subset \set{\mathrm{dist}(\cdot,K_i) \leq \varepsilon}$.
        It is also equivalent to the two equalities
        \begin{equation*}
            K = \set{x \in \Omega | \liminf_i \mathrm{dist}(x,K_i) = 0} = \set{x \in \Omega | \lim_i \mathrm{dist}(x,K_i) = 0}.
    \end{equation*}}.
    Then for all $0 < r < r_0$ such that $B(x_0,2r_0) \subset \Omega$, we have
    \begin{equation}\label{eq_beta_lim}
        \left(\frac{r}{r_0}\right) \limsup_i \beta_{K_i}(x_0,r) \leq \beta_K(x_0,r_0) \leq \liminf_i \beta_{K_i}(x_0,r_0).
    \end{equation}
    At the left-hand side of (\ref{eq_beta_lim}), we are forced to work with a radius $r < r_0$ because certain parts of $K_i \cap B(0,r_0)$ could contribute to $\beta_{K_i}(x_0,r_0)$ while converging to $\partial B(0,r_0)$ (so $\beta_K(x_0,r_0)$ won't see them).
\end{remark}







\medskip

\noindent{\bf The separating condition.}
The separation in a ball will be fundamental in our analysis.
We will consider the situation where $K \cap B(x_0,r_0)$ is contained in a narrow strip of thickness $\varepsilon r_0$ for some $\varepsilon \in ]0,1/2]$, i.e., there exists an hyperplane $P$ passing through $x_0$ such that
\begin{equation*}
    K \cap B(x_0,r_0) \subset \set{x \in B(x_0,r_0) | \mathrm{dist}(x,P) \leq \varepsilon r_0}.
\end{equation*}
We can then define two balls $D^+(x_0,r_0)$ and $D^{-}(x_0,r_0)$ of radius $r_0/4$ and such that $D^\pm(x_0,r_0) \subset B(x_0,r_0) \setminus K$.
Indeed, set $x_0^\pm := x_0 \pm (3/4) r_0 \nu$, where $\nu$ is a normal unit vector to $P$.
We can check that $D^\pm(x_0,r_0) := B(x_0^\pm,r_0/4)$ satisfy the above requirements.

\begin{definition}\label{defi_separation}
    Let $K$ be a relatively closed subset of $\Omega$, let $x_0 \in K$ and $r_0 > 0$ be such that $B(x_0,r_0) \subset \Omega$.
    We say that $K$ \emph{separates} $B(x_0,r_0)$ if there exists an hyperplane $P$ through $x_0$ satisfying
    \begin{equation*}
        K \cap B(x_0,r_0) \subset \set{x \in B(x_0,r_0) | \mathrm{dist}(x,P) \leq r_0/2}
    \end{equation*}
    and such that the corresponding balls $D^\pm(x_0,r_0)$ are contained in distincts connected components of $B(x_0,r_0) \setminus K$.
\end{definition}

\begin{remark}[The separating property does not depend on a particular hyperplane]
    If $K$ separates $B(x_0,r_0)$ with respect to an hyperplane $P$ passing through $x_0$ and if $Q$ is another hyperplane passing through $x_0$ satisfying
    \begin{equation*}
        K \cap B(x_0,r_0) \subset \set{x \in B(x_0,r_0) | \mathrm{dist}(x,Q) \leq r_0/2},
    \end{equation*}
    then $K$ also separates with respect to $Q$.
    For the proof, let $\nu_P$ and $\nu_Q$ be unit normal vectors to $P$ and $Q$ respectively.
    We orient them in such a way that $\nu_P \cdot \nu_Q \geq 0$.
    The sets
    \begin{equation*}
        \begin{gathered}
            \set{y \in B(x_0,r_0) | (y - x_0) \cdot \nu_P > r_0/2},\\
            \set{y \in B(x_0,r_0) | (y - x_0) \cdot \nu_Q > r_0/2}
        \end{gathered}
    \end{equation*}
    are connected, disjoint from $K$ and they have a nonempty intersection because $\nu_P \cdot \nu_Q \geq 0$.
    Therefore, their union $U^+$ is a connected subset of $B(x_0,r_0) \setminus K$.
    Similarly, the union of the sets
    \begin{equation*}
        \begin{gathered}
            \set{y \in B(x_0,r_0) | (y - x_0) \cdot \nu_P < -r_0/2},\\
            \set{y \in B(x_0,r_0) | (y - x_0) \cdot \nu_Q < -r_0/2},
        \end{gathered}
    \end{equation*}
    denoted by $U^-$, is a connected subset of $B(x_0,r_0) \setminus K$.
    As $K$ separates with respect to $P$, the sets $U^+$ and $U^-$ must be contained in distincts connected components of $B(x_0,r_0) \setminus K$ and we deduce that $K$ separates with respect to $Q$.
\end{remark}

\begin{remark}[Equivalence of unilateral flatness and bilateral flatness for separating sets]\label{rmk_beta_equivalence}
    If $B(x_0,r) \subset \Omega$, $K$ separates $B(x_0,r_0)$ and there exists $0 < \varepsilon \leq 1/2$ such that
    \begin{equation*}
        K \cap B(x_0,r_0) \subset \set{x \in B(x_0,r_0) | \mathrm{dist}(x, P) \leq \varepsilon r_0},
    \end{equation*}
    then we necessarily have
    \begin{equation*}
        P \cap B(x_0,r_0) \subset \set{x \in B(x_0,r_0) | \mathrm{dist}(x, K) \leq 2 \varepsilon r_0}.
    \end{equation*}
    Thus, for a separating set, the unilateral and bilateral flatness are equivalent by a factor $2$.
    For the proof, let $\nu$ be a unit vector $\nu$ to $P$.
    For all $x \in P \cap B(x_0,(1 - \varepsilon) r_0)$, the segment $x + [- \varepsilon r_0 \nu, \varepsilon r_0 \nu]$ is contained in $B(x_0,r_0)$ and it must meet $K$ otherwise it can be used to connect $D^\pm(x_0,r_0)$ in $B(x_0,r_0) \setminus K$.
    This shows that $\mathrm{dist}(x,K) \leq \varepsilon r_0$.
    And for all others $x \in P \cap B(x_0,r_0)$, we can find $y \in P \cap B(x_0,(1 - \varepsilon) r_0)$ such that $\abs{x - y} \leq \varepsilon r_0$ so $\mathrm{dist}(x,K) \leq 2 \varepsilon r_0$.
\end{remark}

The following lemma guarantees that when passing from a ball $B(x_0,r_0)$ to a smaller one $B(x, \tau r_0)$, the separation property is preserved if $\beta(x_0,r_0)$ is small enough compared to $\tau$.
We will see later a variant of this result, Lemma \ref{lem_orientation}, where the separation property is preserved in $B(x,r)$ if $\beta(x,t)$ stays small at all intermediate scales $t \in [r,r_0/4]$.
\begin{lemma}\label{lem_separation}
    Let $K$ be a relatively closed subset of $\Omega$, let $x_0 \in K$ and $r_0 > 0$ be such that $B(x_0,r_0) \subset \Omega$ and and $K$ separates $B(x_0,r_0)$.
    Then for all $x \in K \cap B(x_0,r_0)$, for all $\tau > 0$ such that $B(x, \tau r_0) \subset B(x_0,r_0)$ and $\beta(x_0,r_0) \leq \tau/4$, the set $K$ still separates $B(x,r)$.
\end{lemma}
\begin{proof}
    Let $P_0$ be a hyperplane passing through $x_0$ which achieves the minimum in the definition of $\beta(x_0,r_0)$, i.e.,
    \begin{equation*}
        K \cap B(x_0,r_0) \subset \set{x \in B(x_0,r_0) | \mathrm{dist}(x, P_0) \leq \varepsilon_0 r_0},
    \end{equation*}
    where $\varepsilon_0 := \beta(x_0,r_0)$.
    Let $P$ be the hyperplan parallel to $P_0$ and passing through $x$.
    Since $x$ is at distance $\leq \varepsilon_0 r_0$ from $P_0$, the hyperplane $P_0$ is also at distance $\leq \varepsilon_0 r_0$ from $P$ and it follows that
    \begin{align*}
        K \cap B(x,\tau r_0) &\subset \set{y \in B(x,\tau r_0) | \mathrm{dist}(y, P_0) \leq \varepsilon_0 r_0}\\
                             &\subset \set{y \in B(x,\tau r_0) | \mathrm{dist}(y, P) \leq \varepsilon \tau r_0}
    \end{align*}
    where $\varepsilon = 2 \varepsilon_0 / \tau \leq 1/2$.
    The two connected components of $B(x,r) \setminus \set{\mathrm{dist}(\cdot, P) \leq \varepsilon \tau r_0}$ are contained in respective connected components of $B(x_0,r_0) \setminus \set{\mathrm{dist}(\cdot,P) \leq \varepsilon_0 r}$.
    We deduce that $K$ separates them as well.
\end{proof}

\begin{definition}[Hypothesis-$H(\varepsilon_0,x_0,r_0)$]\label{def_hypH}
    Let $K$ be a relatively closed subset of $\Omega$, let $x_0 \in K$, $r_0 > 0$ and $\varepsilon_0 \in [0,1/2]$.
    We say that $K$ satisfies Hypothesis-$H(\varepsilon_0,x_0,r_0)$ if $B(x_0,r_0) \subset \Omega$ and if $K$ satisfies the two following assumptions:
    \begin{enumerate}[label=\roman*)]
        \item $K$ separates $B(x_0,r_0)$ (as in Definition \ref{defi_separation});
        \item $\beta(x_0,r_0)\leq \varepsilon_0$.
    \end{enumerate}
    In this case we can define $D^\pm(x_0,r_0)$ as in Definition \ref{defi_separation} and we can also define $\Omega_h(x_0,r_0)$, for $h=1,2$ (or simply $\Omega_h$), the two connected components of $B(x_0,r_0)\setminus K$ that contains $D^\pm(x_0,r_0)$, respectively. 
\end{definition}
The set $B(x_0,r_0) \setminus K$ might have other connected components than $\Omega_1$ and $\Omega_2$ but they are contained in a narrow strip of thickness $\varepsilon_0 r$ passing through the center of the ball.

\medskip

\noindent{\bf Geometric functions.}
Geometric functions are Lipschitz functions $\delta : K \cap \overline{B}(x_0,3r_0/4) \to [0,r_0/4]$ that we will use to build a covering $(B(x_i, \delta(x_i))_i$ of balls centred in $K$ with good overlapping properties.
We will take advantage of them in Section \ref{section_extension} to build extensions via a partition of unity.

\begin{definition}[Geometric function]
    Let $K$ be a relatively closed subset of $\Omega$, let $x_0 \in K$, $r_0 > 0$, $\varepsilon_0 \in [0,1/2]$ be such that $K$ satisfies Hypothesis-$H(\varepsilon_0,x_0,r_0)$.
    Let $\rho \in [r_0/2,3r_0/4]$ and $\tau \in [8\varepsilon_0, 1/2]$.
    Then we say that a $100$-Lipschitz function $$\delta: K \cap \overline{B}(x_0,\rho) \to [0,r_0/4]$$ is a geometric function with parameters $(\rho, \tau)$ if for all $x\in K \cap \overline{B}(x_0,\rho)$ and for all $r \in (\delta(x), r_0/4]$, we have
    \begin{equation}\label{defi_gf_beta}
        \beta(x,r)\leq \tau.
    \end{equation}
\end{definition}

Let us make a few comments.
We excluded the case $r = \delta(x)$ in \eqref{defi_gf_beta} because the flatness $\beta(x,\delta(x))$ would not be defined when $\delta(x) = 0$.
That being said, if $\delta(x) > 0$, we have necessarily $\beta(x,\delta(x)) \leq \tau$.
To justify it, we distinguish two cases.
If $0 < \delta(x) < r_0/4$, we use the fact that for all $r \in (\delta(x),r_0/4)$,
\begin{equation*}
    \beta(x,\delta(x)) \leq \frac{r}{\delta(x)} \beta(x,r) \leq \frac{r}{\delta(x)} \tau.
\end{equation*}
If $\delta(x) = r_0/4$, we have $\beta(x,\delta(x)) \leq \tau$ as well because of the scaling property of $\beta$ and because $\varepsilon_0 \leq \tau/8$.

A simple example of geometric function is the constant function $\delta = 2 \tau^{-1} \beta(x_0,r_0) r_0 \leq r_0/4$. We have indeed by the scaling property of $\beta$,
\begin{equation*}
    \beta(x,\delta) \leq 2 \left(\frac{r_0}{\delta}\right) \beta(x_0,r_0) \leq \tau.
\end{equation*}

\begin{remark}[The Lipschitz condition]\label{rmk_delta}
    The Lipschitz property of $\delta$ implies that two balls with nonempty intersection have comparable radii.
    More precisely, let us fix $0 \leq t \leq 1/300$ and let us consider $x, y \in K \cap \overline{B}(x_0,\rho)$ such that $\overline{B}(x,t \delta(x))$ and $\overline{B}(y, t \delta(y))$ meet. Then we have
    \begin{equation*}
        \abs{\delta(x) - \delta(y)} \leq 100 \abs{x - y} \leq  100(t\delta(x)+t\delta(y))\leq \tfrac{1}{3}\left( \delta(x) + \delta(y)\right)
    \end{equation*}
    whence
    \begin{equation}\label{eq_rmk_delta}
        \tfrac{1}{2} \delta(x) \leq \delta(y) \leq 2 \delta(x) \quad \text{and} \quad \abs{x - y} \leq 3 t \delta(x).
    \end{equation}
    In particular, for all $y \in K \cap \overline{B}(x_0,\rho) \cap \overline{B}(x,t \delta(x))$, we have (\ref{eq_rmk_delta}) as well.
\end{remark}

In Lemma \ref{lem_orientation} below, we will see that the condition \ref{defi_gf_beta} of geometric functions implies that in each ball $B(x,\delta(x))$, the set $K$ still separates and the two main components of $B(x,t \delta(x)) \setminus K$ are still subsets of $\Omega_1$ and $\Omega_2$ respectively.


\section{The Extension Lemma}\label{section_extension}


{\color{black}
    We start by motivating this section and introducing one of the main techniques of the next sections.
    We consider a Griffith minimizer $(u,K)$ in $\Omega$ and $x_0 \in K$, $r_0 > 0$ such that $B(x_0,r_0) \subset \Omega$ and $K$ separates $B(x_0,r_0)$.
    We let $L = f(K)$ be a deformation of $K$ such that $L = K$ in $\Omega \setminus B(x_0,r_0)$.
    We would like to use the minimality of $(u,K)$ to compare $\HH^{N-1}(K \cap B(x_0,r_0))$ and $\HH^{N-1}(L \cap B(x_0,r_0))$.
    If $K$ is a smooth surface, it could be possible to build a function $v \in W^{1,2}_{\mathrm{loc}}(\Omega \setminus L)$ such that $v = u$ in $\Omega \setminus B(x_0,r_0)$ and
    \begin{equation*}
        \int_{B(x_0,r_0) \setminus L} \abs{e(v)}^2 \dd{x} \leq C \int_{B(x_0,r_0) \setminus K} \abs{e(u)}^2 \dd{x}.
    \end{equation*}
    Then the energy comparison between $(u,K)$ and $(v,L)$ yield
    \begin{equation}\label{eq_energy_control}
        \HH^{N-1}(K \cap B(x_0,r_0)) \leq \HH^{N-1}(L \cap B(x_0,r_0)) + \omega(x_0,r_0) r_0^{N-1}.
    \end{equation}
    Here, we see that the energy control the minimality of $K$ under deformation.
    Once we will prove that $\omega$ decays as a power, (\ref{eq_energy_control}) will allow us to conclude that $K$ has an almost-minimal area under deformation. Then our $\varepsilon$-regularity theorem will follow from the regularity theory of almost-minimal sets.
    The goal of this section is to build such a function $v$ without assuming $K$ smooth a priori.
    The lack of regularity of $K$ will force us to add a certain ``wall set'' $\ZZ$ where the extension $v$ does not connect well with $u$ and where we cannot estimate its elastic energy.
}

To simplify, let us say that $K$ delimits two sides $\Omega_1$ and $\Omega_2$ in $B(x_0,r_0)$.
For some $\rho \in (1/2,3/4)$, we partially cover $K \cap \overline{B}(x_0,\rho)$ by a sequence of balls $(B_i)_i = (B(x_i,r_i))_i$ centred on $K$ which are induced by a geometric function (i.e., $r_i \sim \delta(x_i)$ for some geometric function $\delta$).
Then for each $h = 1,2$, we want to build a Sobolev function $v_h$ on $\Omega_h \cup \bigcup_i 10 B_i$ which is a kind of extension of $u\vert_{\Omega_h}$.
Without loss of generality, let us focus on $v_1$.
For an isolated ball $B_i$, we would define $v_1$ in $10 B_i \cup \Omega_1$ as follow. We set $v_1$ as a well-chosen rigid motion in $9 B_i$ and then $v_1 = u$ in $\Omega_1 \setminus 10 B_i$ whereas the part $\Omega_1 \cap 10 B_i \setminus 9 B_i$ is a transition area.
However, the lack of Korn-Poincaré inequality prevent us from estimating the energy of $v_1$ in the part of the transition area that is near $K$.
Therefore, we never consider isolated balls but we build $v_1$ via a partition of unity with respect to the family $(B_i)$, taking advantage of the overlapping properties induced by $\delta$.
Each of these balls is equipped with an orientation $\nu_i$, which is a unit vector pointing in the direction of $\Omega_1$.
We will need that when two balls $B_i$ and $B_j$ meet, their radii are comparable and their orientation are very close.
We will also need that for all ball $B_i$, the part of the transition area $\Omega_1 \cap 10 B_i \setminus 9 B_i$ at distance $\leq r_i$ from $K$ is covered by other balls $9 B_j$, with $j \ne i$. Thus, $v_1$ is an interpolation of ridig motion in the bad part the transition area near $K$ and this allows to estimate the elastic energy of $v$ without the Korn-Poincaré inequality.
However, we cannot estimate the energy of $v_1$ in the balls $B_i$ which meet $\partial B(x_0,\rho)$ because we cannot find other balls $9 B_j$ anymore to cover the bad part of the transition area. This limitation will appear in item \ref{lem_W_covering2} of Lemma \ref{lem_W}.
Therefore, our wall set $\ZZ$ will be composed of all the balls $B_i$ which meet $\partial B(x_0,\rho)$.
As said above, this is where we don't estimate the energy of $v$ and where $v$ might not connect with $u$.

{\color{black}
    Let us mention that there could be other ways to extend $v_1$ and that it is typical for extensions Lemma to add a wall set $\ZZ$ near $K \cap \partial B(x_0,\rho)$, for some $\rho \in (1/2,3/4)$, where the extension does not connect with the original function.
    The simplest example \cite[Lemma 4.2]{bil} which adds a wall set of size $\beta(x_0,r_0) r_0^{N-1}$, namely
    \begin{equation*}
        \set{x \in \partial B(x_0,\rho) | \mathrm{dist}(x,P) \leq \beta(x_0,r_0) r_0},
    \end{equation*}
    where $P$ is an hyperplane which achieves the infimum in the definition of $\beta(x_0,r_0)$.
    The measure of the wall set then appears as an error term in the estimates that involve the construction of a competitor by extension.
    We cannot use \cite[Lemma 4.2]{bil} however because an error term like $\beta(x_0,r_0) r_0^{N-1}$ would not be good enough to bootstrap our joint decay Lemma \ref{lem_decay}.
    In our extension technique, the size of the wall set $\ZZ$ is more finely controlled by the chosen geometric function $\delta$ inducing $(B_i)_i$.

    Taking $\delta \sim \beta(x_0,r_0) r_0$ would yield a wall set of size $\beta(x_0,r_0) r_0^{N-1}$ similar to \cite[Lemma 4.2]{bil}. But if $K$ is flat enough near $\partial B(x_0,\rho)$, we could rather consider a geometric function $\delta$ such that $\delta(x) \to 0$ when $x$ gets closer to $\partial B(x_0,\rho)$. The radius $r_i$ of the balls $B_i$ would tend to $0$ as $x_i \to \partial B(x_0,\rho)$ and possibly fast enough so that that no balls $B_i$ would meet $\partial B(x_0,\rho)$, resulting in an empty wall set $\ZZ$.
    However, we stress that the existence of this kind of geometric function is already a regularity property of $K$ near $\partial B(x_0,r_0)$.

    In Section \ref{section_badmass}, we will build a certain geometric function $\delta$ for which the size of the wall set $\ZZ$ is controlled by $\beta(x_0,r_0) m(x_0,r_0)$, where $m$ is a quantity called `bad mass'.
    More precisely, $\ZZ$ will be contained in an open set $O$ such that $\HH^{N-1}(\partial O) \leq C \beta(x_0,r_0) m(x_0,r_0)$. We will set $v = 0$ in $O$, add $\partial O$ to the crack and this will be penalized in the energy of $v$ by an error term $\beta(x_0,r_0) m(x_0,r_0)$.
    The bad mass can be interpreted as a way of measuring how $K$ differs from being Reifenberg-flat\footnote{We say that a relatively closed subset $K \subset B(x_0,r_0)$ is $\tau$-Reifenberg flat in $B(x_0,r_0)$ for some $\tau > 0$ provided that for all $x \in K \cap B(x_0,9r_0/10)$ and for all $0 < r \leq r_0/10$, we have $\beta_K(x,r) \leq \tau$.}. In the particular case where $K$ is Reifenberg-flat in $B(x_0,r_0)$, we have $m(x_0,r_0) = 0$ and the geometric function $\delta$ of Section \ref{section_badmass} does not induce a wall set.
}

We are now ready to write our extension Lemma.
Let $K$ be a relatively closed subset of $\Omega$.
Let $x_0 \in K$, $r_0 > 0$ and $\varepsilon_0 \in [0,1/2]$ be such that $K$ satisfies Hypothesis-$H(\varepsilon_0,x_0,r_0)$.
We consider two parameters $\rho \in [r_0/2, 3 r_0/4]$ and $\tau \in [8\varepsilon_0,1/2]$ and a geometric function $\delta$ with parameters $(\rho,\tau)$.
\begin{equation}\label{eq_U}
    \text{We fix $\U := 10^5$ and we assume in addition that $\tau \leq 10^{-8}$.}
\end{equation}

For $x \in \overline{B}(x_0,\rho)$, we set
\begin{equation*}
    r_x := \delta(x) / \U
\end{equation*}
and we define $B_x$ as the open ball $B(x,r_x)$, which is possibly empty.
We will work with balls like $B_x$, $10 B_x$, $50 B_x$ and they are small enough to apply Remark \ref{rmk_delta} because $50 r_x \leq \delta(x) / 300$.
We define
\begin{equation*}
    \W = \bigcup_{x \in K \cap \overline{B}(x_0,\rho)} B(x,r_x) \qquad \text{and} \qquad \WW = \bigcup_{x \in K \cap \overline{B}(x_0,\rho)} B(x,10 r_x)
\end{equation*}
They are open subsets of $B(x_0,r_0)$ containing $\set{x \in K \cap \overline{B}(x_0,\rho) | \delta(x) > 0}$ and by Remark \ref{rmk_delta}, they do not contain the points $x \in K \cap \overline{B}(x_0,\rho)$ such that $\delta(x) = 0$.
We introduce the open set
\begin{equation*}
    \ZZ := \bigcup \set{B(x,10 r_x) | x \in K \cap \overline{B}(x_0,\rho),\ 50 B(x,r_x) \cap \partial B(x_0,\rho) \ne \emptyset}
\end{equation*}
which will be our wall set.
And finally, we define for $h = 1,2$,
\begin{equation*}
    V_h  := \Omega_h \cup \W.
\end{equation*}
It is an open set of $B(x_0,r_0)$ which extends $\Omega_h$ by adding the nonempty balls $B_x$ centered on $K$.
We will see later in Lemma \ref{lem_W} that $W$ covers the connected components of $\overline{B}(x_0,r_0) \setminus K$ that are not contained in $\Omega_1$ or $\Omega_2$, i.e.,
\begin{equation*}
    \overline{B}(0,\rho) \setminus \left(K \cup \Omega_1 \cup \Omega_2\right) \subset W.
\end{equation*}
The rest of this section is devoted to prove the following Lemma.

\begin{lemma}[Extension Lemma]\label{lem_extension}
    Under the notation above, for all function $u \in LD(B(x_0,r_0) \setminus K)$ and for all $h = 1,2$, there exists a function $v_h \in LD_{\mathrm{loc}}(V_h)$ and a relatively relatively subset $S_h$ of $V_h$ such that
    \begin{equation*}
        \W \subset S_h \subset \WW,\qquad
        v_h = u \ \text{in} \ V_h \setminus S_h
    \end{equation*}
    and
    \begin{equation*}
        \int_{V_h \cap B(x_0,\rho) \setminus \ZZ} \abs{e(v_h)}^2 \dd{x} \leq C \int_{B(x_0,\rho) \cap \Omega_h} \abs{e(u)}^2 \dd{x},
    \end{equation*}
    for some universal constant $C \geq 1$.
\end{lemma}

{\color{black}
    Here one can think of $S_h$ as being essentially equal to $\WW$ but it will be more convenient to have $S_h$ being a relatively closed subset of $V_h$. The set $\ZZ$ plays the role of a domain around $K \cap \partial B(x_0,\rho)$ where we cannot estimate the energy of $v_h$, neither make sure that $v_h$ connects with $u_h$.
    In Section \ref{section_badmass}, we will build a geometric function $\delta$ for which the size of $\ZZ$ is related to the `bad mass', a quantity which measures how much $K$ differs from being Reifenberg-flat.
}

\begin{remark}
    Using the inclusion $S_h \subset \WW$ and the definition of $\ZZ$, we see that
    \begin{equation*}
        S_h \setminus B(x_0,\rho) \subset \ZZ
    \end{equation*}
    and as $v_h = u$ in $V_h \setminus S_h$, we deduce that
    \begin{equation*}
        v_h = u \quad \text{in} \quad V_h \setminus \left(B(x_0,\rho) \cup \ZZ\right).
    \end{equation*}
    We can therefore also control the energy of $v$ in the larger domain $V_h \setminus \ZZ$;
    \begin{equation*}
        \int_{V_h \setminus \ZZ} \abs{e(v_h)}^2 \dd{x} \leq C \int_{B(x_0,r_0)} \abs{e(u)}^2 \dd{x}.
    \end{equation*}
\end{remark}

\subsection{The Orientation Lemma} 

{\color{black}
    In the next Lemma, we assume that $K$ separates $B(x_0,r_0)$ and that $K \cap B(x_0,3r_0/4)$ is covered by balls $B(x,r)$ such that $\beta(x,t)$ stays small at all intermediate scales $t \in [r,r_0/4]$. We then show that in these balls, $K$ still separates and we can keep track of the $\Omega_1$ side and the $\Omega_2$ side.
    In fact, they cover all the connected components of $B(x_0,3r_0/4) \setminus K$ which are not $\Omega_1$ or $\Omega_2$.
    We show in addition that the approximation plane of $K$ does not change too much when passing from a ball $B(y,s)$ to a smaller ball $B(x,t)$ with comparable radius.
    This is related to the fact that when the bilateral flatness is small, all the approximation planes must be close to each other.
}
\begin{lemma}[Orientation Lemma]\label{lem_orientation}
    Let $K$ be a relatively closed subset of $\Omega$.
    Let $x_0 \in K$, $r_0 > 0$ and $\varepsilon_0 \in [0,1/2]$ be such that $K$ satisfies Hypothesis-$H(\varepsilon_0,x_0,r_0)$.
    Let $\Omega_1$, $\Omega_2$ be the connected components of $B(x_0,r_0) \setminus K$ introduced in Definition \ref{def_hypH}.
    We assume that there exists positive numbers $\rho \in [r_0/2, 3 r_0/4]$, $\varepsilon \in [8 \varepsilon_0,10^{-3}]$ and a function
    \begin{equation*}
        \begin{array}{rcl}
            K \cap \overline{B}(x_0,\rho) & \longrightarrow & [0,r_0/4] \\
            x & \longmapsto & r_x
        \end{array}
    \end{equation*}
    such that for all $x \in K \cap \overline{B}(x_0,\rho)$, for all $r \in ]r_x,r_0/4]$, we have
    \begin{equation*}
        \beta(x,r) \leq \varepsilon.
    \end{equation*}
    Then for all $x \in K \cap \overline{B}(x_0,\rho)$ and for all $r \in ]r_x,r_0/4]$, there exists a unit vector $\nu_x \in \mathbf{S}^{N-1}$ such that
    \begin{gather*}
        \set{y \in B(x,r) | (y - x) \cdot \nu_x > \varepsilon r} \subset \Omega_1\\
        \set{y \in B(x,r) | (y - x) \cdot \nu_x < -\varepsilon r} \subset \Omega_2.
    \end{gather*}
    and in particular
    \begin{equation*}
        K \cap B(x,r) \subset \set{y \in B(x,r) | \abs{(y - x) \cdot \nu_x} \leq \varepsilon r}
    \end{equation*}
    In addition, whenever for some $x, y \in K \cap \overline{B}(x_0,\rho)$ and $t \in ]r_x,r_0/4]$, $s \in ]r_y,r_0/4]$, we have $B(x,t) \subset B(y,s)$ and $t \geq s/10$, then $\nu_x$ and $\nu_y$ are close to each other in the sense that
    \begin{equation*}
        \nu_x \cdot \nu_y \geq 1 - 100 \varepsilon.
    \end{equation*}
    Finally, we have
    \begin{equation}\label{eq_reifenberg_inclusion}
        \overline{B}(x_0,\rho) \setminus \left(K \cup \Omega_1 \cup \Omega_2\right) \subset \bigcup_{x \in K \cap \overline{B}(0,\rho)} B(x,r_x).
    \end{equation}
\end{lemma}

\begin{proof}
    We fix $x \in K \cap \overline{B}(x_0,\rho)$. We start by justifying that the property holds in the ball $B(x,r_0/4)$. We proceed as in the usual proof of the inequality $\beta(x,r_0/4) \leq 8 \beta(x_0,r_0)$.
    According to Hypothesis-$H(\varepsilon_0,x_0,r_0)$, there exists a unit vector $\nu_0$ such that
    \begin{equation}\label{eq_P0_a}
        K \cap B(x_0,r_0) \subset \set{x \in B(x_0,r_0) | \abs{(x - x_0) \cdot \nu_0} \leq \varepsilon_0 r_0}.
    \end{equation}
    Moreover, we can orient $\nu_0$ in such a way that
    \begin{gather*}
        \set{y \in B(x_0,r_0) | (y - x_0) \cdot \nu_0 > \varepsilon_0 r_0} \subset \Omega_1\\
        \set{y \in B(x_0,r_0) | (y - x_0) \cdot \nu_0 < -\varepsilon_0 r_0} \subset \Omega_2.
    \end{gather*}
    Since $x \in K \cap B(x_0,r_0)$, we have $\abs{(x - x_0) \cdot \nu_0} \leq \varepsilon_0 r_0$ by (\ref{eq_P0_a}).
    For all $y \in B(x,r_0/4)$ such that $(y - x) \cdot \nu_0 > \varepsilon (r_0/4)$, we have $(y - x) > 2 \varepsilon_0 r_0$ because $\varepsilon_0 \geq 8 \varepsilon$ and we deduce that
    \begin{align*}
        (y - x_0) \cdot \nu_0   &= (y - x) \cdot \nu_0 + (x - x_0) \cdot \nu_0\\
                                &> \varepsilon_0 r_0,
    \end{align*}
    whence $y \in \Omega_1$. One can deal similarly with the other side.

    Next, we consider two points $x,y \in K \cap \overline{B}(x_0,\rho)$ and two radii $t \in ]r_x,r_0/4]$, $s \in ]r_y,r_0/4]$, such that $B(x,t) \subset B(y,s)$ and $t \geq s/10$.
    We assume that $B(y,s)$ satisfy the property and we are going to deduce that $B(x,t)$ satisfies it as well.
    Let $\nu_y$ be a unit vector such that
    \begin{equation*}
        K \cap B(y,s) \subset \set{z \in B(y,s) | \abs{(z - y) \cdot \nu_y} \leq \varepsilon s}
    \end{equation*}
    and
    \begin{gather*}
        A_1(y,s) := \set{z \in B(y,s) | (z - y) \cdot \nu_y > \varepsilon s} \subset \Omega_1\\
        A_2(y,s) := \set{z \in B(y,s) | (z - y) \cdot \nu_y < -\varepsilon s} \subset \Omega_2.
    \end{gather*}
    Since $\beta(x,t) \leq \varepsilon$, there also exists a unit vector $\nu_x$ such that
    \begin{equation}\label{eq:dist_Px}
        K \cap B(x,t) \subset \set{z \in B(x,t) | \abs{(z - x) \cdot \nu_x} \leq \varepsilon t}.
    \end{equation}
    Without loss of generality, we orient $\nu_x$ so that $\nu_x \cdot \nu_y \geq 0$.

    In this paragraph, we use repeateadly the fact that $\varepsilon \leq 10^{-3}$ without mention.
    As $x \in B(y,s)$, we know that $\abs{(x - y) \cdot \nu_y} \leq \varepsilon s$.
    For $u \in \nu_y^\perp \cap B(0,(1 - 20 \varepsilon) t)$, the segment
    \begin{equation*}
        (x + u) + [-20 \varepsilon t, 20 \varepsilon t] \nu_y
    \end{equation*}
    is contained in $B(x,t) \subset B(y,s)$ and it connects $A_1(y,s)$ to $A_2(y,s)$ (here we use $\abs{(x - y) \cdot \nu_y} \leq \varepsilon s$ and $t \geq s/10$).
    The segment crosses $K$ at some point $z$ otherwise $A_1(y,s)$ and $A_2(y,s)$ would be connected in $B(x_0,r_0) \setminus K$.
    Equation (\ref{eq:dist_Px}) says that $\abs{(z  - x) \cdot \nu_x} \leq \varepsilon t$ and it implies that $\abs{u \cdot \nu_x} \leq 21 \varepsilon t$.
    We conclude by homogeneity of the hyperplane $\nu_y^\perp$ that for all $u \in \nu_y^\perp$, we have
    \begin{equation*}
        \abs{u \cdot \nu_x} \leq \frac{21 \varepsilon}{1 - 20 \varepsilon} \abs{u} \leq 100 \varepsilon \abs{u}.
    \end{equation*}
    Taking in particular $u := \nu_x - (\nu_x \cdot \nu_y) \nu_y$ (the orthogonal projection of $\nu_x$ on $\nu_y^\perp$), we obtain
    \begin{equation*}
        1 - (\nu_x \cdot \nu_y)^2 \leq 100 \varepsilon.
    \end{equation*}
    Since $\nu_x \cdot \nu_y \geq 0$, we have
    \begin{equation*}
        1 - (\nu_x \cdot \nu_y)^2 = (1 - (\nu_x \cdot \nu_y)) (1 + (\nu_x \cdot \nu_y)) \geq 1 - \nu_x \cdot \nu_y,
    \end{equation*}
    whence the estimate on $\nu_x \cdot \nu_y$.
    Next, we want to show that
    \begin{gather*}
        A_1(x,t) := \set{z \in B(x,t) | (z - x) \cdot \nu_x > \varepsilon t}  \subset \Omega_1\\
        A_2(x,t) := \set{z \in B(x,t) | (z - x) \cdot \nu_x < -\varepsilon t} \subset \Omega_2.
    \end{gather*}
    The idea is that $A_1(x,t)$ and $A_1(y,s)$ are two connected subsets of $B(x_0,r_0) \setminus K$ with a nonempty intersection (try the point $x + (t/2) \nu_x$ with $\varepsilon$ small enough).
    Therefore, the union $A_1(x,t) \cup A_1(y,s)$ is also a connected subset of $B(x_0,r_0) \setminus K$.
    The union $A_1(x,t) \cup A_1(y,s)$ meets $\Omega_1$ because $A_1(y,s)$ does and since $\Omega_1$ is a connected component of $B(x_0,r_0) \setminus K$, it is completely contained in $\Omega_1$.

    Now we fix $x \in \overline{B}(x_0,\rho)$ and we want to show that for all $r \in ]r_x,r_0/4]$, the ball $B(x,r)$ satisfies the property. We can prove it by induction because it holds for $r = r_0/4$ and if it holds for some $r \in ]r_x,r_0/4]$, then it also holds for all $r' \in ]\max(r_x,r/10),r]$ by the previous paragraph, and we can iterate until we reach $r_x$.

    We come finally to the proof of the inclusion (\ref{eq_reifenberg_inclusion}).
    It relies on the following intermediate result.
    Let $x \in K \cap \overline{B}(x_0,\rho)$, let $r > 0$ be such that $B(x,r) \subset B(x_0,r_0)$ and let $\nu$ be a unit vector such that
    \begin{gather*}
        \set{y \in B(x,r) | (y - x) \cdot \nu > \varepsilon r} \subset \Omega_1\\
        \set{y \in B(x,r) | (y - x) \cdot \nu < -\varepsilon r} \subset \Omega_2.
    \end{gather*}
    Then, for all $y \in \overline{B}(x_0,\rho) \cap \overline{B}(x, 3r/4) \setminus (K \cup \Omega_1 \cup \Omega_2)$, there exists a point $x' \in K \cap \overline{B}(x_0,\rho)$ such that 
    \begin{equation*}
        \abs{y - x'} \leq 2 \varepsilon r.
    \end{equation*}
    The assumptions are satisfied for the ball $B(x,r) = B(x_0,r_0)$ and all the balls $B(x,r)$ with $x \in K \cap \overline{B}(x_0,\rho)$ and $r \in ]r_x,r_0/4]$.
    The proof of the intermediate result relies on a usual argument but we need to make sure that $x'$ stays in $\overline{B}(x_0,\rho)$.
    We consider $z := y - \varepsilon r (y - x_0) / \rho$, a small translation of $y$ which satisfies $\abs{z - y} \leq \varepsilon r$ and
    \begin{equation}\label{eq_z_inclusion}
        \overline{B}(z, \varepsilon r) \subset \overline{B}(x_0,\rho).
    \end{equation}
    If $[y,z]$ meet $K$, then the proof is finished. Otherwise $z$ stays in the same connected component of $B(x_0,r_0) \setminus K$ as $y$ so
    \begin{equation*}
        z \in B(x,r) \setminus \left(K \cup \Omega_1 \cup \Omega_2\right).
    \end{equation*}
    The path $z + [-\varepsilon r \nu,\varepsilon r \nu]$ is contained in $B(x,r)$ and, depending whether $(z - x) \cdot \nu \geq 0$ or $(z - x) \cdot \nu \leq 0$, it connects $z$ to $\Omega_1$ or $\Omega_2$. Since $z$ belongs to a connected component of $B(x_0,r_0) \setminus K$ distinct from $\Omega_1$ and $\Omega_2$, we deduce that the path must cross $K$.
    Hence, there exists $x' \in K$ such that $\abs{z - x'} \leq \varepsilon r$ and thanks to (\ref{eq_z_inclusion}), $x' \in \overline{B}(x_0,\rho)$. Then, we also have $\abs{x' - y} \leq 2 \varepsilon r$ as wanted.

    Now, we are ready for the proof of (\ref{eq_reifenberg_inclusion}). We consider $y \in \overline{B}(x_0,\rho) \setminus \left(K \cup \Omega_1 \cup \Omega_2\right)$ and we want to show that there exists $x \in K \cap \overline{B}(x_0,\rho)$ such that $\abs{x - y} < r_x$.
    We start by applying the intermediate result in $B(x,r) = B(x_0,r_0)$.
    There exists a point $x_1 \in K \cap \overline{B}(x_0,\rho)$ such that
    \begin{equation*}
        \abs{y - x_1} \leq 2 \varepsilon r_0.
    \end{equation*}
    If $2 \varepsilon r_0 < r_{x_1}$, then our claim is proved.
    Otherwise, the radius $r_1 := (4/3) 2 \varepsilon r_0$ belongs to $]r_{x_1},r_0/4]$ and we can apply the intermediate result in the ball $B(x_1,r_1)$; there exists $x_2 \in K \cap \overline{B}(x_0,\rho)$ such that
    \begin{equation*}
        \abs{y - x_2} \leq (4/3) (2 \varepsilon)^2 r_0.
    \end{equation*}
    If $(4/3) (2 \varepsilon)^2 r < r_{x_2}$, then our claim is proved.
    Otherwise, we iterate as before with the radius $r_2 := (4/3)^2 (2 \varepsilon)^2 r_0$ which belongs to $]r_{x_2},r_0/4]$.
    If the process never stops, there exists a sequence $(x_k)_{k \geq 1}$ in $K$ such that
    \begin{equation*}
        \abs{y - x_k} < (4/3)^{k-1} (2 \varepsilon)^k r_0
    \end{equation*}
    so $y \in K$. Contradiction.
\end{proof}

\subsection{Proof of the Extension Lemma}

\begin{proof}
    Throughout the proof, the letter $C$ is a constant $\geq 1$ that depends only on $N$ and whose value might change from one line to another.
    {\color{black}
        We recall that for $x \in \overline{B}(x_0,\rho)$, we set $r_x = \delta(x)/U$.
    }
    We consider a maximal family $(W_i)_{i \in I}$ in
    \begin{equation*}
        \set{B(x,r_x) | x \in K \cap \overline{B}(x_0,\rho),\ \delta(x) > 0}
    \end{equation*}
    such that whenever $i \ne j$,
    \begin{equation*}
        \tfrac{1}{10} W_i \cap \tfrac{1}{10} W_j \ne \emptyset.
    \end{equation*}
    We let $x_i \in K \cap \overline{B}(x_0,\rho)$ denotes the center of $W_i$ and $r_i > 0$ its radius, so that $W_i = B(x_i,r_i)$ with $r_i = \delta(x_i)/U$.

    \begin{lemma}\label{lem_W}
        Here we recall that $K$ satisfies hypothesis-$H(x_0,r_0,\varepsilon_0)$, we recall the open sets $\Omega_1$, $\Omega_2$ introduced in Definiton \ref{def_hypH} and we recall that $\U = 10^5$ and $8 \varepsilon_0 \leq \tau \leq 10^{-8}$.
        \begin{enumerate}[label=\roman*)]
            \item\label{lem_W_delta1}
                For all $x, y \in K \cap \overline{B}(x_0,\rho)$ such that $y \in 50 B_x$, we have $r_x/2 \leq r_y \leq 2 r_x$.

            \item\label{lem_W_delta2}
                For all $x, y \in K \cap \overline{B}(x_0,\rho)$ such that $50 B_x \cap 50 B_y \ne \emptyset$, we have $r_x/2 \leq r_y \leq 2 r_x$. Moreover, if $10 B_x \cap 10 B_y \ne \emptyset$, then $\abs{x - y} \leq 30 r_x$ and in particular, $10 B_y \subset 50 B_x$.

            \item\label{lem_W_orientation}
                For all $x \in K \cap \overline{B}(x_0,\rho)$, for all $r \in ]r_x,r_0/4]$, we have $\beta(x,r) \leq \tau U$.
                In particular, there exists a unit vector $\nu_x$ such that
                \begin{equation*}
                    K \cap B(x,r) \subset \set{y \in B(x,r) | \abs{(y - x) \cdot \nu_x} \leq \tau \U r}
                \end{equation*}
                and
                \begin{align*}
                &\set{y \in B(x,r) | (y  - x) \cdot \nu_x > \tau \U r} \subset \Omega_1\\
                &\set{y \in B(x,r) | (y  - x) \cdot \nu_x < -\tau \U r} \subset \Omega_2.
                \end{align*}
                We have also
                \begin{equation*}
                    \overline{B}(x_0,\rho) \setminus \left(K \cup \Omega_1 \cup \Omega_2\right) \subset W.
                \end{equation*}

            \item\label{lem_W_covering1}
                We have
                \begin{align*}
                    \set{y \in K \cap \overline{B}(x_0,\rho) | \delta(y) > 0}  &= \set{y \in K \cap \overline{B}(x_0,\rho) | y \in \bigcup_i W_i}\\
                                                                               &= \set{y \in K \cap \overline{B}(x_0,\rho) | \exists x \in K \cap \overline{B}(x_0,\rho),\ y \in 50 B_x}
                \end{align*}
                and for all $t > 0$,
                \begin{equation*}
                    \bigcup_{y \in K \cap \overline{B}(x_0,\rho)} \overline{B}(y, t r_x) \subset \bigcup_i (2t + 1) W_i.
                \end{equation*}

            \item\label{lem_W_covering2}
                Let $k$ be an index such that $50 W_k \subset B(x_0,\rho)$.
                Then for all $x \in 10 W_k$ such that $\mathrm{dist}(x,K) \leq 2 r_k$, there exists an index $i$ such that $x \in 9 W_i$.

            \item\label{lem_W_locally_finite}
                There exists a universal constant $C \geq 1$ such that for all $x \in K \cap \overline{B}(x_0,\rho)$ with $\delta(x) > 0$, there exists at most $C$ balls $50 W_i$ which meet $50 B_x$.

            \item\label{lem_W_locally_finite2}
                The family $(50 W_i)_i$ is locally finite in
                \begin{equation*}
                    B(x_0,r_0) \setminus \set{x \in K \cap \overline{B}(x_0,\rho) | \delta(x) = 0}.
                \end{equation*}

            \item\label{lem_W_pointwise_finite}
                There exists a universal constant $C \geq 1$ such that for all $x \in \overline{B}(x_0,\rho)$, there at most $C$ balls $50 W_i$ which contains $x$.
        \end{enumerate}
    \end{lemma}

    \begin{proof}\strut

        \emph{Item \ref{lem_W_delta1} and \ref{lem_W_delta2}.} 
        They are immediate consequences of Remark \ref{rmk_delta}.

        \medskip

        \emph{Item \ref{lem_W_orientation}.}
        We check that for all $x \in K \cap \overline{B}(x_0,\rho)$ and for all $r \in ]r_x,r_0/4]$, we have $\beta(x,r) \leq \tau \U$.
        If $r > \delta(x)$, then it holds by definition of the geometric function.
        Otherwise $r_x < r \leq \delta(x)$ and since $r_x = \delta(x)/U$, we get
        \begin{equation*}
            \beta(x, r) \leq \frac{\delta(x)}{r_x} \beta(x,\delta(x)) \leq \tau \U.
        \end{equation*}
        The properties stated in this item follows from an application of Lemma \ref{lem_orientation} with $\varepsilon := \tau U \leq 10^{-3}$.

        \medskip

        \emph{Item \ref{lem_W_covering1}.}
        For $y \in K \cap \overline{B}(x_0,\rho)$ such that $\delta(y) > 0$, there exists an index $i$ such that the ball $\tfrac{1}{10} B(y, r_y)$ meet $\tfrac{1}{10} W_i$.
        As in Remark \ref{rmk_delta}, $\abs{y - x_i} \leq 3 r_i / 10 < r_i$ so $y$ belongs to $W_i$. 
        Reciprocally, for $y \in K \cap \overline{B}(x_0,\rho)$ such that there exists $x \in K \cap \overline{B}(x_0,\rho)$ with $y \in 50 B_x$, we have $\delta(x) > 0$ (because $50 B_x$ is non empty) so $\delta(y) > 0$ by item \ref{lem_W_delta1}.
        Finally, for all $y \in K \cap \overline{B}(x_0,\rho)$, we have either $\delta(y) = 0$ and thus $B(y,t r_y) = \emptyset$ or there exists $i$ such that $\abs{y  - x_i} < r_i$ and $r_y \leq 2 r_i$ so $\overline{B}(y, t r_y) \subset (2t + 1) W_i$.

        \medskip

        \emph{Item \ref{lem_W_covering2}.}
        We fix $k$ such that $50 W_k \subset B(x_0,\rho)$.
        For $x \in 10 W_k$ such that $\mathrm{dist}(x,K) \leq 2 r_k$, there exists $y \in K$ such that $\abs{x - y} \leq 2 r_k$ and since $y \in 50 W_k \subset B(x_0,\rho)$, we have $r_k \leq 2 r_y$.
        We apply \ref{lem_W_covering1} to conclude that
        \begin{equation*}
            x \in \overline{B}(y, 2 r_k) \subset \overline{B}(y, 4 r_y) \subset \bigcup_i 9 W_i.
        \end{equation*}

        \medskip

        \emph{Item \ref{lem_W_locally_finite}.}
        We fix a point $x \in K \cap \overline{B}(x_0,\rho)$ such that $\delta(x) > 0$.
        Whenever a ball $50 W_i$ meet $50 B_x$, we have $r_i/2 \leq r_x \leq 2 r_i$ and $\abs{x - x_i} \leq 50 r_i + 50 r_x \leq 150 r_x$ so $W_i \subset 200 B_x$.
        Thus, the set
        \begin{equation}\label{eq_set}
            \Set{\tfrac{1}{10} W_i | 50 W_i \cap 50 B_x \ne \emptyset}
        \end{equation}
        is composed of disjoint balls whose radii are bounded from below by $r_x / 20$ and which are contained in $200 B_x$.
        We deduce that the cardinal of (\ref{eq_set}) is bounded by a universal constant.

        \emph{Item \ref{lem_W_locally_finite2}.}
        We want to prove that for all
        \begin{equation*}
            x \in B(x_0,r_0) \setminus \set{x \in K \cap \overline{B}(x_0,\rho) | \delta(x) = 0},
        \end{equation*}
        there exists a neighborhood of $x$ which meet only a finite number of balls $(50 W_i)_i$.
        According to item \ref{lem_W_covering1}, we have
        \begin{multline*}
            B(x_0,r_0) \setminus \set{x \in K \cap \overline{B}(x_0,\rho) | \delta(x) = 0} \\= \bigcup \Set{50 B_y | y \in K \cap B(x_0,\rho)} \cup B(x_0,r_0) \setminus \left[K \cap \overline{B}(x_0,\rho)\right].
        \end{multline*}
        If there exists $y \in K \cap \overline{B}(x_0,\rho)$ such that $x \in 50 B_y$, we must have in particular $\delta(y) > 0$ and then the previous item shows that $50 B_y$ has the required property.
        We now focus on the points $x \in B(x_0,r_0) \setminus \left[K \cap \overline{B}(x_0,\rho)\right]$.
        We proceed by contradiction and assume that there is a sequence of distincts elements $(W_{i_k})_{k \in \mathbf{N}}$ extracted from $(W_i)$ such that $\mathrm{dist}(x,50 W_{i_k}) \to 0$ when $k \to +\infty$.
        The balls $(1/10) W_{i_k}$ are disjoint and contained in $B(x_0,r_0)$ so their radii must go to $0$ as $k$ goes to $+\infty$.
        We deduce that $\mathrm{dist}(x,x_{i_k}) \to 0$, where $x_{i_k}$ denotes the center of $W_{i_k}$.
        But since the points $x_{i_k}$ belong to $K \cap B(x_0,\rho)$, we must have $x \in K \cap \overline{B}(x_0,\rho)$.
        Contradiction.

        \emph{Item \ref{lem_W_pointwise_finite}.}
        In view of the two previous items, it suffices to focus on the points $x \in K \cap \overline{B}(x_0,\rho)$ such that $\delta(x) = 0$.
        But there are no ball $50 W_i$ containing $x$ (otherwise $r_x \geq r_i / 2 > 0$).
    \end{proof}

    Let us draw two consequences of Lemma \ref{lem_W}.
    According to item \ref{lem_W_covering1}, we have
    \begin{equation}\label{eq_W_Wi}
        \W \subset \bigcup_i 3 W_i,
    \end{equation}
    where we recall that $W = \bigcup \set{B_x | x \in K \cap \overline{B}(x_0,\rho)}$.
    Next, we observe that according to item \ref{lem_W_locally_finite2}, the family $(50 W_i)$ is locally finite in each open set $V_h = \Omega_h \cup W$.

    For every ball $W_i$, we consider a function $\varphi_i \in C_c^\infty(\R^N)$ with compact support on $10 W_i$, equal to $1$ on $9 W_i$ and such that $0 \leq \varphi_i \leq 1$.
    We also choose $\varphi_i$ so that its Lipschitz constant is $\leq 2 r_i^{-1}$. Frow now on, we fix $h = 1,2$.
    Letting $\spt(\varphi_i)$ denote the compact support of $\varphi_i$ in $10 W_i$, we define
    \begin{equation}\label{eq_Sh}
        S_h := V_h \cap \bigcup_i \spt(\varphi_i).
    \end{equation}
    Since the family $(10 W_i)_i$ is locally finite in $V_h$, the set $S_h$ is a relatively closed subset of $V_h$. It is also clear using (\ref{eq_W_Wi}) and $\spt(\varphi_i) \subset 10 W_i$ that
    \begin{equation*}
        \W \subset S_h \subset \WW.
    \end{equation*}
    We are going to complete the family $(\varphi_i)_i$ with a function $\varphi_0$ to obtain a partition of unity in $V_h$.
    \begin{lemma}\label{lem_varphi}
        There is a function $\varphi_0 \in C^\infty(V_h)$ such that $0 \leq \varphi_0 \leq 1$,
        \begin{align*}
            \varphi_0  &= 1 \text{ in } V_h \setminus S_h,\\
            \varphi_0  &= 0 \text{ in } \bigcup_i 9 W_i
        \end{align*}
        and
        \begin{equation}\label{eq:phi2}
            V_h \cap \set{\varphi_0 \ne 0} \subset \Omega_h \setminus \bigcup_i 9 W_i.
        \end{equation}
        We also have
        \begin{equation*}
            \varphi_0 + \sum_{i \geq 1} \varphi_i \geq 1 \text{ in } V_h\label{eq:phi3}
        \end{equation*}
        and there is a universal constant $C \geq 1$ such that for all $k$, for all $x \in 10 W_k$,
        \begin{equation*}
            \abs{\nabla \varphi_0(x)} + \sum_{i \in I} \abs{\nabla \varphi_i(x)} \leq C r_k^{-1}.
        \end{equation*}
    \end{lemma}

    \begin{proof}
        The letter $C$ denotes a constant $\geq 1$ that depends only on $N$ but whose value might change from one line to another.
        We define 
        \begin{equation*}
            \varphi_0(x) = \prod_{i \in I} (1 - \varphi_i).
        \end{equation*}
        Thanks to item \ref{lem_W_locally_finite2} of Lemma \ref{lem_W}, we know that every point $x \in V_h$ admits a neighborhood which meets a finite number of balls $10 W_i$. 
        Therefore, $\varphi_0$ is well defined and smooth. We can say more about the Lipschitz constant of $\varphi_0$ in a fixed set $10 W_k$.
        According to \ref{lem_W_locally_finite}, there exists a universal constant $C \geq 1$ such that at most $C$ balls $10 W_i$ meet $10 W_k$. Moreover, we have in this case $r_k/2 \leq r_i \leq 2 r_k$.
        So in the set $10 W_k$, $\varphi_0$ is a product of $C$ functions which are $C r_k^{-1}$-Lipschitz and with range in $[0,1]$ so it itself $C r_k^{-1}$-Lipschitz. By the same argument, the sum $\sum_{i \in I} \varphi^j$ is $C r_k^{-1}$-Lipschitz in $10 W_k$. 
        It is straightforward that $\varphi_0 = 1$ in $V_h \setminus S_h$. It is also clear that $\varphi_0 = 0$ in $\bigcup_i 9 W_i$ whence
        \begin{equation*}
            V_h \cap \set{\varphi_0 \ne 0} \subset V_h \setminus \bigcup_i 9 W_i.
        \end{equation*}
        But $V_h = \left(\Omega_h \cup \W\right)$ and we have seen in (\ref{eq_W_Wi}) that $W \subset \bigcup_i 9 W_i$ so we have in fact
        \begin{equation*}
            V_h \cap \set{\varphi_0 \ne 0} \subset \Omega_h \setminus \bigcup_i 9 W_i.
        \end{equation*}
        The property (\ref{eq:phi3}) follows from the observation that for all sequence $(t_k)_{k \in \N}$ in $[0,1]$, we have
        \begin{equation*}
            \prod_k (1 - t_k) + \sum_k t_k \geq 1.
        \end{equation*}
        It can be proved by induction because for all $s,t \in [0,1]$, $s (1 - t) + t \geq s$. 
    \end{proof}

    We set for $i \in I$,
    \begin{equation*}
        \theta_i = \frac{\varphi_i}{\varphi_0 + \sum_{i \in I} \varphi_i}
    \end{equation*}
    and
    \begin{equation*}
        \theta_0 = \frac{\varphi_0}{\varphi_0 + \sum_{i \in I} \varphi_i}.
    \end{equation*}
    Now we have a smooth partition of unity on $V_h$. The information that we have gathered imply easily that for all $k \in I$, the function $\theta_i$ is $C r_k^{-1}$-Lipschitz in $V_h \cap 10 W_k$. 

    We are ready to build the functions $v_h$. For each $i$, there exists an hyperplane $P_i$ passing through $x_i$ such that
    \begin{equation*}
        K \cap 10 W_i \subset \set{x \in 10 W_i | \mathrm{dist}(x,P_i) \leq 10 \tau \U r_i}
    \end{equation*}
    and a unit normal vector $\nu_i$ to $P_i$ such that
    \begin{align*}
        A_1(x_i,10 r_i) &:=\set{y \in 10 W_i | (y - x_i) \cdot \nu_i > \tau \U r_i} \subset \Omega_1\\
        A_2(x_i,10 r_i) &:=\set{y \in 10 W_i | (y - x_i) \cdot \nu_i < -\tau \U r_i} \subset \Omega_2.
    \end{align*}
    Now, we define
    \begin{align*}
        a_1^i &= x_i + 8 r_i \nu_i\\
        a_2^i &= x_i - 8 r_i \nu_i
    \end{align*}
    and $D_h^i = B(a_h^i,r_i)$ for $h = 1,2$. We also denote by $R_h^i$ the averaged rigid displacement of $u$ on $D_h^i$;
    \begin{equation*}
        R_h^i : x \mapsto \fint_{D_h^i} u(y) \dd{y} + \left(\fint_{D_h^i} \frac{\nabla u(y) - \nabla u(y)^T}{2} \dd{y}\right) \left(x - \fint_{D_h^i} y \dd{y}\right).
    \end{equation*}
    Finally, we define on $V_h$, the function
    \begin{equation*}
        v_h(x) := \theta_0(x) u(x) + \sum_{i \in I} \theta_i(x) R_h^i.
    \end{equation*}
    It is clear that $v_h \in LD_{\mathrm{loc}}(V_h)$ and that $v_h = u$ in $V_h \setminus S_h$, where $S_h$ is the relatively closed subset of $V_h$ defined in (\ref{eq_Sh}) by
    \begin{equation*}
        S_h = V_h \cap \bigcup_i \spt(\varphi_i).
    \end{equation*}
    We are going to show that for all $k$ such that $50 W_k \subset B(x_0,\rho)$, we have
    \begin{equation}\label{eq_vh}
        \int_{V_h \cap 10 W_k} \abs{e(v_h)}^2 \dd{x} \leq C \int_{\Omega_h \cap 50 W_k} \abs{e(u)}^2 \dd{x}.
    \end{equation}
    We require $50 W_k \subset B(x_0,\rho)$ so that the right-hand side integral in (\ref{eq_vh}) does not outreach $B(x_0,\rho)$ but also because we will use crucially the item \ref{lem_W_covering2} of Lemma \ref{lem_W}.
    Let us explain how to conclude from (\ref{eq_vh}).
    Thanks to item \ref{lem_W_pointwise_finite} of Lemma \ref{lem_W}, the sum of indicators functions $\sum_i \mathbf{1}_{50 W_i}$ is bounded by a constant $C$ so we can add up the local inequalities in (\ref{eq_vh}) to obtain 
    \begin{equation*}
        \int_{V_h \cap A} \abs{e(v_h)}^2 \dd{x} \leq C \int_{B(x_0,\rho) \cap \Omega_h} \abs{e(u)}^2 \dd{x},
    \end{equation*}
    where $A := \bigcup \set{10 W_i | 50 W_i \subset B(x_0,\rho)}$.
    As
    \begin{equation*}
        S_h = V_h \cap \bigcup_i \spt(\varphi_i) \subset V_h \cap \bigcup_i (10 W_i),
    \end{equation*}
    we have $S_h \setminus \ZZ \subset V_h \cap A$ and thus
    \begin{equation}\label{eq_SZ}
        \int_{S_h \setminus \ZZ} \abs{e(v_h)}^2 \dd{x} \leq C \int_{B(x_0,\rho) \cap \Omega_h} \abs{e(v_h)}^2 \dd{x}
    \end{equation}
    Finally, we use the fact that $v_h = u_h$ in $V_h \setminus S_h$ to conclude
    \begin{equation*}
        \int_{V_h \cap B(x_0,\rho) \setminus \ZZ} \abs{e(v_h)}^2 \dd{x} \leq C \int_{B(x_0,\rho) \cap \Omega_h} \abs{e(v_h)}^2 \dd{x}.
    \end{equation*}

    From now on, we fix an index $k$ such that $50 W_k \subset B(x_0,\rho)$ and we prove (\ref{eq_vh}).
    We define $I_k$ the set of index $i \in I$ such that $10 W_i \cap 10 W_k \ne \emptyset$. In particular, we recall that for $i \in I_k$, we have $r_k / 2 \leq r_i \leq 2 r_k$, $\abs{x_i - x_k} \leq 30 r_k$ and $10 W_i \subset 50 W_k$.
    We want to show first that for all $i \in I_k$,
    \begin{equation}\label{eq_rik}
        \int_{10 W_k} \abs*{R_h^i - R_h^k}^2 \dd{x} \leq C r_k^2 \int_{\Omega_h \cap 50 W_k} \abs{e(u)}^2 \dd{x}.
    \end{equation}
    We cannot apply the Korn-Poincaré inequality in the whole ball $10 W_k$ (because the structure of $K$ is a priori unknown) but we will apply it in a Lipschitz regular open subset set $D_h$ of $B(x_0,\rho) \cap \Omega_h$ of diameter $\leq 100 r_k$ and which contains all the $D_h^i$ for $i \in I$.
    We consider an hyperplane $P$ passing through $x_k$ such that
    \begin{equation*}
        K \cap 50 W_k \subset \set{y \in 50 W_k | \mathrm{dist}(y,P) \leq 50 \tau \U r_k}
    \end{equation*}
    and a unit normal vector $\nu$ to $P$ such that
    \begin{align*}
        A_1(x_k,50r_k) &:= \set{y \in 50 W_k | (y - x_k) \cdot \nu > 50 \tau \U r_k} \subset \Omega_1\\
        A_2(x_k,50r_k) &:= \set{y \in 50 W_k | (y - x_k) \cdot \nu < -50\tau \U r_k} \subset \Omega_2.
    \end{align*}
    It will be useful to observe, as in Remark \ref{rmk_beta_equivalence}, that for all $x \in P \cap 50 W_k$, we have $\mathrm{d}(x,K) \leq 100 \tau \U r_k \leq r_k/2$.
    Now, we introduce
    \begin{subequations}
        \begin{align}
            D_1 &:= \set{y \in 49 W_k | (y - x_k) \cdot \nu > (3/2) r_k}\label{eq_DhD1}\\
            D_2 &:= \set{y \in 49 W_k | (y - x_k) \cdot \nu < - (3/2) r_k}.\label{eq_DhD2}
        \end{align}
    \end{subequations}
    It is clear that $D_h \subset \Omega_h$ because $D_h \subset A_h(x_k,50 r_k)$.
    We are going to justify that for all $i \in I_k$, the ball $D_h^i = B(a_h^i,r_i)$ is contained in $D_h$.
    We fix $i \in I_k$ and we recall that $r_k/2 \leq r_i \leq 2 r_k$ and $\abs{x_i - x_k} \leq 30 r_k$.
    By construction, we have $D_h^i \subset 9 W_i$ so $D_h^i \subset 48 W_k$.
    Then, we show that $(a^i_1 - x_k) \cdot \nu > (7/2) r_k$ and $(a^i_2 - x_k) \cdot \nu < - (7/2) r_k$.
    We only detail the case $h = 1$.
    Since $x_i \in K \cap 10 W_i \subset K \cap 50 W_k$, we have $\abs{(x_i - x_k) \cdot \nu} \leq 50 \tau U r_k$. By Lemma \ref{lem_orientation}, we also know that $\nu_i \cdot \nu \geq 1 - 100 \tau \U$.
    Using the formula $a_1^i = x_i + 8 r_i \nu_i$, we have
    \begin{align*}
        (a_1^i - x_k) \cdot \nu   &= 8(\nu_i \cdot \nu)r_i +  (x_i - x_k) \cdot \nu\\
                                  &\geq 8 (1 - 100 \tau U) r_i - 50 \tau U r_k\\
                                  &\geq 4 (1 - 100 \tau U) r_k - 50 \tau U r_k\\
                                  &\geq (7/2) r_k,
    \end{align*}
    because $\tau U \leq 10^{-3}$. It follows that $D_h^i = B(a_h^i,r_i) \subset B(a_h^i,2r_k) \subset D_h$.
    Note also that the diameter of $D_h$ is bounded by $100 r_k$ and is comparable to radii of all the balls $D_h^i$ for $i \in I_k$.
    We estimate the $L^2$ norm of $R_h^i - R_h^k$ in $D_h^k$ using the Korn-Poincaré inequality,
    \begin{align*}
        \int_{D_h^k} \abs*{R_h^i - R_h^k}^2 \dd{x} &\leq \int_{D_h} \abs*{R_h^i - R_h^k}^2 \dd{x}\\
                                                   &\leq 2\int_{D_h} \abs*{R_h^i - u}^2 \dd{x} + 2\int_{D_h} \abs*{u - R_h^k}^2 \dd{x} \\
                                                   &\leq C r_k^2 \int_{D_h} \abs*{e(u)}^2 \dd{x}.
    \end{align*}
    As $R_h^i - R_h^k$ is an affine function, we have
    \begin{equation*}
        \int_{20 D_h^k} \abs*{R_h^i - R_h^k}^2 \dd{x} \leq C \int_{D_h^k} \abs*{R_h^i - R_h^k}^2 \dd{x}
    \end{equation*}
    and since $10 W_k \subset 20 D_h^k$, the estimate (\ref{eq_rik}) is now proved.



    We come back to the estimation of $\int_{V_h \cap 10 W_k} \abs{e(v_h)}^2 \dd{x}$. Using $e(R_h^i) = 0$, we compute
    \begin{equation*}
        e(v_h) = \nabla \theta_0 \odot u + \theta_0 e(u) + \sum_{i \geq 1} \nabla \theta_i \odot R_h^i,
    \end{equation*}
    where given two vectors $v,w \in \R^N$,
    \begin{equation*}
        v \odot w := \frac{v w^T + w v^T}{2} = \left(\frac{v_i w_j + v_j w_i}{2}\right)_{ij} \in \mathbb M^{N \times N}.
    \end{equation*}
    Since $(\theta_i)_i$ is a partition of unity, we know that $\theta_0 + \sum_{i \in I} \theta_i = 1$ hence we can substract $R_h^k$ and obtain
    \begin{equation*}
        e(v_h) = \nabla \theta_0 \odot (u - R_h^k) + \theta_0 e(u) + \sum_{i \in I} \nabla \theta_i \odot (R_h^i - R_h^k).
    \end{equation*}
    so
    \begin{equation*}
        \abs{e(v_h)} \leq \abs{\nabla \theta_0} \abs*{u - R_h^k} + \abs{\theta_0 e(u)} + \sum_{i \in I} \abs{\nabla \theta_i} \abs*{R_h^i - R_h^k}.
    \end{equation*}
    To deal with the first term, we show that
    \begin{equation}\label{eq_theta0_Dh}
        \set{x \in V_h \cap 10 W_k | \nabla \theta_0 \ne 0} \subset D_h.
    \end{equation}
    As we have seen in Lemma \ref{lem_varphi}, we have $\nabla \theta_0 = \theta_0 = 0$ in $\bigcup 9 W_i$, whence
    \begin{align*}
        \set{x \in V_h \cap 10 W_k | \nabla \theta_0 \ne 0} &\subset V_h \cap 10 W_k \setminus \bigcup_i 9 W_i\\
                                                            &= \Omega_h \cap 10 W_k \setminus \bigcup_i 9 W_i.
    \end{align*}
    Then we recall that $50 W_k \subset B(x_0,\rho)$, the item \ref{lem_W_covering2} of Lemma \ref{lem_W} and the fact that for $x \in P \cap 50 W_k$, we have $\mathrm{dist}(x,K) \leq r_k/2$ to see that
    \begin{align*}
        \Omega_h \cap 10 W_k \setminus \bigcup_i 9 W_i                  &\subset \set{x \in \Omega_h \cap 10 W_k | \mathrm{dist}(x,K) > 2 r_k}\\
                                                                        &\subset \set{x \in \Omega_h \cap 10 W_k | \mathrm{dist}(x,P) > (3/2) r_k},
    \end{align*}
    which is contained in $D_h$.
    Then we use the Korn-Poincaré inequality in $D_h$ to estimate
    \begin{align*}
        \int_{V_h \cap 10 W_k} \abs{\nabla \theta_0}^2 \abs*{u - R_h^k}^2 \dd{x} &\leq C r_k^{-2} \int_{D_h} \abs*{u - R_h^k}^2 \dd{x}\\
                                                                                 &\leq C \int_{D_h} \abs{e(u)}^2 \dd{x}.
    \end{align*}
    For the second term, there is nothing to do thanks to (\ref{eq:phi2}) and the fact that $\abs{\theta_0} \leq 1$ in $V_h$.
    For the last term, remember that there are at most $C$ indices in $I_k$ and that for all $i \in I_k$, we have $\abs{\nabla \theta_i} \leq C r_k^{-1}$ in $V_h \cap 10 W_k$. We apply (\ref{eq_rik}) to deduce
    \begin{align*}
        \sum_{i \in I} \int_{V_h \cap 10 W_k} \abs{\nabla \theta_i}^2 \abs*{R_h^i - R_h^k}^2 \dd{x}   &\leq C r_k^{-2} \sum_{i \in I_k} \abs*{R_h^i - R_h^k}^2 \dd{x}\\
                                                                                                      &\leq C \int_{D_h} \abs{e(u)}^2 \dd{x}.
    \end{align*}
\end{proof}

\begin{remark}\label{rmk_vh}
    A careful look at the proof reveals that we have actually proved that for all $k$ such that $50 W_k \subset B(x_0,\rho)$,
    \begin{equation}\label{eq_evh_Dh}
        \int_{V_h \cap 10 W_k} \abs{e(v_h)}^2 \dd{x} \leq C \int_{D_h} \abs{e(u)}^2 \dd{x},
    \end{equation}
    where $D_h$ is defined in (\ref{eq_DhD1})-(\ref{eq_DhD2}) by
    \begin{align*}
        D_1 &:= \set{y \in 49 W_k | (y - x_k) \cdot \nu > (3/2) r_k} \subset \Omega_1\\
        D_2 &:= \set{y \in 49 W_k | (y - x_k) \cdot \nu < - (3/2) r_k} \subset \Omega_2.
    \end{align*}
    The fact that
    \begin{equation*}
        K \cap 50 W_k \subset \set{y \in 50 W_k | \mathrm{dist}(y,P) \leq r_k/2},
    \end{equation*}
    implies that for $y \in D_h$, we have $\mathrm{dist}(y,K) \geq r_k$.
    We conclude as we did near (\ref{eq_SZ}) that
    \begin{equation*}
        \int_{S_h \setminus \ZZ} \abs{e(v_h)}^2 \dd{x} \leq C \int_{D} \abs{e(u)}^2 \dd{x},
    \end{equation*}
    where
    \begin{equation*}
        D := B(x_0,\rho) \cap \bigcup \set{y \in 50 B(x,r_x) | x \in K \cap \overline{B}(x_0,\rho),\ \mathrm{dist}(y,K) \geq r_x},
    \end{equation*}
    with $B_x = B(x,r_x)$ and $r_x = \delta(x) / U$.
    This estimation is a bit finer than the Lemma's statement and will be useful for Proposition \ref{prop_energy_decay}.
\end{remark}

{\color{black}
    \begin{remark}
        We also control the $L^2$ norm of the extension $v_h$ via the inequality
        \begin{equation}\label{eq_vh_Dh}
            \int_{S_h \setminus \ZZ} \abs{v_h^2} \dd{x} \leq C \int_{D} \abs{u}^2 \dd{x} + C r_0^2 \int_D \abs{e(u)}^2 \dd{x},
        \end{equation}
        where as before
        \begin{equation*}
            D := B(x_0,\rho) \cap \bigcup \set{y \in 50 B(x,r_x) | x \in K \cap \overline{B}(x_0,\rho),\ \mathrm{dist}(y,K) \geq r_x},
        \end{equation*}
        with $B_x = B(x,r_x)$ and $r_x = \delta(x) / U$.
        The control of the $L^2$ norm of $v_h$ will not be needed in this paper but can be used to adapt our technique to a Griffith almost-minimizer with a fidelity term.
        Inequality (\ref{eq_vh_Dh}) follows from the fact that for all $k$ with $50 W_k \subset B(x_0,\rho)$, we have
        \begin{equation*}
            \int_{V_h \cap 10 W_k} \abs{v_h}^2 \dd{x} \leq C \int_{D_h} \abs{u}^2 \dd{x} + C r_k^2 \int_{D_h} \abs{e(u)}^2 \dd{x},
        \end{equation*}
        whose proof is similar to (\ref{eq_evh_Dh}) stated just above for $e(v_h)$ but simpler.
        Indeed, we recall that for $x \in V_h \cap 10 W_k$,
        \begin{equation*}
            v_h(x) = \theta_0(x) u(x) + \sum_{i \in I_k} \theta_i(x) R_h^i(x)
        \end{equation*}
        and that there are at most $C$ indices in $I_k$ so
        \begin{equation*}
            \int_{V_h \cap 10 W_k} \abs{v_h}^2 \dd{x} \leq C \int_{V_h \cap 10 W_k} \abs{\theta_0 u}^2 \dd{x} + C \sum_{i \in I_k} \int_{10 W_k} \abs{R_h^i}^2 \dd{x}.
        \end{equation*}
        Just like (\ref{eq_theta0_Dh}), we have
        \begin{equation*}
            \set{x \in V_h \cap 10 W_k | \theta_0 \ne 0} \subset D_h.
        \end{equation*}
        so we can bound the first term by
        \begin{equation*}
            \int_{V_h \cap 10 W_k} \abs{\theta_0 u}^2 \dd{x} \leq \int_{D_h} \abs{u}^2 \dd{x}.
        \end{equation*}
        For $i \in I_k$, we use the fact that $10 W_k \subset 20 D_h^k$ and that $R_h^i$ is an affine map, to see that
        \begin{equation*}
            \int_{10 W_k} \abs{R_h^i}^2 \dd{x} \leq C \int_{20 D_h^k} \abs{R_h^i}^2 \dd{x} \leq C \int_{D_h^k} \abs{R_h^i}^2 \dd{x}
        \end{equation*}
        and then we can bound
        \begin{align*}
            \int_{D_h^k} \abs{R_h^i}^2 \dd{x} &\leq 2 \int_{D_h^k} \abs{R_h^i - u}^2 \dd{x} + \int_{D_h^k} \abs{u}^2 \dd{x}\\
                                              &\leq C r_k^2 \int_{D_h} \abs{e(u)}^2 \dd{x} + \int_{D_h} \abs{u}^2 \dd{x}.
        \end{align*}
    \end{remark}
}


\section{Control of the flatness by the minimality defect}


{\color{black}
    The goal of this section is to explain how a small minimality defect implies the decay of the flatness.
    We start by introducing the notion of deformation competitor.
}

\begin{definition}
    Let $E$ be a relatively closed subset of $\Omega$.
    A \emph{deformation} of $E$ in an open ball $B \subset \Omega$ is a Lipschitz function $f : E \to \Omega$ such that $f(E \cap B) \subset B$ and $\set{f \ne \mathrm{id}} \subset \subset B$.
    The image of $E$ by a deformation is called a \emph{deformation competitor} of $E$ in $B$.
\end{definition}

{\color{black}
    \begin{remark}\label{rmk_borsuk}
        Topological competitors preserve separation properties.
        Let us give more details.
        Let $E$ be a relatively closed subset of $B(0,1)$ and let $p,q \in B(0,1) \setminus E$ be two points separated by $E$ in $B(0,1)$.
        If $f$ is a deformation of $E$ in $B(0,1)$ such that $p,q \notin [x,f(x)]$ for all $x \in E$, then $p$ and $q$ are still separated by $f(E)$.
        Indeed, the theory of Borsuk maps (\cite[Chap. XVII, 4.3]{dugundji}) shows that if $A$ is a compact set of $\R^N$ which separates two points $p, q \in \R^N \setminus A$ and if $\phi : A \times [0,1] \to \R^N$ is a continuous map such that
        \begin{equation*}
            \phi(\cdot,0) = \mathrm{id} \quad \text{and} \quad p, q \notin \phi(A \times [0,1]),
        \end{equation*}
        then $p$ and $q$ are still separated by $\phi(A,1)$.
        In our case, we can extend $f$ continuously on the compact set $A := E \cup \partial B(0,1)$ by setting $f = \mathrm{id}$ on $\partial B(0,1)$ and we can apply the previous statement with $\phi(x,t) = (1 - t) x + t f(x)$.
    \end{remark}
}

We now introduce the notion of almost-minimal set and then we show that the flatness of minimal sets decay as a power.
We recall that a gauge is a non-decreasing function $h : (0,+\infty) \to [0,+\infty]$ such that $\lim_{t \to 0^+} h(t) = 0$.
Our definition of almost-minimal sets is slightly different from \cite[Definition 1.10, p.5/841]{d5} and \cite[Definition 1.6]{Lab} but it does not matter.
\begin{definition}
    Let $E$ be a relatively closed subset of $\Omega$ with $\HH^{N-1}$-locally finite measure.
    We say that $E$ is an almost-minimal set with gauge $h$ in $\Omega$ if for all $x \in E$, for all $r > 0$ such that $B(x,r) \subset \Omega$ and for all deformation $f$ of $E$ in $B(x,r)$, we have
    \begin{equation*}
        \HH^{N-1}(E \cap B(x,r)) \leq \HH^{N-1}(f(E \cap B(x,r))) + h(r) r^{N-1}.
    \end{equation*}
    Moreover, we say that $E$ is coral if for all $x \in E$ and all $r > 0$, $\HH^{N-1}(E \cap B(x,r)) > 0$.
\end{definition}

As usual, if E is almost-minimal with gauge $h$ in a ball $B(x_0,r_0)$, then the set $r_0^{-1}(E - x_0)$ is almost-minimal in $B(0,1)$ with gauge $\tilde{h}(t) := h(r_0 t)$.
The next lemma should be standard but since we have not found the statement in the litterature, we justify it by using the ideas of Allard's regularity theorem. The reader can skip the proof.

\newcommand{\EG}{\varepsilon_{*}}
\newcommand{\CG}{C_{*}}
\begin{theorem}\label{thm_flatness}
    For all $\gamma \in (0,1)$, there exists $\EG \in (0,1)$ and $\CG \geq 1$ (depending on $N$, $\gamma$) such that the following property holds.
    Let $r_0 > 0$, let $E$ be a coral minimal set in $B(0,r_0)$ such that $0 \in E$.
    If $\beta(0,r_0) \leq \EG$, then for all $0 < r \leq r_0$,
    \begin{equation*}
        \beta(0,r) \leq \CG \left(\frac{r}{r_0}\right)^\gamma \beta(0, r_0).
    \end{equation*}
\end{theorem}
\begin{proof}
    We follow the reference \cite[Theorem 23.1]{simon}.
    To reduce cumbersome notations, we replace $B(0,r_0)$ by $B(0,2r_0)$ in the assumptions, i.e., we assume that $E$ is minimal in $B(0,2r_0)$ and that $\beta(0,2r_0) \leq \EG$.

    Our  minimal set $E$ induces a stationary $n$-rectifiable varifold $V$ with multiplicity $1$ in $B(0,2r_0)$ (see \cite[\S 15--16]{simon}).
    In Simon's notation, the letter $H$ stands for the generalized mean curvature of $V$ which is $0$, the letter $\mu$ stands the weight measure of $V$ which is $\HH^{n-1} \mres E$, the letter $\theta$ stands for the multiplicy of $V$ which is $1$ and the symbol $\spt(V)$ stands for $\spt(\mu)$ which is the set $E$ (because $E$ is coral).


    We let $(e_1,\ldots,e_N)$ be the canonical basis of $\R^N$.
    We identify $\R^{N-1}$ with the hyperplane of $\R^N$ generated by the $e_1,\ldots,e_{N-1}$.
    We decompose each vector $x \in \R^N$ as $x = x' + x_N e_N$, where $x' \in \R^{N-1}$ and $x_N \in \R$.
    Without loss of generality, we assume that the hyperplane $P = \R^{N-1}$ achieves the minimum in the definition of $\beta(0,2r_0)$.

    For $x \in E \cap B(0,r_0)$ and $0 < r \leq r_0$, we introduce
    \begin{equation*}
        d(x,r) := \frac{\HH^{N-1}(E \cap B(x,r))}{\omega_{N-1} r^{N-1}}, 
    \end{equation*}
    where $\omega_{N-1}$ is the measure of a $(N-1)$-dimensional unit disk.
    We claim that there exists a universal constant $C \geq 1$ such that 
    \begin{equation}\label{eq_density_flatness}
        d(0,r_0) \leq 1 + C \beta(0,2r_0).
    \end{equation}
    If $\beta(0,2r_0) \geq 1/4$, this is trivial by Ahlfors-regularity of minimal sets (\cite[Lemma 2.15]{d3}).
    We pass to the case $\beta(0,2r_0) \leq 1/4$.
    We start by introducing the tubular neighborhood
    \begin{equation*}
        A := \set{x \in B(0,2r_0) | \abs{x_N} \leq \beta(0,2r_0) r},
    \end{equation*}
    which, by definition of $\beta(0,2r_0)$, contains $E \cap B(0,2r_0)$.
    We also consider the set
    \begin{equation*}
        F = (\partial B(0,r_0) \cap A) \cup (P \cap B(0,r_0)).
    \end{equation*}
    One can build a Lipschitz function $f: \R^N \to \R^N$ such that $f(B(0,2r_0)) \subset B(0,2 r_0)$, $\set{f \ne \mathrm{id}} \subset \subset B(0,2r_0)$,
    \begin{equation*}
        f(B(0,r_0)) \subset F
    \end{equation*}
    and
    \begin{equation*}
        f = \mathrm{id} \ \text{in} \ \left(A \setminus B(0,r_0)\right).
    \end{equation*}
    It is a deformation of $E$ in $B(0,2r_0)$ so
    \begin{equation*}
        \HH^{N-1}(E \cap B(0,2r_0)) \leq \HH^{N-1}(f(E \cap B(0,2r_0)))
    \end{equation*}
    but since $f = \mathrm{id}$ in $E \setminus B(0,r_0)$, we deduce
    \begin{equation*}
        \HH^{N-1}(E \cap B(0,r_0)) \leq \HH^{N-1}(f(E \cap B(0,r_0)).
    \end{equation*}
    To conclude, we observe that $f(E \cap B(0,r_0)) \subset F$ and
    \begin{equation*}
        \HH^{N-1}(F) \leq \omega_{N-1} r_0^{N-1} + C \beta(0,2r_0) r_0^{N-1}.
    \end{equation*}
    Our claim is proved.

    For $x \in E \cap B(0,r_0)$, $0 < r \leq r_0$ and for any (linear) hyperplane $T$, we introduce the tilt-excess
    \begin{equation*}
        E(x,r,T) = r^{1-N} \int_{E \cap B(x,r)} \norm{p_{T_x} - p_T}^2 \dd{\HH^{N-1}},
    \end{equation*}
    where $T_x$ is the (linear) approximate tangent plane of $E$ at $x$, $p_{T_x}$ and $p_T$ are the orthogonal projection onto $T_x$ and $T$ respectively and $\norm{\cdot}$ is the operator norm.
    If we let $\nu_x$ denote a unit normal vector to $T_x$ and $\nu$ a unit normal vector to $T$, one can compute that
    \begin{equation*}
        \norm{p_{T_x} - p_T} = \sqrt{1 - (\nu_x \cdot \nu)^2} = \abs{\sin(\alpha)},
    \end{equation*}
    where $\alpha$ is the angle between $T_x$ and $T$.

    According to Caccioppoli inequality for minimal sets (\cite[Lemma 22.2]{simon}), there exists a universal constant $C \geq 1$ such that for all (linear) hyperplane $T$ in $\R^N$,
    \begin{equation*}
        E(0,r_0,T) \leq C r_0^{1-N} \int_{E \cap B(0,2r_0)} \left(\frac{\mathrm{dist}(y - x,T)}{r}\right)^2 \dd{\HH^{N-1}}.
    \end{equation*}
    We have in particular for all $T$,
    \begin{equation*}
        E(0,r_0,T) \leq C \left(r^{-1} \sup_{y \in E \cap B(0,2 r_0)} \mathrm{dist}(y,T)\right)^2
    \end{equation*}
    and thus for $T = \R^{N-1}$,
    \begin{equation}\label{eq_excess_flatness}
        E(0,r_0,\R^{N-1}) \leq C \beta(0,2r_0)^2.
    \end{equation}

    From now on, we are going to review the proof of \cite[Theorem 23.1]{simon} to prove the decay of the flatness.
    Here, $c \geq 1$ is a generic constant that depends only on $N$ and $\gamma$, whose value might change without mention.
    To be consistent with Simon's notation, we define $p := n/(1-\gamma)$, where $n := N-1$, so that $\gamma = 1 - n/p$.
    We consider a constant $\varepsilon > 0$ that will be chosen small enough depending on $N$ and $\gamma$.
    In view of (\ref{eq_density_flatness}) and (\ref{eq_excess_flatness}), we can choose $\EG$ small enough (depending on $N$ and $\varepsilon$) so that $f(0,r_0) \leq 3/2$ and $E(0,r_0,\R^n) \leq \varepsilon$.
    Now, provided that $\varepsilon$ is small enough (depending on $N$ and $p$), the assumptions of \cite[Theorem 23.1]{simon} are satisfied.

    The proof of \cite[Theorem 23.1]{simon}, gives at \cite[\S 22 (11)]{simon} that there exists an hyperplane $S_0$ such that for all $0 < r \leq r_0$,
    \begin{equation*}
        E(0,r,S_0) \leq c \left(\frac{r}{r_0}\right)^{2(1 - n/p)} E(0,r_0,\R^n).
    \end{equation*}
    We also get at \cite[\S 23 (14) and (19)]{simon} that there exists a constant $a \in (0,1)$ (depending on $N$ and $p$) and a $C^{1,1-n/p}$ function $f : \R^n \to \R$ such that $f(0) = 0$, $f$ is $(1/2)$-Lipschitz,
    \begin{equation*}
        E \cap B(0,a r_0) = \textrm{graph}(f) \cap B(0,a r_0)
    \end{equation*}
    and for all $v \in \R^n \cap B(0, a r_0)$,
    \begin{equation}\label{eq_nabla_f}
        \abs{\nabla f(v) - \nabla f(0)} \leq c \left(\frac{\abs{v}}{r_0}\right)^{1 - n/p} E(0,r_0,\R^n)^\frac{1}{2}.
    \end{equation}
    Since $\lim_{r \to 0} E(0,r,S_0) = 0$ and $E$ is a $C^1$ surface in a neighborhood of $0$, it is easy to see that $T_0 = S_0$, where $T_0$ is the tangent plane to $E$ at $0$.
    Thus, the vector
    \begin{equation*}
        \nu_0 := \frac{-\nabla f(0) + e_N}{\sqrt{1 + \abs{\nabla f(0)}^2}}
    \end{equation*}
    is a unit normal to $S_0$.
    For all $v \in \R^n \cap B(0,a r_0)$, we estimate
    \begin{align*}
        \mathrm{dist}(v + f(v) e_N, S_0) &= \abs{(v + f(v) e_N) \cdot \nu_0}\\
                                         &= \sqrt{1 + \abs{\nabla f(0)}^2}^{-1} \abs{f(v) - \nabla f(0) \cdot v}\\
                                         &\leq \abs{f(v) - \nabla f(0) \cdot v}
    \end{align*}
    and we deduce by the mean value inequality, (\ref{eq_nabla_f}) and (\ref{eq_excess_flatness}) that
    \begin{align*}
        \mathrm{dist}(v + f(v) e_N, S_0) &\leq c \abs{v} \left(\frac{\abs{v}}{r_0}\right)^{1 - n/p} E(0,r_0,\R^n)^\frac{1}{2}\\
                                         &\leq c \abs{v} \left(\frac{\abs{v}}{r_0}\right)^{1 - n/p} \beta(0,2 r_0).
    \end{align*}
    It follows that for all $x \in E \cap B(0,a r_0)$,
    \begin{equation*}
        \mathrm{dist}(x,S_0) \leq c \abs{x} \left(\frac{\abs{x}}{r_0}\right)^{1 - n/p} \beta(0,2 r_0)
    \end{equation*}
    and thus for all $0 < r \leq a r_0$,
    \begin{equation*}
        r^{-1} \sup_{x \in E \cap B(0,r)} \mathrm{dist}(x,S_0) \leq c \left(\frac{r}{r_0}\right)^{1 - n/p} \beta(0,2 r_0).
    \end{equation*}
    This is enough to control the bilateral flatness $\beta(0,r)$ by Remark \ref{rmk_beta_equivalence} (the set $E$ is a $(1/2)$-Lipschitz graph passing through $0$ so it separates $B(0,r)$).
    The case $r \in [a r_0,2 r_0]$ is trivial because
    \begin{equation*}
        \beta(0,r) \leq \left(\frac{2}{a}\right) \beta(0,2r_0)
    \end{equation*}
    and $(r/r_0)^{1 - n/p}$ is bounded from below by $a^{1 - n/p}$ which depends only on $N$ and $p$.
\end{proof}

Here is a new proof of \cite[Lemma 31]{l2} which avoids the ``uniform concentration property'' (not yet available in Griffith setting) by exploiting the limiting properties of minimizing sequences (see \cite{I1}, \cite{I2}, \cite{I3}, \cite{I4} or \cite{Lab}).

\begin{lemma}\label{lem_flatness}
    Let $\EG \in (0,1)$ and $\CG \geq 1$ be the universal constants introduced in Theorem \ref{thm_flatness} for $\gamma = 1/2$.
    Let $E$ be a relatively closed subset of $B(0,1)$ such that $0 \in E$. We assume that there exists $C \geq 1$ such that for all $x \in E$, for all $r > 0$ with $B(x,r) \subset B(0,1)$, we have
    \begin{equation}\label{eq_af_flatness}
        \HH^{N-1}(E \cap B(x,r)) \geq C^{-1} r^{N-1}.
    \end{equation}
    For all $0 < \varepsilon \leq \EG$ and for all $0 < a \leq 1/2$, there exists a constant $\lambda > 0$ (depending on $N$, $C$, $\varepsilon$, $a$) such that if $\beta_E(0,1) \leq \varepsilon$ and if for all deformation competitor $F$ of $E$ in $B(0,1)$, we have
    \begin{equation*}
        \HH^{N-1}(E) \leq \HH^{N-1}(F) + \lambda,
    \end{equation*}
    then
    \begin{equation*}
        \beta_E(0,a) \leq 4 \CG \sqrt{a} \varepsilon.
    \end{equation*}
\end{lemma}

\begin{proof}
    We proceed by contradiction.
    We assume that for all $i \in \N$, there exists a relatively closed set $E_i$ of $B(0,1)$ which contains $0$, which satisfies (\ref{eq_af_flatness}) with a uniform constant $C$ and such that
    \begin{align*}
        \beta_{E_i}(0,1)            & \leq \varepsilon \\
        \beta_{E_i}(0,a)            & \geq  4 \CG \sqrt{a} \varepsilon
    \end{align*}
    but for all deformation $f$ in $B(0,1)$,
    \begin{equation}\label{eq:min_i}
        \HH^{N-1}(E_i) \leq \HH^{N-1}(f(E_i)) + 2^{-i}.
    \end{equation}
    We start by justifying that the sequence $(E_i)$ has a locally uniformly bounded measure in $B(0,1)$.
    Let us fix an index $i$, consider a point $x_0 \in B(0,1) \setminus E_i$ (at least one exists because $\beta_{E_i}(0,1) < 1$) and an open ball $B_0$ such that $x_0 \in B_0 \subset \subset B(0,1)$.
    We let $f: E_i \cap B_0 \to \R^N$ be the radial projection onto $\partial B_0$ centered on $x_0$ and we extend it by the identity map on $E_i \setminus B_0$. This defines a deformation of $E_i$ in $B(0,1)$ that we can use in (\ref{eq:min_i}) to bound
    \begin{equation*}
        \HH^{N-1}(E_i \cap B_0) \leq C \mathrm{diam}(B_0)^{N-1} + 1,
    \end{equation*}
    for some universal constant $C \geq 1$.

    Since the sequence of measures $(\HH^{N-1} \mres E_i)_i$ is locally uniformly bounded in $B(0,1)$, a subsequence (not relabeled) converges to a Radon measure $\mu$ in $B(0,1)$.
    According to \cite[Corollary 4.1]{Lab}, the minimality property (\ref{eq:min_i}) implies that $\mu = \HH^{N-1} \mres E$, where $E$ is the support of $\mu$ in $B(0,1)$. Moreover, $E$ is minimal in the sense that for all deformation $f$ of $E$ in $B(0,1)$, one has
    \begin{equation*}
        \HH^{N-1}(E) \leq \HH^{N-1}(f(E)).
    \end{equation*}
    We also mention \cite[Theorem 1.7]{I5} as an alternative to \cite{Lab}.
    It is itself the conclusion of a series of works on limits of minimizing sequences \cite{I1}, \cite{I2}, \cite{I4}.
    We observe that since $E$ is the support of $\mu = \HH^{N-1} \mres E$, it is in particular a coral set.

    Next, we justify that $(E_i)_i$ converges to $E$ in local Hausdorff distance in $B(0,1)$, that is,
    \begin{equation*}
        E = \set{x \in \Omega | \liminf_i \mathrm{dist}(x,E_i) = 0} = \set{x \in \Omega | \lim_i \mathrm{dist}(x,E_i) = 0}.
    \end{equation*}
    For all $x \in E$, and for all $r > 0$ such that $B(x,r) \subset B(0,1)$, we have 
    \begin{equation*}
        0 < \mu(B(x,r)) \leq \liminf_i \HH^d(E_i \cap B(x,r))
    \end{equation*}
    so for $i$ big enough, $\HH^d(E_i \cap B(x,r)) > 0$. This proves that $\limsup_i \mathrm{dist}(x,E_i) \leq r$ but since $r$ is arbitrary small, we actually have $\lim_i \mathrm{dist}(x,E_i) = 0$.
    For $x \in B(0,1)$ such that $\liminf_i \mathrm{dist}(x,E_i) = 0$, there exists a subsequence $(E_j)_j$ and a sequence of points $(x_j)_j$ such that $x_j \in E_j$ and $x_j \to x$.
    For all $r > 0$ small enough we have $\overline{B}(x,r) \subset B(0,1)$ and for $j$ big enough, we have $B(x_j,r/2) \subset B(x,r)$ so using (\ref{eq_af_flatness}), we get
    \begin{align*}
        \mu(\overline{B}(x,r)) & \geq \limsup_j \HH^d(E_j \cap \overline{B}(x,r))     \\
                               & \geq \limsup_j \HH^d(E_j \cap \overline{B}(x_j,r/2)) \\
                               & \geq C^{-1} (r/2)^{N-1}.
    \end{align*}
    This proves that $x \in E$. We have shown that $(E_i)_i$ converges to $E$ in local Hausdorff distance in $B(0,1)$. According to Remark \ref{rmk_beta}, we have
    \begin{align*}
        \beta_E\left(0,1\right) &\leq \liminf_i \beta_{E_i}(0,1) \leq \varepsilon \\
        \beta_E\left(0,2a\right) &\geq \tfrac{1}{2} \limsup_i \beta_{E_i}(0,a) \geq 2 \CG \sqrt{a} \varepsilon
    \end{align*}
    The set $E$ is a coral minimal in $B(0,1)$ and satisfies $\beta_E(0,1) \leq \varepsilon_*$ so we can apply Theorem \ref{thm_flatness} with $\gamma = 1/2$, $r_0 = 1$, $r = 2a$ to arrive at
    \begin{equation*}
        \beta_E(0,2 a) \leq \CG \sqrt{2 a} \beta_E\left(0,1\right) < 2 \CG \sqrt{a} \varepsilon.
    \end{equation*}
    Contradiction.
\end{proof}


\section{Stopping time and regularity estimates}\label{section_badmass}


\label{section5}

\subsection{Definition of the bad mass}
The bad mass is a quantity that measures how much $K$ differs from being $\tau$-Reifenberg-flat\footnote{We say that a relatively closed subset $K \subset B(x_0,r_0)$ is $\tau$-Reifenberg flat in $B(x_0,r_0)$ for some $\tau > 0$ provided that for all $x \in K \cap B(x_0,9r_0/10)$ and for all $0 < r \leq r_0/10$, we have $\beta_K(x,r) \leq \tau$.}.
According to the Reifenberg parametrization theorem, a $10^{-3}$-Reifenberg-flat set is a Hölder surface.

Let $\EG \in (0,1)$ and $\CG \geq 1$ be the universal constants of Theorem \ref{thm_flatness} for $\gamma = 1/2$.
We fix for the rest of the paper the following universal constants:
\begin{equation}\label{eq_constantes}
    \tau := \min(\EG, 10^{-9}), \quad \quad A_0 := (4\CG)^2 \geq 1 \quad \text{ and } \quad A = U A_0.
\end{equation}
Here, $\U = 10^5$ is the constant used in Section \ref{section_extension}.

Let $(u,K)$ be a Griffith almost-minimizer with gauge $h$ in $\Omega$.
Let $x_0 \in K$, $r_0 > 0$ be such that $B(x_0,r_0) \subset \Omega$, $K$ separates $B(x_0,r_0)$, $h(r_0) \leq \eaf$ and
\begin{equation}\label{eq_varepsilon00}
    \beta(x_0,r_0) \leq \tau/ (400 A).
\end{equation}
We recall that $\eaf$ is the constant used in (\ref{eq_AF}) to ensure the Ahlfors-regularity of $K$ in $B(x_0,r_0)$.
We don't actually need the almost-minimality of $(u,K)$ to define the bad mass but we will need fact that $K$ is Ahlfors-regular.

For $x \in K \cap \overline{B}(x_0,9r_0/10)$, we define the \emph{stopping time function}
\begin{equation*}
    d(x) := \inf \Set{ r>0 | \beta(x,t) \leq \tau \text{ for all } t \in [r, r_0 / 10]}
\end{equation*}
and then we define the \emph{bad mass} of $K$ in $B(x_0,r_0)$ by
\begin{equation*}
    m_K(x_0,r_0):= \frac{1}{r_0^{N-1}} \mathcal{H}^{N-1}(K \cap R(x_0,r_0)),
\end{equation*}
where
\begin{equation*}
    R(x_0,r_0):= \bigcup \Set{B(x,A d(x)) | x \in K \cap \overline{B}\left(x_0,9r_0/10\right) \text{ such that } d(x)>0}.
\end{equation*}
When there is no ambiguity, we write $m(x_0,r_0)$ instead of $m_K(x_0,r_0)$.
Note that we need the above assumptions to consider the bad mass to be well-defined in $B(x_0,r_0)$.

For all $x \in \overline{B}(x_0,9r_0/10)$ and $0 < t \leq r_0/10$, we have
\begin{equation*}
    \beta(x,t)\leq (2 r_0 / t) \beta(x_0,r_0)
\end{equation*}
so
\begin{equation}\label{eq_dx_estimate}
    d(x) \leq 2 \tau^{-1} \beta(x_0,r_0) r_0.
\end{equation}
The condition (\ref{eq_varepsilon00}) gives in particular $d(x) \leq r_0/(200 A)$ so $10 \overline{B}(x, Ad(x)) \subset B(x_0, 99r_0/100)$ and
\begin{equation*}
    m(x_0,r_0) \leq \frac{1}{r_0^{N-1}} \HH^{N-1}(K \cap B(x_0,r_0)).
\end{equation*}
Since $K$ is Ahlfors-regular, the bad mass is bounded from above by the Ahlfors-regularity constant of $K$.

It is immediate from the definition of $d$ that for all $t \in ]d(x),r_0/10]$, we have $\beta(x,t) \leq \tau$.
But if $d(x) > 0$, one can also use the usual scaling property (\ref{eq_beta_scaling}) to see that
\begin{equation*}
    \beta(x,d(x)) = \tau.
\end{equation*}

\begin{remark}\label{rmk_badmass}
    Just like $\omega(x,r)$ and $\beta(x,r)$, the bad mass $m(x,r)$ has scaling properties.
    Let $x \in K \cap B(x_0,r_0)$ and $r > 0$ be such that $B(x,9r/10) \subset B(x_0,9r_0/10)$ and such that $\beta(x,r) \leq \tau/(400 A)$.
    This holds true for example if $\beta(x_0,r_0)$ is small enough compared to $r/r_0$.
    Then one can see that $R(x,r) \subset R(x_0,r_0)$ and
    \begin{equation*}
        m(x,r) \leq \left(\frac{r_0}{r}\right)^{N-1} m(x_0,r_0).
    \end{equation*}
\end{remark}

\subsection{Preparation of the Extension Lemma}

{\color{black}
    In the next sections, we will obtain three estimates controlling respectively the bad mass, the minimality defect and the energy decay (Propositions \ref{prop_badmass_decay}, \ref{prop_flatness_decay} and \ref{prop_energy_decay}).
    We will need that $\beta(x_0,r_0) \leq \varepsilon_0$, where $\varepsilon_0 > 0$ is a small constant such that
    \begin{equation}\label{eq_varepsilon0}
        \text{$\varepsilon_0$ is small enough (depending on $N$)}.
    \end{equation}
    This means at least that we want (\ref{eq_varepsilon00}) to hold so that the bad mass in $B(x_0,r_0)$ is well-defined.
    But we will also invoke (\ref{eq_varepsilon0}) whenever we need $\varepsilon_0$ to be less than a universal constant that we don't try to make explicit.

    The competitors in Propositions \ref{prop_badmass_decay}, \ref{prop_flatness_decay}, \ref{prop_energy_decay} are built using the extension Lemma \ref{lem_extension} with a suitable geometric function.
    The stopping time function $d$ introduced above looks very much like a geometric function except that it is not Lipschitz. We will see in Lemma \ref{lem_delta} how a Vitali covering of bad balls $(B(x_i, d(x_i))_i$ induces a natural geometric function $\delta$.
    When we will apply the extension Lemma with this geometric function $\delta$ and a given radius $\rho \in [r_0/2,3r_0/4]$, we will check that the ``wall set" $\ZZ$ is contained in $\bigcup_{i \in I(\rho)} 10 B_i$, where
    \begin{equation*}
        I(\rho) := \set{i \in I | 10 B_i \cap \partial B(x_0,\rho) \ne \emptyset}.
    \end{equation*}
    The set $\ZZ$ is a domain around $K \cap \partial B(x_0,\rho)$, where the extension $v$ has an unknown elastic energy and does not connect well with the original function $u$.
    Thus, we will force $v = 0$ in $\bigcup_{i \in I(\rho)} 10 B_i$ and this will add a crack of measure bounded by $C \sum_{i \in I(\rho)} r_i^{N-1}$.

    Next, we will see in Lemma \ref{lem_goodradius} that, in average, it is possible to choose $\rho \in [r_0/2, 3r_0/4]$ such that this contribution is bounded by $\beta(x_0,r_0) m(x_0,r_0)$.
    This term will thus appears in the estimates and should be thought of as an error term accounting for the wall set introduced in the extension Lemma.
    If we knew that $K$ was Reifenberg-flat set, there would be no wall set and no bad mass, making the proof of our main theorem much more straightforward.

    Finally, let us note that the gauge will not play any meaningful role in the proof of Propositions \ref{prop_badmass_decay}, \ref{prop_flatness_decay}, \ref{prop_energy_decay} and even Lemma \ref{lem_decay} in the last section of our article.
    We will just make sure to work in balls where $h(r) \leq \eaf$ in order to have the Ahlfors-regularity of $K$.
    The constants in these propositions will depend on $N$ and the Ahlfors-regularity constant of $K$ so they will be universals.
}

We now proceed to extract a Vitali covering of bad balls $(B_i)_i$ and deduce a geometric function.
From the family of open balls
\begin{equation*}
    \set{B(x,Ad(x)) | x \in K \cap \overline{B}\left(x_0,9r_0/10\right) \text{ such that } d(x) > 0},
\end{equation*}
the Vitali covering Lemma yields a countable disjoint subfamily $(B_i)_{i\in I}$ such that
\begin{equation*}
    \bigcup_i B_i \subset R(x_0,r_0) \subset \bigcup_i 4 B_i.
\end{equation*}
We let $x_i \in \overline{B}(x_0,9r_0/10)$ and $r_i = A d(x_i)$ denote the center and the radius of $B_i$, that is, $B_i = B(x_i, r_i)$.
By construction, we always have $r_i > 0$.
According to (\ref{eq_dx_estimate}), we have
\begin{equation*}
    r_i \leq 2 A \tau^{-1} \beta(x_0,r_0) r_0 \leq C \varepsilon_0 r_0,
\end{equation*}
where $C \geq 1$ is a universal constant.
The condition (\ref{eq_varepsilon0}) gives in particular
\begin{equation*}
    r_i \leq r_0/200 \quad \text{and} \quad 10 \overline{B}_i \subset B\left(x_0,99r_0/100\right).
\end{equation*}
We observe that by Ahlfors-regulary of $K$,
\begin{equation*}
    m(x_0,r_0) \simeq \frac{1}{r_0^{N-1}} \sum_{i \in I} r_i^{N-1}.
\end{equation*}
We also have $\beta(x_0,r_0/4) \leq \tau/(400)$ by (\ref{eq_varepsilon0}) so one can see as in Remark \ref{rmk_badmass} that $R(x_0,r_0/4) \subset R(x_0,r_0)$ and bound
\begin{equation}\label{eq_m4}
    m(x_0,r_0/4) \lesssim \frac{1}{r_0^{N-1}} \sum_{i \in I_0} r_i^{N-1},
\end{equation}
where $I_0 = \set{i \in I | 4 B_i \cap B(x_0,r_0/4) \ne \emptyset}$.
Finally, we consider the function $\delta : K \cap \overline{B}(x_0,3r_0/4) \to [0,+\infty)$ defined by
\begin{equation}\label{eq_delta}
    \delta(x) := \inf_{i \in I} \Set{\abs{x - x_i} + r_i} \wedge \mathrm{dist}\left(x, \R^n \setminus \bigcup_i 9 B_i\right),
\end{equation}
where $\wedge$ denotes the minimum.

\begin{lemma}\label{lem_delta}
    Let $(u,K)$ be a Griffith almost-minimizer with gauge $h$ in $\Omega$.
    Let $x_0 \in K$, $r_0 > 0$, $\varepsilon_0 \geq 0$ be such that (\ref{eq_varepsilon0}) holds, $K$ satisfies Hypothesis-$H(\varepsilon_0,x_0,r_0)$ and $h(r_0) \leq \eaf$.
    Then the function $\delta$ defined in (\ref{eq_delta}) is $1$-Lipschitz and
    \begin{enumerate}
        \item for all $k$, $\delta(x_k) = r_k$;
        \item for all $x \in K \cap \overline{B}(x_0,3r_0/4)$, $\delta(x) \leq \max_k (10 r_k) \leq r_0/4$;
        \item for all $x \in K \cap \overline{B}(x_0,3r_0/4)$ and for all $r \in ]\delta(x),r_0/4]$, we have $\beta(x,r) \leq 4 \tau$.
    \end{enumerate}
\end{lemma}
Thus, $\delta$ is a geometric function with parameters $(3r_0/4, 4\tau)$.
The condition (\ref{eq_U}) in Section \ref{section_extension} is satisfied with $4 \tau$ instead of $\tau$ and this will allow us to apply Lemma \ref{lem_extension}.
\begin{proof}
    It is readily seen that $\delta$ is $1$-Lipschitz and that $\delta(x_k) \leq r_k$.
    One can see that $\delta(x_k) = r_k$ using the fact that the balls $(B_i)_i$ are disjoint, i.e., for all $i \ne k$,
    \begin{equation*}
        \abs{x_i - x_k} \geq r_i + r_k,
    \end{equation*}
    and that
    \begin{equation*}
        \mathrm{dist}(x_k, \R^\N \setminus \bigcup_i 9 B_i) \geq \mathrm{dist}(x_k, \R^N \setminus 9 B_k) \geq 9 r_k.
    \end{equation*}
    Next, we justify that for all $x \in K \cap \overline{B}(x_0,3r_0/4)$, we have $\delta(x) \leq r_0/4$. If $x \notin \bigcup_i 9 B_i$, then $\delta(x) = 0$. Otherwise, there exists $k$ such that $x \in 9 B_k$ and thus $\delta(x_k) \leq \abs{x - x_k} + r_k \leq 10 r_k \leq r_0/4$.
    Finally, we prove that for all $x \in K \cap \overline{B}(x_0,3r_0/4)$ and for all $r \in ]\delta(x),r_0/4]$, we have $\beta(x,r) \leq 4 \tau$.
    If $d(x) = 0$, this is trivial by definition of $d(x)$.
    Otherwise, there exists $k$ such that $x \in 4 B_k$ because $(4 B_i)_i$ covers $R(x_0,r_0)$.
    In this case,
    \begin{equation*}
        \inf_{i \in I} \abs{x - x_i} + r_i \leq \abs{x - x_k} + r_k \leq 5 r_k
    \end{equation*}
    whereas
    \begin{equation*}
        \mathrm{dist}(x, \R^N \setminus 9 B_k) \geq 5 r_k
    \end{equation*}
    so $\delta(x) = \inf_{i \in I} \abs{x - x_i} + r_i$.
    If $r \in ]\delta(x), r_0/20]$, there exists $i$ such that $\abs{x - x_i} + r_i \leq r$ so in particular, $B(x,r) \subset B(x_i,2r)$ and $r \geq r_i > d(x_i)$.
    As $2 r \in (d(x_i),r_0/10)$, we have $\beta(x_i,2r) \leq \tau$ and we deduce that
    \begin{equation*}
        \beta(x,r) \leq 4 \beta(x_i,2r) \leq 4 \tau.
    \end{equation*}
    If $r \in [r_0/20, r_0/10]$, we use $\beta(x_0,r_0) \leq \varepsilon_0$ and $\varepsilon_0 \leq \tau / 40$ to get directly
    \begin{equation*}
        \beta(x,r) \leq 40 \beta(x_0,r_0) \leq \tau.
    \end{equation*}
\end{proof} 

\begin{lemma}[Selection of good radii]\label{lem_goodradius}
    Let $(u,K)$ be a Griffith almost-minimizer with gauge $h$ in $\Omega$.
    Let $x_0 \in K$, $r_0 > 0$, $\varepsilon_0 \geq 0$ be such that (\ref{eq_varepsilon0}) holds, $K$ satisfies Hypothesis-$H(\varepsilon_0,x_0,r_0)$ and $h(r_0) \leq \eaf$.
    Then there exists $\rho \in [r_0/2, 3r_0/4]$ such that
    \begin{equation*}
        \sum_{i \in I(\rho)} r_i^{N-1}  \leq C_0 \beta(x_0,r_0) m(x_0,r_0) r_0^{N-1},
    \end{equation*}
    where
    \begin{equation*}
        I(\rho) := \set{i \in I | 10 B_i \cap \partial B(x_0,\rho) \ne \emptyset}.
    \end{equation*}
    and $C_0 \geq 1$ is a universal constant.
\end{lemma}

\begin{proof}
    Throughout the proof, the letter $C$ is a universal constant $\geq 1$ whose value might change from one line to another.
    We select a radius $\rho \in [r_0/2, 3r_0/4]$ such that the mass of the balls $(10 B_i)_i$ that are meeting $\partial B(0,\rho)$ is less than average.
    More precisely, we choose $\rho \in [r_0/2, 3r_0/4]$ such that
    \begin{equation*}
        \sum_{i \in I(\rho)} r_i^{N-1}  \leq \frac{4}{r_0} \int_{r_0/2}^{3r_0/4} \left(\sum_{i \in I(t)} r_i^{N-1}\right) \dd{t},
    \end{equation*}
    where
    \begin{equation*}
        I(t) := \set{i \in I | 10 B_i \cap \partial B(0,t) \ne \emptyset}.
    \end{equation*}
    Then we use Fubini's theorem and $r_i \leq C \varepsilon_0 r_0$ to compute
    \begin{align*}
        \int_{r_0/2}^{3r_0/4}  \sum_{i \in I(t)} r_i^{N-1} \dd{t}   &\leq \sum_{i\in I} \int_{\set{t | i \in   I(t)}} r_i^{N-1} \dd{t}\\
                                                                    &\leq C \sum_{i\in I} r_i^N \\
                                                                    &\leq C \varepsilon_0 m(x_0,r_0) r_0^N.
    \end{align*}
    We can apply the whole proof with $\varepsilon_0 = \beta(x_0,r_0)$.
\end{proof}

\subsection{Decay of the bad mass}

The crucial fact to estimate the quantity $m(x_0,r_0)$ is that in each bad ball, we can find a deformation competitor of $K$ which has less measure in a quantified way.
We will use this observation in Proposition \ref{prop_badmass_decay} to build a deformation $L$ of $K$ in $B(x_0,r_0)$ such that $L = K$ in $\Omega \setminus B(x_0,r_0)$ and 
\begin{equation*}
    \HH^{N-1}(L \cap B(x_0,r_0)) \leq \HH^{N-1}(K \cap B(x_0,r_0)) - C^{-1} m(x_0,r_0/4) r_0^{N-1}
\end{equation*}
We will then build an appropriate function $v \in W^{1,2}_{\mathrm{loc}}(\Omega \setminus L)$ such that $(v,L)$ is a competitor of $(u,K)$ in $B(x_0,r_0)$ and compare their Griffith energies to bound $m(x_0,r_0/4)$ from above.

\begin{lemma}[Win of bad mass]\label{lem_competitor}
    Let $(u,K)$ be a Griffith almost-minimizer with gauge $h$ in $\Omega$.
    Let $x_0 \in K$, $r_0 > 0$, $\varepsilon_0 \geq 0$ be such that (\ref{eq_varepsilon0}) holds, $K$ satisfies Hypothesis-$H(\varepsilon_0,x_0,r_0)$ and $h(r_0) \leq \eaf$.
    For all $i$, there exists a deformation competitor $L$ of $K$ in $D_i := B(x_i, A_0 d(x_i))$ such that
    \begin{equation*}
        \HH^{N-1}(K \cap D_i) - \HH^{N-1}(L \cap D_i) \geq C^{-1} r_i^{N-1},
    \end{equation*}
    for some universal constant $C \geq 1$.
\end{lemma}

\begin{proof}
    Let $i$ be a fixed index.
    By definition of $d(x_i)$ and since $A_0 d(x_i) \in ]d(x_i),r_0/10]$, we have
    \begin{align*}
        \beta(x_i, A_0 d(x_i)) &\leq \tau\\
        \beta(x_i,d(x_i)))     &\geq \tau.
    \end{align*}
    We recall that we have chosen $\tau$ and $A_0$ in (\ref{eq_constantes}) in such a way that $\tau \leq \EG$ and $A_0 \geq (4 C_*)^2$.
    We apply Lemma \ref{lem_flatness} in $D_i$ with $\varepsilon = \tau$ and $a = A_0^{-1}$ and since
    \begin{align*}
        \beta(x_i, A_0 d(x_i)) &\leq \varepsilon\\
        \beta(x_i,d(x_i))) &\geq 4 \CG \sqrt{a} \varepsilon,
    \end{align*}
    there exists a universal constant $\lambda > 0$ and a deformation competitor $L$ or $K$ in $B(x,r)$ such that
    \begin{equation*}
        \HH^{N-1}(K \cap D_i) - \HH^{N-1}(L \cap D_i) \geq \lambda (A_0 d(x_i))^{N-1}.
    \end{equation*}
\end{proof}

\begin{proposition}[Control of the bad mass]\label{prop_badmass_decay}
    Let $(u,K)$ be a Griffith almost-minimizer with gauge $h$ in $\Omega$.
    Let $x_0 \in K$, $r_0 > 0$, $\varepsilon_0 \geq 0$ be such that (\ref{eq_varepsilon0}) holds, $K$ satisfies Hypothesis-$H(\varepsilon_0,x_0,r_0)$ and $h(r_0) \leq \eaf$.
    Then we have
    \begin{equation*}
        m(x_0,r_0/4)\leq C \left[\omega(x_0,r_0) + \beta(x_0,r_0) m(x_0,r_0) + h(r_0)\right],
    \end{equation*}
    for some universal constant $C \geq 1$.
\end{proposition}

\begin{proof}
    Throughout the proof, the letter $C$ is a universal constant $\geq 1$ whose value might change from one line to another.
    We recall that
    \begin{equation*}
        m(x_0,r_0/4) \leq \frac{C}{r_0^{N-1}} \sum_{i \in I_0} r_i^{N-1},
    \end{equation*}
    where $I_0 = \set{i \in I | 4 B_i \cap B(x_0,r_0/4) \ne \emptyset}$, see (\ref{eq_m4}).
    Let us note that for $i \in I_0$, we have $\overline{B}_i \subset B(x_0,r_0/2)$ (because $r_i < r_0/200$).
    We find more convenient to work with a finite number of balls so we consider a finite subset $J \subset I_0$ such that
    \begin{equation}\label{eq_m4bis}
        m(x_0,r_0/4) \leq \frac{C}{r_0^{N-1}} \sum_{i \in J} r_i^{N-1}.
    \end{equation}
    For all $i \in J$, Lemma \ref{lem_competitor} yields a competitor $L_i$ of $K$ in the ball $D_i = B(x_i,A_0 d(x_i)) \subset B_i$ such that
    \begin{equation}\label{eq_Li}
        \HH^{N-1}(L \cap D_i) \leq \HH^{N-1}(K \cap D_i) - C^{-1} r_i^{N-1}.
    \end{equation}
    We introduce
    \begin{equation*}
        L := \left(K \setminus \bigcup_{i \in J} D_i\right) \cup \bigcup_{i \in J} \left(L_i \cap D_i\right).
    \end{equation*}
    For each $i \in J$, there exists a deformation $f_i$ of $K$ in $D_i$ such that $L_i = f_i(K)$.
    The balls $(\overline{D}_i)_{i \in J}$ are mutually disjoint and contained in $B(x_0,r_0/2)$ so we can glue the deformation $f_i$ to obtain a deformation $f$ of $K$ in $B(x_0,r_0/2)$ such that $L = f(K)$.
    Let us note that $L$ coincides with $K$ outside $B(x_0,r_0/2)$.
    We recall that there exists an hyperplane $P_0$ (of unit vector $\nu_0$) passing through $x_0$ such that $K \cap B(x_0,r_0) \subset \set{\mathrm{dist}(\cdot,P_0) \leq \varepsilon_0}$ and such that the sets
    \begin{equation*}
        \begin{gathered}
            \set{x \in B(x_0,r_0) | (x - x_0) \cdot \nu_0 > \varepsilon_0 r_0}\\
            \set{x \in B(x_0,r_0) | (x - x_0) \cdot \nu_0 < -\varepsilon_0 r_0}
        \end{gathered}
    \end{equation*}
    belong to different connected component of $B(x_0,r_0) \setminus K$, denoted by $\Omega_1$ and $\Omega_2$.
    According to the theory of Borsuk maps (see Remark \ref{rmk_borsuk}), the sets
    \begin{equation*}
        \begin{gathered}
            \set{x \in B(x_0,r_0) | (x - x_0) \cdot \nu_0 > r_0/2}\\
            \set{x \in B(x_0,r_0) | (x - x_0) \cdot \nu_0 < -r_0/2}
        \end{gathered}
    \end{equation*}
    also lie in different connected component of $B(x_0,r_0) \setminus L$, denoted by $X_1$ and $X_2$ respectively, and we have for $h=1,2$,
    \begin{equation}\label{eq_X}
        X_h \subset \Omega_h \cup \bigcup_{i \in J} D_i.
    \end{equation}
    Let us note that $X_h \setminus B(x_0,r_0/2) = \Omega_h \setminus B(x_0,r_0/2)$.
    To conclude this paragraph, we use (\ref{eq_Li}) and (\ref{eq_m4bis}) to estimate
    \begin{align}
        \HH^{N-1}(L \cap B(x_0,r_0))    &\leq \HH^{N-1}(K \cap B(x_0,r_0)) - C^{-1} \sum_{i \in J} r_i^{N-1}\notag\\
                                        &\leq \HH^{N-1}(K \cap B(x_0,r_0)) - C^{-1} m(x_0,r_0/4) r_0^{N-1}.\label{eq_L}
    \end{align}

    We would like to build a function $v \in LD(B(x_0,r_0) \setminus L)$ such that $(v,L)$ is a competitor of $(u,K)$ in $B(x_0,9r_0/10)$ and the energy of $v$ is under control, i.e.,
    \begin{equation*}
        \int_{B(x_0,r_0) \setminus L} \abs{e(v)}^2 \dd{x} \leq C \int_{B(x_0,r_0) \setminus K} \abs{e(u)}^2 \dd{x}.
    \end{equation*}
    For this purpose, we use the extension $v_h$ built in Lemma \ref{lem_extension}. However, there is a certain set $\ZZ$ where we don't control the energy of $v_h$ so we will force $v_h = 0$ in an open set $O$ containing $\ZZ$ and we will add $\partial O$ to the crack.

    We use Lemma \ref{lem_goodradius} to select a radius $\rho \in [r_0/2, 3 r_0/4]$ such that
    \begin{equation}\label{eq_choice_rho}
        \sum_{i \in I(\rho)} r_i^{N-1} \leq C \varepsilon_0 m(x_0,r_0) r_0^{N-1},
    \end{equation}
    where
    \begin{equation*}
        I(\rho) = \set {i \in I | 10 B_i \cap \partial B(x_0,\rho) \ne \emptyset}.
    \end{equation*}
    Then we introduce
    \begin{equation*}
        G := L \cup \bigcup_{i \in I(\rho)} \partial (10 B_i).
    \end{equation*}
    We observe that
    \begin{equation*}
        G \setminus B\left(x_0, 9r_0/10\right) = K \setminus B\left(x_0, 9r_0/10\right)
    \end{equation*}
    because $L$ coincides with $K$ outside $B(x_0,r_0/2)$ and for all $i \in I(\rho)$, $10 \overline{B}_i \subset B(x_0, 9 r_0/10)$.
    Next, we are going to justify that $G$ is relatively closed in $\Omega$.
    The family of spheres $(\partial (10 B_i))_{i \in I(\rho)}$ might have accumulation points but we are going to see that they can only be located on $K \cap \partial B(x_0,\rho)$, which is a subset of $L$ and thus of $G$.
    For 
    \begin{equation*}
        x \in \overline{\bigcup_{i \in I(\rho)} \partial (10 B_i)} \setminus \bigcup_{i \in I(\rho)} \partial (10 B_i),   
    \end{equation*}
    we can extract a sequence of distincts elements $(10 B_{i_k})_{k \in \mathbf{N}}$ from the family $(10 B_i)_{i \in I(\rho)}$ such that $\mathrm{dist}(x, 10 \overline{B}_{i_k}) \to 0$ when $k \to +\infty$.
    The balls $(B_{i_k})_k$ are disjoint and contained in $B(x_0,r_0)$ so their radii must go to $0$ as $k$ goes to $+\infty$.
    We deduce that $\mathrm{dist}(x,x_{i_k}) \to 0$, where $x_{i_k}$ denotes the center of $B_{i_k}$. 
    But since the points $x_{i_k}$ belong to $K$ and the balls $10 B_{i_k}$ meet $\partial B(x_0,\rho)$ by definition of $I(\rho)$, we must have $x \in K \cap \partial B(x_0,\rho) \subset L$.
    This proves that we have in fact
    \begin{equation*}
        G = L \cup \overline{\bigcup_{i \in I(\rho)} \partial (10 B_i)}.
    \end{equation*}
    To conclude this paragraph, we use (\ref{eq_choice_rho}) and (\ref{eq_L}) to estimate
    \begin{align}
        \HH^{N-1}(G \cap B(x_0,r_0)) &\leq \HH^{N-1}(L \cap B(x_0,r_0)) + \sum_{i \in I(\rho)} \HH^{N-1}(\partial (10 B_i))\notag\\
                                     &\leq \HH^{N-1}(K \cap B(x_0,r_0)) - C^{-1} m(x_0,r_0/4) r_0^{N-1} + C \varepsilon_0 m(x_0,r_0) r_0^{N-1}.\label{eq_energy_G}
    \end{align}

    Now, we build a function $v \in LD(B(x_0,r_0) \setminus G)$ such that $(v,G)$ is a competitor of $(u,K)$ in $B(x_0,9r_0/10)$ and the energy of $v$ is under control, i.e.,
    \begin{equation}\label{eq_energy_v}
        \int_{B(x_0,r_0) \setminus G} \abs{e(v)}^2 \dd{x} \leq C \int_{B(x_0,r_0) \setminus K} \abs{e(u)}^2 \dd{x}.
    \end{equation}
    We apply Lemma \ref{lem_extension} with respect to the geometric function $\delta$ defined in (\ref{eq_delta}) and the radius $\rho \in [r_0/2, 3r_0/4]$ selected above.
    We obtain functions $v_h \in LD_{\mathrm{loc}}(V_h)$ (for $h=1,2$) and a relatively closed subset $S_h$ of $V_h = \Omega_h \cup W$ such that
    \begin{equation*}
        \W \subset S_h \subset \WW,\qquad
        v_h = u \ \text{in} \ V_h \setminus S_h
    \end{equation*}
    and
    \begin{equation*}
        \int_{V_h \cap B(x_0,\rho) \setminus \ZZ} \abs{e(v_h)}^2 \dd{x} \leq C \int_{B(x_0,\rho) \cap \Omega_h} \abs{e(u)}^2 \dd{x}.
    \end{equation*}
    We recall that
    \begin{align*}
        \W      &= \bigcup \set{B(x,\delta(x)/U) | x \in K \cap \overline{B}(x_0,\rho)}\\
        \WW     &= \bigcup \set{B(x,10\delta(x)/U) | x \in K \cap \overline{B}(x_0,\rho)}
    \end{align*}
    and
    \begin{equation*}
        \ZZ := \bigcup \set{B(x,10 \delta(x)/U) | x \in K \cap \overline{B}(x_0,\rho),\ B(x,50 \delta(x)/U) \cap \partial B(x_0,\rho) \ne \emptyset}.
    \end{equation*}
    We are going to justify that
    \begin{equation}\label{eq_XV}
        X_h \subset V_h
    \end{equation}
    and
    \begin{equation}\label{eq_FB}
        \ZZ \subset \bigcup_{i \in I(\rho)} 10 B_i.
    \end{equation}
    The inclusion (\ref{eq_XV}) is a rather easy implication of (\ref{eq_X}) and the fact that, since $\delta(x_i) = r_i = A d(x_i)$ and $A = U A_0$, we have
    \begin{equation*}
        D_i = B(x_i, A_0 d(x_i)) = B(x_i, \delta(x_i)/U).
    \end{equation*}
    Now, we consider $x \in K \cap \overline{B}(x_0,\rho)$ such that $B(x, 50 \delta(x)/U) \cap \partial B(x_0,\rho) \ne \emptyset$.
    As $x$ belongs to $K \cap \overline{B}(x_0,\rho)$ and is such that $\delta(x) > 0$, then it follows from the definition of $\delta$ that there exists $i$ such that $x \in 9 B_i$ so
    \begin{equation*}
        \delta(x) \leq \abs{x - x_i} + r_i \leq 10 r_i.
    \end{equation*}
    In addition, recall that $\U = 10^5$ so $B(x,50 \delta(x)/U) \subset 10 B_i$.
    We have proved (\ref{eq_FB}).

    In this last paragraph, we are going to use repeatedly and without mention the fact that when $V \subset B(x_0,r_0)$ is an open set that is disjoint from $K \cap \partial B(x_0,\rho)$ (we think mainly of the case $V \subset X_h$), then the set $V \setminus \bigcup_{i \in I(\rho)} 10 \overline{B}_i$ is open.
    This is not straightforward because $\bigcup_{i \in I(\rho)} 10 \overline{B}_i$ may not be closed but we have already seen this kind of argument: for the points
    \begin{equation*}
        x \in \overline{\bigcup_{i \in I(\rho)} 10 B_i} \setminus \bigcup_{i \in I(\rho)} 10 \overline{B}_i,
    \end{equation*}
    we can show that $x \in K \cap \partial B(x_0,\rho)$ so $x \notin V$.
    This proves that
    \begin{equation*}
        V \setminus \bigcup_{i \in I(\rho)} 10 \overline{B}_i = V \setminus \overline{\bigcup_{i \in I(\rho)} 10 B_i}.
    \end{equation*}
    We are ready to define our function $v$.
    According to (\ref{eq_XV}), we have $X_h \subset V_h$ so $v_h \in LD_{\mathrm{loc}}(X_h)$.
    By (\ref{eq_FB}) and the comment just below Lemma \ref{lem_extension}, we actually have
    \begin{equation*}
        v \in LD\left(X_h \setminus \bigcup_{i \in I(\rho)} 10 \overline{B}_i\right)
    \end{equation*}
    and
    \begin{equation*}
        v_h = u \quad \text{in} \quad X_h \setminus \left(B(x_0,\rho) \cup \bigcup_{i \in I(\rho)} 10 \overline{B}_i\right).
    \end{equation*}
    We finally define $v \in LD(B(x_0,r_0) \setminus G)$ by
    \begin{equation*}
        v =
        \begin{cases}
            v_1 &\text{in} \ X_1 \setminus \bigcup_{i \in I(\rho)} 10 \overline{B}_i\\
            v_2 &\text{in} \ X_2 \setminus \bigcup_{i \in I(\rho)} 10 \overline{B}_i\\
            u   &\text{in} \ B(x_0,r_0) \setminus \left(L \cup X_1 \cup X_2 \cup \bigcup_{i \in I(\rho)} 10 \overline{B}_i\right)\\
            0   &\text{in} \ \bigcup_{i \in I(\rho)} 10 B_i.
        \end{cases}
    \end{equation*}
    This is a well-defined function in $LD(B(x_0,r_0) \setminus G)$ because the piecewise domains in the construction are disjoint open sets which cover $\Omega \setminus G$.
    We have $v = u$ outside $B(x_0,9r_0/10)$ because $B(x_0,\rho) \subset B(x_0,9r_0/10)$ and for all $i \in I(\rho)$, $10 \overline{B}_i \subset B(x_0,9r_0/10)$.
    The pair $(v,G)$ is thus a competitor of $(u,K)$ in $B(x_0,9r_0/10)$ and we can compare their energies (with (\ref{eq_energy_G}) and (\ref{eq_energy_v})) to obtain
    \begin{multline*}
        \int_{B(x_0,r_0) \setminus K} \abs{e(u)}^2 \dd{x} + \HH^{N-1}(K \cap B(x_0,r_0)) \leq C \int_{B(x_0,r_0) \setminus K} \abs{e(u)}^2 \dd{x} + \HH^{N-1}(K \cap B(x_0,r_0)) \\- C^{-1} m(x_0,r_0/4) r_0^{N-1} + C \varepsilon_0 m(x_0,r_0) r_0^{N-1} + h(r_0) r_0^{N-1}.
    \end{multline*}
    Of course, we can apply the whole proof with $\varepsilon_0 = \beta(x_0,r_0)$.
\end{proof}

\subsection{Control of the minimality defect}

{\color{black}
    The next proposition shows that the minimality defect is bounded by the normalized elastic energy and the error terms $\beta m$ and $h$.
    This control in turns the flatness via Lemma \ref{lem_flatness}.
}

\newcommand{\cc}{\kappa}
\begin{proposition}[Control of the minimality defect]\label{prop_flatness_decay}
    Let $(u,K)$ be a Griffith almost-minimizer with gauge $h$ in $\Omega$.
    Let $x_0 \in K$, $r_0 > 0$ and $\varepsilon_0 \geq 0$ be such that (\ref{eq_varepsilon0}) holds, $K$ satisfies Hypothesis-$H(\varepsilon_0,x_0,r_0)$ and $h(r_0) \leq \eaf$.
    Then for all deformation competitor $L$ of $K$ in $B(x_0, \cc r_0)$, we have
    \begin{multline}\label{eq_flatness_decay}
        \left[\HH^{N-1}(K \cap B(x_0, \cc r_0)) - \HH^{N-1}(L \cap B(x_0, \cc r_0))\right] \\ \leq C \left[\omega(x_0,r_0) + \beta(x_0,r_0) m(x_0,r_0) + h(r_0)\right] r_0^{N-1},
    \end{multline}
    for some universal constant $C \geq 1$ and $\cc = 10^{-6}$.
\end{proposition}

{\color{black}
    We don't know apriori what is the competitor of $K$ with minimal area in $B(x_0,\cc r_0)$.
    In dimension $2$, if $K \cap \partial B(x_0,\cc r_0)$ is composed of two points $p,q$, the best possible competitor consists in replacing $K$ by the segment $[p,q]$. 
    In higher dimension, if $K$ coincides with a Lipschitz graph on $\partial B(x_0, \cc r_0)$, a clever choice of competitor would be the harmonic graph spanning $K \cap \partial B(x_0, \cc r_0)$.
    These are the typical competitors used in the theory of almost-minimal sets.

    Let us explain how Proposition \ref{prop_flatness_decay} allows to control the flatness.
    Let $\EG \in (0,1)$ and $\CG \geq 1$ be the universal constants introduced in Theorem \ref{thm_flatness} for $\gamma = 1/2$.
    Let us assume $\varepsilon_0$ small enough so that $\varepsilon_0 \leq \EG$ and let us fix a universal parameter $0 < a \leq 1/2$ such that $4 \CG \sqrt{a} \leq 1$.
    Since $\beta(x_0,r_0) \leq \varepsilon_0$, Lemma \ref{lem_flatness} tells us that if the right-hand side of (\ref{eq_flatness_decay}) is small enough compared to $\varepsilon_0$, then $\beta(x_0, a \cc r_0) \leq 4 \CG \sqrt{a} \varepsilon_0 \leq \varepsilon_0$.
    Hence we see that when $\omega$, $m$ and $h$ are small, the flatness decays.
}

\begin{proof}
    Throughout the proof, the letter $C$ is a universal constant $\geq 1$ whose value might change from one line to another.
    The proof is very similar to the proof of Proposition \ref{prop_badmass_decay} but instead of deforming $K$ in small balls $D_i$, we will deform $K$ is a single large ball at the center.
    This will force us to enlarge the geometric function $\delta$ at the center.
    Most of the technical details are the same as in the proof Proposition \ref{prop_badmass_decay} so we will often skip them. 

    We will work with the geometric function $\delta_1 : K \cap \overline{B}(x_0,3r_0/4) \to [0,r_0/4]$ defined by
    \begin{equation}\label{eq_delta1}
        \delta_1(x) := \max\set{r_0/4 - \abs{x - x_0}, \delta(x)}.
    \end{equation}
    It is still a $1$-Lipschitz geometric function with parameters $(3r_0/4, 4\tau)$ but now we have $\delta_1(x_0) = r_0/4$.
    We recall that $\U = 10^5$ and we observe that for $\cc = 1 / (10 U) = 10^{-6}$, we have $B(x_0, \cc r_0) \subset B(x_0, \delta_1(x_0)/U)$.
    This means that the domain of the extension in Lemma \ref{lem_extension} will be large enough to contain $B(x_0, \cc r_0)$.

    We fix $f$ be a deformation of $K$ in $B(x_0, \cc r_0)$ and we introduce $L := f(K)$.
    We note that $L$ coincides with $K$ outside $B(x_0,\cc r_0) \subset B(x_0,r_0/2)$.
    We recall that there exists an hyperplane $P_0$ (of unit vector $\nu_0$) passing through $x_0$ such that $K \cap B(x_0,r_0) \subset \set{\mathrm{dist}(\cdot,P_0) \leq \varepsilon_0}$ and such that the sets
    \begin{equation*}
        \begin{gathered}
            \set{x \in B(x_0,r_0) | (x - x_0) \cdot \nu_0 > \varepsilon_0 r_0}\\
            \set{x \in B(x_0,r_0) | (x - x_0) \cdot \nu_0 < -\varepsilon_0 r_0}
        \end{gathered}
    \end{equation*}
    belong to different connected component of $B(x_0,r_0) \setminus K$, denoted by $\Omega_1$ and $\Omega_2$.
    According to the theory of Borsuk maps (see Remark \ref{rmk_borsuk}), the sets
    \begin{equation*}
        \begin{gathered}
            \set{x \in B(x_0,r_0) | (x - x_0) \cdot \nu_0 > r_0/2}\\
            \set{x \in B(x_0,r_0) | (x - x_0) \cdot \nu_0 < -r_0/2}
        \end{gathered}
    \end{equation*}
    also lie in different connected component of $B(x_0,r_0) \setminus L$, denoted by $X_1$ and $X_2$ respectively, and we have 
    \begin{equation}\label{eq_X2}
        X_h \subset \Omega_h \cup B(x_0, \cc r_0).
    \end{equation}
    We note that $X_h \setminus B(x_0,r_0/2) = \Omega \setminus B(x_0,r_0/2)$.

    We use Lemma \ref{lem_goodradius} to select a radius $\rho \in [r_0/2, 3r_0/4]$ such that
    \begin{equation*}
        \sum_{i \in I(\rho)} r_i^{N-1} \leq C \varepsilon_0 m(x_0,r_0) r_0^{N-1},
    \end{equation*}
    where
    \begin{equation*}
        I(\rho) = \set {i \in I | 10 B_i \cap \partial B(x_0,\rho) \ne \emptyset}.
    \end{equation*}
    Then we define
    \begin{equation*}
        G := L \cup \bigcup_{i \in I(\rho)} \partial (10 B_i).
    \end{equation*}
    We observe that $G$ is relatively closed in $\Omega$ and
    \begin{equation*}
        G \setminus B\left(x_0,9r_0/10\right) = K \setminus B\left(x_0,9 r_0/10\right).
    \end{equation*}
    We apply Lemma \ref{lem_extension} with respect to the geometric function $\delta_1$ defined in (\ref{eq_delta1}).
    We obtain functions $v_h \in LD_{\mathrm{loc}}(V_h)$ (for $h=1,2$) and a relatively closed subsets $S_h$ of $V_h = \Omega_h \cup W$ such that
    \begin{equation*}
        \W \subset S_h \subset \WW,\qquad
        v_h = u \ \text{in} \ V_h \setminus S_h
    \end{equation*}
    and
    \begin{equation*}
        \int_{V_h \cap B(x_0,\rho) \setminus \ZZ} \abs{e(v_h)}^2 \dd{x} \leq C \int_{B(x_0,\rho) \cap \Omega_h} \abs{e(u)}^2 \dd{x}.
    \end{equation*}
    We check that
    \begin{equation}\label{eq_XV2}
        X_h \subset V_h
    \end{equation}
    and
    \begin{equation}\label{eq_Z2}
        \ZZ \subset \bigcup_{i \in I(\rho)} 10 B_i.
    \end{equation}
    The proof of the similar inclusions in Proposition \ref{prop_badmass_decay} does not apply directly because $\delta_1$ is larger than $\delta$.
    The inclusion (\ref{eq_XV2}) still works thanks to (\ref{eq_X2}) and because we have chosen $\cc$ so that $B(x_0, \cc r_0) \subset B(x_0, \delta_1(x_0)/U)$.
    We pass to (\ref{eq_Z2}).
    For $x \in K \cap \overline{B}(x_0,\rho)$ such that $B(x, 50 \delta_1(x)/U) \cap \partial B(x_0,\rho) \ne \emptyset$, we have
    \begin{equation*}
        \abs{x - x_0} \geq \rho - 50 \delta_1(x)/U \geq r_0/2 - 50 r_0/(4 U) \geq r_0/4
    \end{equation*}
    so $\delta_1(x) = \delta(x)$ and from there, we can follow the proof of the similar inclusion in Proposition \ref{prop_badmass_decay}.

    We can finally define $v \in LD(B(x_0,r_0) \setminus G)$ by
    \begin{equation*}
        v =
        \begin{cases}
            v_1 &\text{in} \ X_1 \setminus \bigcup_{i \in I(\rho)} 10 \overline{B}_i\\
            v_2 &\text{in} \ X_2 \setminus \bigcup_{i \in I(\rho)} 10 \overline{B}_i\\
            u   &\text{in} \ B(x_0,r_0) \setminus \left(L \cup X_1 \cup X_2 \cup \bigcup_{i \in I(\rho)} 10 \overline{B}_i\right)\\
            0   &\text{in} \ \bigcup_{i \in I(\rho)} 10 B_i.
        \end{cases}
    \end{equation*}
    It is clear that $v = u$ outside $B(x_0,9 r_0/10)$.
    The pair $(v,G)$ is a competitor of $(u,K)$ in $B(x_0,9r_0/10)$ so we can compare their energies and obtain
    \begin{multline*}
        \int_{B(x_0,r_0) \setminus K} \abs{e(u)}^2 \dd{x} + \HH^{N-1}(K \cap B(x_0,r_0)) \leq C \int_{B(x_0,r_0) \setminus K} \abs{e(u)}^2 \dd{x} + \HH^{N-1}(L \cap B(x_0,r_0))\\ + C \varepsilon_0 m(x_0,r_0) r_0^{N-1} + h(r_0) r_0^{N-1}.
    \end{multline*}
\end{proof}

\subsection{Decay of the energy}

{\color{black}
    Our last estimate deals with the decay of the normalized elastic energy.
    The decay of $\omega$ as a power is normally an elliptic regularity property (when $h = 0$, $u$ solves a elliptic PDE with a Neumann boundary condition on each side of $K$) which requires $K$ to be regular enough.
    Since we don't know the regularity of $K$ a priori and we have a gauge $h$ in the minimality condition, we do not get $\omega(x,r) \leq C r^\alpha$ at once. Instead, we show that $\omega$ decays when $\beta$ is small enough and when the error terms $\beta m$ and $h$ are small compared to $\omega$.
}

\newcommand{\epse}{\varepsilon_e}
\newcommand{\eeta}{a}
\begin{proposition}[Decay of the energy]\label{prop_energy_decay}
    Let $(u,K)$ be a Griffith almost-minimizer with gauge $h$ in $\Omega$.
    Let $x_0 \in K$, let $r_0 > 0$ be such that $B(x_0,r_0) \subset \Omega$, $h(r_0) \leq \eaf$ and $K$ separates $B(x_0,r_0)$.
    For all $0 < b \leq 1$, there exists $\epse > 0$ (depending on $N$, $b$) such that if
    \begin{equation*}
        \beta(x_0,r_0) \leq \epse \quad \text{and} \quad m(x_0,r_0)\beta(x_0,r_0) + h(r_0) \leq \epse \omega(x_0,r_0),
    \end{equation*}
    then
    \begin{equation*}
        \omega(x_0,b r_0)\leq C b \omega(x_0,r_0),
    \end{equation*}
    for some universal constant $C \geq 1$.
\end{proposition}

\newcommand{\CB}{C_b}
{\color{black}
    \begin{remark}\label{rmk_energy_decay}
        Note that if $\beta(x_0,r_0) \leq \epse$, either $m(x_0,r_0)\beta(x_0,r_0) + h(r_0) \leq \epse \omega(x_0,r_0)$ holds and we have $$\omega(x_0,b r_0)\leq C b \omega(x_0,r_0),$$ either it does not hold and we have
        $$\omega(x_0,r_0) \leq \epse^{-1}(m(x_0,r_0)\beta(x_0,r_0) + h(r_0)).$$
        In all cases, we have
        \begin{equation}\label{eq_energy_decay}
            \omega(x_0,b r_0) \leq C b \omega(x_0,r_0) + \CB [m(x_0,r_0)\beta(x_0,r_0) + h(r_0)],
        \end{equation}
        where $\CB \geq 1$ depends on $N$ and $b$.
        We will always use Proposition \ref{prop_energy_decay} via (\ref{eq_energy_decay}) in the rest of the paper.
    \end{remark}
}

{\color{black}
    As usual, the term $m(x_0,r_0) \beta(x_0,r_0)$ comes from the wall set of the extension Lemma \ref{lem_extension}.
    It would be simpler to use an extension Lemma that adds a wall $\ZZ$ of measure $\leq C \beta(x_0,r_0) r_0^{n-1}$ (as in \cite[Lemma 4.2]{bil}) but the resulting estimate would be $$\omega(x_0,b r_0) \leq C b \omega(x_0,r_0) + \CB [\beta(x_0,r_0) + h(r_0)]$$ instead of (\ref{eq_energy_decay}) and this would not be good enough to conclude that $\omega$ decays a power.
    One of the point of the bad mass is that once we know that $\beta$ stays small at all scales and location (i.e. that $K$ is Reifenberg-flat), then $m$ disappears and (\ref{eq_energy_decay}) gives the decay of $\omega$ in a quite straightforward way (provided that $h$ decays as a power).
}

\begin{proof}
    Throughout the proof, the letter $C$ is a universal constant $\geq 1$ whose value might change from one line to another.
    The case $1/2 \leq b \leq 1$ is trivial so we only focus on the case $0 < b < 1/2$.
    By scaling, it suffices to prove the result in $B := B(0,1)$.
    Then, the argument is by contradiction and compactness.
    We fix a parameter $b \in (0,1/2)$ and a constant $C_e \geq 1$.
    If the Lemma is false for this choice of constants, then there exists a sequence of Griffith almost-minimizers $(v_n,K_n)$ with gauge $h_n$ in $B$ such that $0 \in K_n$, $h_n(1) \leq \eaf$, the set $K_n$ separates $B(0,1)$, $\beta_n(0,1) \to 0$,
    \begin{equation*}
        \omega_n(0,1)^{-1}\left[\beta_n(0,1) m_n(0,1) + h_n(1)\right] \to 0,
    \end{equation*}
    and 
    \begin{equation*}
        \omega_n(0,b)> C_e b \omega_n(0,1),\label{ineq1}
    \end{equation*}
    where $\beta_n, \omega_n$, $m_n$ are the flatness, the normalized elastic energy and the bad mass defined with respect to $(v_n,K_n)$.
    {\color{black}
        Now, our goal is to find a contradiction for $C_e$ big enough, but not depending on $b$.
        We underline that the sets $K_n$ are all Ahlfors-regular in $B(0,1)$ with the same constant because we have $h_n(1) \leq \eaf$ uniformly in $n$.
    }

    We let $P_0$ be the vector subspace of $\R^N$ generated by the first $(N-1)$ vectors of the canonic base.
    Up to apply a sequence of rotations we may assume that $P_0$ achieves the minimium in the definition of $\beta_n(0,1)$.
    We decompose $\R^N = P_0 \times P_0^\perp$ and we use the notation $x = (x', x_N)$, where $x' \in P_0$ and $x_N \in \R$, for an element of $\R^N$.
    We normalize the energy by setting
    \begin{equation*}
        u_n(x):= \frac{1}{\sqrt{e_n}} v_n(x),
    \end{equation*}
    where
    \begin{equation*}
        e_n:= \int_{B \setminus K_n}|e(v_n)|^2 \dd{x}.
    \end{equation*}
    This is well defined because \eqref{ineq1} directly implies $e_n > 0$.
    In summary, for all $n$, we have $u_n \in LD(B \setminus K_n)$ and the following properties:
    \begin{equation}\label{eq_Kn_minimality}
        \begin{gathered}
            \text{for all competitor } (v,L) \text{ for }  (u_n,K_n) \text{ in } B(0,9/10) \text{ we have } \\
            \int_{B \setminus K_n} |e(u_n)|^2 \dd{x} + e_n^{-1} \mathcal{H}^{N-1}(K_n) \leq \int_{B \setminus K_n}|e(v)|^2 \dd{x} + e_n^{-1} \mathcal{H}^{N-1}(L) + e_n^{-1} h_n(1),
        \end{gathered}
    \end{equation}
    \begin{equation*}
        \int_{B \setminus K_n}|e(u_n)|^2 \dd{x} = 1,\label{energyBound}
    \end{equation*}
    \begin{equation*}
        0 \in K_n \ \text{and} \ \varepsilon_n := \max \Set{\sup_{x \in K_n \cap \overline{B}(0,1)} \mathrm{dist}(x,P_0), \sup_{x \in P_0 \cap \overline{B}(0,1)} \mathrm{dist}(x,K_n)} \to 0,\label{hausdorffD}
    \end{equation*}
    \begin{equation*}
        \int_{B(0,b)\setminus K_n}|e(u_n)|^2 \dd{x} \geq C_e b^N.\label{eq_contradiction}
    \end{equation*}
    We also write $m_n := m_n(0,1)$ to simplify the notations and we recall that
    \begin{equation*}
        e_n^{-1} (\varepsilon_n m_n + h_n(1)) \to 0.
    \end{equation*}

    The strategy of the proof is to be able, by some sort of $\Gamma$-convergence technique, to prove that $u_n$ admits a limit $u$ which is a minimizer in $B \setminus P_0$ with energy $1$ and for which an inequality of the type \eqref{eq_contradiction} is contradicted by elliptic regularity (if $C_e$ is too big).
    To get a contradiction with that approach, the main issue is to prove a strong convergence in $L^2$ for $e(u_n)$.
    This convergence is false in general, but here it follows from the minimality property (\ref{eq_Kn_minimality}) of $(u_n,K_n)$ and the fact that $K_n$ separates.

    {\color{black}
        At the beginning, our proof is similar to \cite[Theorem 8.19]{afp} but our last step is different and relies on the construction of an appropriate competitor inspired by \cite[Theorem 9]{l3}.
        What makes the construction of a competitor delicate is that the surface term is penalized in the energy comparison (\ref{eq_Kn_minimality}) by a coefficient $e_n^{-1}$ which might go to $+\infty$.
        For each $n$, we will build a competitor $v_n$ of $u_n$ which adds a wall set of size $\beta_n m_n$ to the crack and this contribution will disappear in the limit thanks to the assumption $e_n^{-1} \beta_n m_n \to 0$.
    }

    In the sequel will shall need to extract several subsequences of $u_n$, that we will still denote by $u_n$, for simplicity.

    \vspace{0.5cm}
    \noindent\emph{Step 1. Convergence locally in $B \setminus P_0$.}
    We first extract a subsequence such that $e(u_n)$ converges weakly in $L^2$ to some $e\in L^2(B;\mathbb M^{N \times N})$, thanks to the energy bound in \eqref{energyBound}.
    We next show that $e$ is the symmetric gradient of some displacement. To this aim, we fix $0<\delta<1/10$ and we introduce the Lipschitz domain
    \begin{equation*}
        A_\delta:= \set{x \in B | \mathrm{dist}(x,P_0) > \delta} = A^+_\delta \cup A^-_{\delta},
    \end{equation*}
    where $A^\pm_\delta$ are the connected components of $A_\delta$. Note that for such $\delta$, $D^\pm := B\left((0,\pm \frac{1}{2}) ,\frac{1}{4}\right) \subset A^\pm_\delta$ and $K_n \cap A_\delta = \emptyset$ for $n$ large enough (depending on $\delta$). Denoting by 
    $$R_n^\pm(x):=\frac{1}{|D^\pm|}\int_{D^\pm} u_n(y) \,dy +\left(\frac{1}{|D^\pm|}\int_{D^\pm} \frac{\nabla u_n(y)-\nabla u_n(y)^{T}}{2}dy\right)\left(x-\frac{1}{|D^\pm|}\int_{D^\pm} y \,dy\right),$$
    the rigid motion associated to $u_n$ in $D^\pm$, by virtue of the Poincar\'e-Korn inequality \cite[Theorem 5.2]{AMR}, we get that 
    \begin{equation*}
        \|u_n-R^\pm_n\|_{H^1(A^\pm_\delta;\R^N)}\leq c \| e(u_n)\|_{L^2(A^\pm_\delta;\mathbb M^{N\times N})},
    \end{equation*}
    for some constant $c>0$ that does not depend on $\delta$ since $0 < \delta < 1/10$.

    Thanks to a diagonalisation argument, for a further subsequence (not relabeled), we obtain a function $u \in H^1(B \setminus P_0;\R^N)$ such that $(u_n - R^\pm_n)_n$ converges weakly to $u$ in $H^1(A^\pm_\delta;\R^N)$ and strongly in $L^2(A^\pm_\delta;\R^N)$, for any $0<\delta<1/10$.
    In particular, we must have $e = e(u)$. This shows that
    \begin{equation*}
        e(u_n) \rightharpoonup e(u) \quad \text{ weakly in } L^2(B;\mathbb M^{N \times N}),
    \end{equation*}
    and
    \begin{equation}\label{eq_semicontinuity}
        \int_{B \setminus P_0} |e(u)|^2 \dd{x} \leq \liminf_{n \to +\infty} \int_{B \setminus K_n} |e(u_n)|^2 \dd{x} = 1.
    \end{equation}

    \vspace{0.5cm}
    \noindent\emph{Step 2. The limit is a minimizer.} At this stage, we are able to  prove that  the limit $u$ is an energy minimizer in $B \setminus P_0$, thus is the weak solution of an elliptic system with Neumann boundary conditions on $P_0$.
    For that purpose, we will use a jump transfer argument, relying on the fact that $K_n$ separates.
    We consider a test function $\varphi \in H^1(B \setminus P_0;\R^N)$ such that $\varphi = 0$ on $B \setminus B(0,9/10)$.
    We denote by $C_n^\pm$ the connected components of $B \setminus K_n$ that contains the point $(0,\pm 1/2)$ and we define a function $\varphi_n \in H^1(B \setminus K_n;\R^N)$ by
    \begin{equation*}
        \varphi_n(x) :=
        \begin{cases}
            \varphi(x',\abs{x_N})  & \ \text{in} \ C_n^+\\
            \varphi(x',-\abs{x_N}) & \ \text{in} \ C_n^-\\
            0                      & \ \text{otherwise}.
        \end{cases}
    \end{equation*}
    Then, one can check that $\varphi_n=0$ on $B \setminus B(0,9/10)$ and that $\varphi_n(x) = \varphi(x)$ for $x \in B$ such that $\abs{y} \geq \varepsilon_n$.
    By \eqref{hausdorffD}, $\varepsilon_n \to 0$ so $\varphi_n \to \varphi$ strongly in $L^2(B;\R^N)$ and $e(\varphi_n) \to e(\varphi)$ strongly in $L^2(B;\mathbb M^{N \times N})$.
    Therefore, using the minimality property (\ref{eq_Kn_minimality}) to compare $(u_n,K_n)$ and $(u_n + \varphi_n,K_n)$ and we obtain that
    \begin{equation*}
        \int_{B \setminus K_n} \abs{e(u_n)}^2 \dd{x} \leq \int_{B \setminus K_n} \abs{e(u_n+\varphi_n)}^2 \dd{x} + e_n^{-1} h_n(1), 
    \end{equation*}
    which implies, by expanding the squares, that
    \begin{equation*}
        0 \leq 2 \int_{B \setminus K_n} e(u_n):e(\varphi_n) \dd{x} +  \int_{B \setminus K_n} \abs{e(\varphi_n)}^2 \dd{x} + e_n^{-1} h_n(1).
    \end{equation*}
    Using that $e_n^{-1} h_n(1) \to 0$, that $e(\varphi_n) \to e(\varphi)$ strongly in $L^2(B;\mathbb M^{N \times N})$ and $e(u_n) \rightharpoonup e(u)$ weakly in $L^2(B;\mathbb M^{N \times N})$, we can pass to the limit $n \to +\infty$ and deduce
    \begin{equation*}
        0 \leq 2 \int_{B \setminus P_0} e(u):e(\varphi) \dd{x} + \int_{B \setminus P_0} \abs{e(\varphi)}^2 \dd{x},
    \end{equation*}
    or still
    \begin{equation*}
        \int_{B \setminus P_0} \abs{e(u)}^2 \dd{x} \leq \int_{B \setminus P_0} \abs{e(u+\varphi)}^2 \dd{x}.
    \end{equation*}
    We conclude that $u$ is a weak solution of $\mathrm{div}(e(u)) = 0$ in $B(0,9/10) \setminus P_0$ with a Neumann condition on $P_0$.
    It follows by elliptic regularity and (\ref{eq_semicontinuity}) that
    \begin{equation}\label{eq_elliptic}
        \int_{B(0,b) \setminus P_0} \abs{e(u)}^2 \leq C b^n \int_{B \setminus P_0} \abs{e(u)}^2,
    \end{equation}
    for some universal constant $C \geq 1$.

    \vspace{0.5cm}
    \noindent\emph{Step 3. Convergence of the $L^2$ norms and conclusion.}
    To arrive at a contradiction between (\ref{eq_elliptic}) and (\ref{eq_contradiction}), we show that
    \begin{equation}\label{eq_energy_decay_goal}
        \lim_{n \to +\infty} \int_{B(0,b) \setminus K_n} \abs{e(u_n)}^2 = \int_{B(0,b) \setminus P_0} \abs{e(u)}^2 \dd{x}.
    \end{equation}
    For that purpose, we take any subsequence such that the left-hand side of (\ref{eq_energy_decay_goal}) converges and we consider the sequence of measures
    \begin{equation*}
        \mu_n:= |e(u_n)|^2 \mathcal{L}^N.
    \end{equation*}
    By \eqref{energyBound}, we can bound uniformly $\mu_n(B) \leq 1$ and thus we can extract a subsequence such that $\mu_n \overset{\ast}{\rightharpoonup} \mu$ in $B$, for some measure $\mu$.
    The rest of the proof is devoted to showing
    \begin{equation*}
        \mu = \abs{e(u)}^2 \mathcal{L}^N \quad \text{in $B(0,1/2)$},
    \end{equation*}
    which implies (\ref{eq_energy_decay_goal}).

    We start by investigating the structure of $\mu$ away from $P_0$.
    We claim that, up to use again a diagonal subsequence, we can assume that $(e(u_n))_n$ converges to $e(u)$ in $L^2_{\mathrm{loc}}(B \setminus P_0;\mathbb M^{N \times N})$.
    Indeed, let us take any ball $\overline{B}(x,r) \subset B \setminus P_0$.
    As $e(u_n)_n \rightharpoonup e(u)$ weakly in $L^2(B(x,r))$, the strong convergence in $L^2(B(x,r))$ will follow from the convergence of the $L^2$ norms.
    More specifically, it will follow from the fact that
    \begin{equation*}
        \limsup_{n \to +\infty} \int_{B(x,r)} |e(u_n)|^2 \dd{x} \leq \int_{B(x,r)} |e(u)|^2 \dd{x},
    \end{equation*}
    since the reverse inequality with a $\liminf$ is already a consequence of the weak convergence.

    We consider a small $\delta > 0$ such that $\overline{B}(x,r+\delta) \subset B \setminus P_0$ and a cut-off function $\varphi \in C^\infty_c(B(x,r+\delta))$ such that $0 \leq \varphi \leq 1$ and $\varphi = 1$ on $B(x,r)$.
    We also work with $n$ big enough so that $\overline{B}(x,r+\delta) \subset B \setminus K_n$.
    We compare $(u_n,K_n)$ with the competitor $(v_n,K_n)$, where
    \begin{equation*}
        v_n := \varphi (u + R_n^\pm) + (1 - \varphi) u_n.
    \end{equation*}
    We estimate
    \begin{equation*}
        \abs{e(v_n)} \leq \varphi \abs{e(u)} + (1 - \varphi) \abs{e(u_n)} + \abs{\nabla \varphi} \abs{u_n - R_n^\pm -u}
    \end{equation*}
    and then using the elementary inequality
    $$(a+b)^2\leq (1+\varepsilon)a^2+\left(1 + \varepsilon^{-1}\right)b^2 \quad \text{for all $\varepsilon > 0$}$$
    and the convexity of $t \mapsto t^2$, we have
    \begin{align*}
        \int_{B(x,r+\delta)} |e(v_n)|^2 \dd{x} &\leq \begin{multlined}[t] (1+\varepsilon) \int_{B(x,r+\delta)} \left(\varphi \abs{e(u)} + (1-\varphi) \abs{e(u_n)}\right)^2 \dd{x} \\+ \left(1 + \varepsilon^{-1}\right) \int_{B(x,r+\delta)} |\nabla \varphi|^2 \abs{u_n - R_n^\pm - u}^2 \dd{x}
        \end{multlined}\\
                                               &\leq \begin{multlined}[t] (1+\varepsilon)\int_{B(x,r+\delta)}\varphi |e(u)|^2 +(1-\varphi)|e(u_n)|^2 \dd{x} \\+ \left(1 + \varepsilon^{-1}\right)\int_{B(x,r+\delta)} |\nabla \varphi|^2 \abs{u_n- R_n^\pm - u}^2 \dd{x}.
                                               \end{multlined}
        \end{align*}
        We use almost-minimality property (\ref{eq_Kn_minimality}) of $(u_n,K_n)$ to compare the energy of $u_n$ and $v_n$,
        \begin{equation*}
            \int_{B(x,r+\delta)} \abs{e(u_n)}^2 \dd{x} \leq \int_{B(x,r+\delta)} \abs{e(v_n)}^2 + e_n^{-1} h_n(1)
        \end{equation*}
        but using the previous estimate, this gives
        \begin{multline*}
            \int_{B(x,r+\delta)} \varphi \abs{e(u_n)}^2 \dd{x} \leq (1+\varepsilon) \int_{B(x,r+\delta)} \varphi \abs{e(u)}^2 \dd{x} + \varepsilon \int_B \abs{e(u_n)}^2 \\+ \left(1 + \varepsilon^{-1}\right) \int_B |\nabla \varphi|^2 \abs{u_n- R_n^\pm - u}^2 \dd{x} + e_n^{-1} h_n(1)
        \end{multline*}
        and then by definition of $\varphi$,
        \begin{multline*}
            \int_{B(x,r)} \abs{e(u_n)}^2 \leq (1+\varepsilon) \int_{B(x,r+\delta)} \abs{e(u)}^2 \dd{x} + \varepsilon \int_B \abs{e(u_n)}^2 \dd{x} \\+ \left(1 + \varepsilon^{-1}\right) \int_{B(x,r+\delta)} |\nabla \varphi|^2 \abs{u_n- R_n^\pm - u}^2 \dd{x} + e_n^{-1} h_n(1).
        \end{multline*}
        We recall that $\int_B \abs{e(u_n)}^2 \dd{x} \leq 1$ and that $u_n - R_n^\pm \to u$ strongly in $L^2(B(x,r+\delta))$ so passing to the limit $n \to +\infty$ gives
        \begin{equation*}
            \limsup_{n \to +\infty} \int_{B(x,r)} \abs{e(u_n)}^2 \dd{x} \leq (1 + \varepsilon) \int_{B(x,r+\delta)} \abs{e(u)}^2 \dd{x} + \varepsilon
        \end{equation*}
        but since $\varepsilon > 0$ and $\delta > 0$, are arbitrary, we conclude that
        \begin{equation*}
            \limsup_{n\to +\infty} \int_{B(x,r)} |e(u_n)|^2 \dd{x} \leq \int_{B(x,r)} |e(u)|^2 \dd{x},
        \end{equation*}
        as claimed.

        At this point, we have proved that $(e(u_n))_n$ converges $e(u)$ in $L^2_{\mathrm{loc}}(B \setminus P_0)$ and therefore
        \begin{equation*}
            \mu \mres \left(B \setminus P_0\right) =|e(u)|^2 \mathcal{L}^N.
        \end{equation*}

        We now show that
        \begin{equation*}
            \mu\left(P_0\cap B\left(0,\tfrac{1}{2}\right)\right) = 0.
        \end{equation*}
        Some of the following quantities will depend over $n$ but we will sometimes not specify this dependence to lighten the notation. By use of Lemma \ref{lem_goodradius}, we select a radius $\rho_n \in [1/2, 3/4]$ such that
        \begin{equation*}
            \sum_{i \in I(\rho_n)} r_i^{N-1} \leq C_0 \varepsilon_n m_n
        \end{equation*}
        for some constant universal constant $C_0 \geq 1$ and where
        \begin{equation*}
            I(\rho_n) = \set {i \in I | 10 B_i \cap \partial B(0,\rho_n) \ne \emptyset}.
        \end{equation*}
        We can extract a subsequence so that $(\rho_n)_n$ converges to some $\rho_\infty \in [1/2, 3/4]$.
        We would also like that $\mu(\partial B(0,\rho_\infty)) = 0$ but we cannot guarantee this property directly.
        We need to refine our procedure.
        Let us fix a small $\varepsilon > 0$ and let $S$ be the finite set of radii $\rho \in [1/2, 3/4]$ such that $\mu(\partial B(0,\rho)) \geq \varepsilon$.
        According to the proof of Lemma \ref{lem_goodradius}, there exists a universal constant $C_0 \geq 1$ such that
        \begin{equation*}
            \int_{1/2}^{3/4}  \sum_{i \in I(t)} r_i^{N-1} \dd{t} \leq C_0 \varepsilon_n m_n
        \end{equation*}
        so for any $C \geq 1$, the set
        \begin{equation}\label{eq_badset1}
            \Set{t \in \left[\tfrac{1}{2}, \tfrac{3}{4}\right] | \sum_{i \in I(t)} r_i^{N-1} \dd{t} \geq C \varepsilon_n m_n}
        \end{equation}
        has $\LL^1$ measure $\leq C^{-1} C_0$.
        We take $C$ big enough (depending on $N$) such that $C^{-1} C_0 < 1/8$.
        Since $S$ is finite, we can also consider a small radius $r > 0$ (which does not depend on $n$) such that the set
        \begin{equation}\label{eq_badset2}
            \Set{t \in \left[\tfrac{1}{2},\tfrac{3}{4}\right] | \mathrm{dist}(t,S) < r}
        \end{equation}
        has $\LL^1$ measure $< 1/8$.
        In this way, it is clear that the two bad sets (\ref{eq_badset1}) and (\ref{eq_badset2}) cannot cover $[1/2, 3/4]$ so we can find $\rho_n \in [1/2, 3/4]$ such that $\mathrm{dist}(\rho_n,S) \geq r$ and
        \begin{equation}\label{eq_energy_radius}
            \sum_{i \in I(\rho_n)} r_i^{N-1} \dd{t} \leq C \varepsilon_n m_n.
        \end{equation}
        We can extract a subsequence so that $(\rho_n)_n$ converges to some $\rho_\infty \in [1/2, 3/4]$ and we have moreover, $\mathrm{dist}(\rho_\infty, S) \geq r$.
        In particular, $\rho_\infty \notin S$ so
        \begin{equation}\label{eq_partial_rho}
            \mu(\partial B(0,\rho_\infty)) \leq \varepsilon.
        \end{equation}
        It will suffice to make $\varepsilon \to 0$ at the end of the proof.

        Now we fix $n$ and we write $\rho = \rho_n$, $K = K_n$ to lighten the notation.
        We introduce
        \begin{equation*}
            G :=  K \cup \bigcup_{i \in I(\rho)} \partial (10 B_i).
        \end{equation*}
        We observe that $G$ is relatively closed in $\Omega$ and that
        \begin{equation*}
            G \setminus B\left(0, 9/10\right) = K \setminus B\left(0,9/10\right).
        \end{equation*}
        We fix some $0 < \eeta \leq 1/10$ and we consider the function $f : [0,+\infty[ \to [0,+\infty)$ defined by
        \begin{equation*}
            f(t) =
            \begin{cases}
                \eeta    & \ \text{for} \ t \leq \rho - \eeta\\
                \rho - t   & \ \text{for} \ \rho - \eeta \leq t \leq \rho\\
                0       & \ \text{for} \ t \geq \rho.
            \end{cases}
        \end{equation*}
        We will work with the geometric function $\delta_1 : K \cap \overline{B}(0,3/4) \to [0,1/4]$ defined by
        \begin{equation}\label{eq_delta2}
            \delta_1(x) := \max{\set{f(\abs{x}), \delta(x)}}.
        \end{equation}
        We recall that according to Lemma \ref{lem_delta}, we have $\delta(x) \leq \max_k (10 r_k) \leq 10 C \varepsilon_n$, whereas $f(\abs{x}) \leq \eeta$ by construction. We assume $n$ big enough such that $10 C \varepsilon_n \leq \eeta$ and in particular, $\delta_1(x) \leq \eeta$, with equality at points $x \in K \cap \overline{B}(0,\rho - \eeta)$.

        We denote by $\Omega_1$, $\Omega_2$ the two connected components of $B(0,1) \setminus K$ that contains respective components of
        \begin{equation*}
            \Set{x \in B(0,1) | \mathrm{dist}(x,P_0) > 1/2}.
        \end{equation*}
        We apply Lemma \ref{lem_extension} with respect to the geometric function $\delta_1$ defined in (\ref{eq_delta2}).
        We obtain functions $v_h \in LD_{\mathrm{loc}}(V_h)$ (for $h=1,2$) and a relatively closed subset $S_h$ of $V_h = \Omega_h \cup W$ such that
        \begin{equation*}
            \W \subset S_h \subset \WW,\qquad
            v_h = u \ \text{in} \ V_h \setminus S_h
        \end{equation*}
        and, according to Remark \ref{rmk_vh},
        \begin{equation}\label{eq_rmk_vh}
            \sum_{h=1,2} \int_{S_h \setminus \ZZ} \abs{e(v_h)}^2 \dd{x} \leq C \int_{D} \abs{e(u_n)}^2 \dd{x},
        \end{equation}
        where
        \begin{equation*}
            D := B(0,\rho) \cap \bigcup \set{y \in B(x, 50 \delta_1(x) / U) | x \in K \cap \overline{B}(0,\rho),\ \mathrm{dist}(y,K) \geq \delta_1(x)/U}.
        \end{equation*}
        We also recall that $U = 10^5$, that
        \begin{align*}
            \W      &= \bigcup \set{B(x,\delta_1(x)/U) | x \in K \cap \overline{B}(0,\rho)}\\
            \WW     &= \bigcup \set{B(x,10\delta_1(x)/U) | x \in K \cap \overline{B}(0,\rho)}
        \end{align*}
        and
        \begin{equation*}
            \ZZ := \bigcup \set{B(x,10 \delta_1(x)/U) | x \in K \cap \overline{B}(0,\rho),\ B(x,50 \delta_1(x)/U) \cap \partial B(0,\rho) \ne \emptyset}.
        \end{equation*}
        Using the inclusion $S_h \subset \WW$ and the definition of $\ZZ$, we see that
        \begin{equation*}
            S_h \setminus B(x_0,\rho) \subset \ZZ
        \end{equation*}
        Next, let us check that
        \begin{equation*}
            \ZZ \subset \bigcup_{i \in I(\rho)} 10 B_i.
        \end{equation*}
        It is easy to see that for all $0 \leq t \leq \rho$, we have $t + f(t) \leq \rho$ so for $x \in K \cap \overline{B}(0,\rho)$, we have $B(x,50 f(\abs{x})/U) \subset B(0,\rho)$.
        In the case where $B(x,50 \delta_1(x)/U) \cap \partial B(0,\rho) \ne \emptyset$, we have necessarily $\delta_1(x) = \delta(x)$ and from there, we can follow the proof of Proposition \ref{prop_badmass_decay}.

        In order to use (\ref{eq_rmk_vh}), we will to replace the sets $S_h \setminus \ZZ$ and $D$ by simpler domains which do not depend on $n$.
        We introduce
        \begin{equation*}
            C(\eeta) := \set{y \in \R^N | \mathrm{dist}(y, P_0 \cap \overline{B}(0,\rho - 2 \eeta)) < \eeta/(2U)}
        \end{equation*}
        and we justify that whenever $n$ is big enough such that $\varepsilon_n < \eeta / (2 U)$, we have
        \begin{equation}\label{eq_rmk_vh1}
            C(\eeta) \subset B(0,\rho) \cap W
        \end{equation}
        and
        \begin{equation}\label{eq_rmk_vh2}
            D \subset \set{y \in \overline{B}(0,\rho) | \mathrm{dist}(y,P_0) \leq \eeta} \setminus C(\eeta).
        \end{equation}
        Let's start with (\ref{eq_rmk_vh1}).
        It is already clear that $C(\eeta) \subset B(0,\rho)$.
        For $y \in C(\eeta)$, there exists $x_0 \in P_0 \cap \overline{B}(0,\rho - 2 \eeta)$ such that $\abs{x_0 - y} < \eeta / (2U)$.
        Since $\varepsilon_n < \eeta / (2U)$, there also exists $x \in K$ such that $\abs{x_0 - x} < \eeta / (2U)$.
        Notice that $x \in K \cap \overline{B}(0,\rho - \eeta)$ so $\delta_1(x) = \eeta$.
        Then it follows by triangular inequality that $y \in B(x, \delta_1(x)/U)$ and thus $y \in \W$.
        We pass to (\ref{eq_rmk_vh2}).
        For $y \in D$, there exists $x \in K \cap \overline{B}(0,\rho)$ such that $y \in B(x, 50 \delta_1(x)/U)$ and $\mathrm{dist}(y,K) \geq \delta_1(x) / U$.
        As $\abs{y - x} \leq 50 \eeta/U$ and $\mathrm{dist}(x,P_0) \leq \varepsilon_n$, we have
        \begin{equation*}
            \mathrm{dist}(y,P_0) \leq 50 \eeta/U + \varepsilon_n \leq \eeta.
        \end{equation*}
        We check that $y \notin C(\eeta)$ by contradiction. We assume that there exists $x_0 \in P_0 \cap \overline{B}(0,\rho - 2 \eeta)$ such that $\abs{y - x_0} < \eeta / (2 U)$. As $\mathrm{dist}(x_0,K) \leq \varepsilon_n$, we have
        \begin{equation*}
            \mathrm{dist}(y,K) < \eeta / (2 U) + \varepsilon_n < \eeta / U.
        \end{equation*}
        Recalling that $\abs{x - y} \leq 50 \eeta/U$, we also have $x \in \overline{B}(0,\rho - \eeta)$ by the triangular inequality and thus $\delta_1(x) = \eeta$. But then, our initial assumption on $y$ means $\mathrm{dist}(y,K) \geq \eeta / U$. Contradiction !

        We make a last observation before defining our competitor $v$.
        According to Lemma \ref{lem_W} (applied with the geometric function $\delta_1$), we have
        \begin{equation}\label{eq_W_inclusion}
            \overline{B}(0,\rho) \setminus \left(K \cup \Omega_1 \cup \Omega_2\right) \subset W
        \end{equation}
        Here, observe that $B(0,1) \setminus (K \cup \Omega_1 \cup \Omega_2)$ is an open set which is the union of the connected component of $B(0,1) \setminus K$ that are neither $\Omega_1$, neither $\Omega_2$.
        The inclusion (\ref{eq_W_inclusion}) says that $W$ cover them in $\overline{B}(0,\rho)$ and this is related to the fact seen in Lemma \ref{lem_delta} that for all $x \in K \cap \overline{B}(0,\rho)$, for all $r \in (\delta(x),1/4]$, the flatness $\beta_K(x,t)$ is very small (less than some universal constant).

        It follows from (\ref{eq_W_inclusion}) that
        \begin{equation*}
            \partial B(0,\rho) \setminus \left(K \cup \Omega_1 \cup \Omega_2\right) \subset \ZZ \subset \bigcup_{i \in I(\rho)} 10 B_i.
        \end{equation*}
        The set
        \begin{equation*}
            B(0,1) \setminus \left(K \cup \Omega_1 \cup \Omega_2 \cup \bigcup_{i \in I(\rho)} 10 \overline{B}_i\right)
        \end{equation*}
        is open because $B(0,1) \setminus \left(K \cup \Omega_1 \cup \Omega_2\right)$ is open and because the topological boundary of the union $\bigcup_{i \in I(\rho)} 10 \overline{B}_i$ lies in $K \cap \partial B(0,\rho)$, as we have seen in earlier proofs.
        And according to (\ref{eq_W_inclusion}), its connected components are either contained in $B(0,\rho)$, or contained in $B(0,1) \setminus \overline{B}(0,\rho)$.
        We will set our competitor to be $v = 0$ in the first case and $v = u_n$ in the second case.

        We finally define $v \in LD(B(0,1) \setminus G)$ by
        \begin{equation*}
            v =
            \begin{cases}
                v_1 &\text{in } \ \Omega_1 \setminus \bigcup_{i \in I(\rho)} 10 \overline{B}_i\\
                v_2 &\text{in } \ \Omega_2 \setminus \bigcup_{i \in I(\rho)} 10 \overline{B}_i\\
                u_n   &\text{in } \ B(0,1) \setminus \left(K \cup \Omega_1 \cup \Omega_2 \cup \overline{B}(0,\rho) \cup \bigcup_{i \in I(\rho)} 10 \overline{B}_i\right)\\
                0   &\text{in } \ B(0,\rho) \setminus \left(K \cup \Omega_1 \cup \Omega_2 \cup \bigcup_{i \in I(\rho)} 10 \overline{B}_i\right)\\
                0   &\text{in } \ \bigcup_{i \in I(\rho)} 10 B_i.
            \end{cases}
        \end{equation*}
        Note that each domain in the piecewise definition is open.
        The pair $(v,G)$ is a competitor of $(u_n,K)$ in $B(0,9/10)$ so we can compare their energies using (\ref{eq_Kn_minimality}).
        We recall that by (\ref{eq_energy_radius}), the spheres $\bigcup_{i \in I(\rho)} \partial (10 B_i)$ add to the crack a contribution bounded by
        \begin{equation*}
            \sum_{i \in I(\rho)} r_i \leq C \varepsilon_n m_n
        \end{equation*}
        so
        \begin{equation}\label{eq_energy_comparison}
            \int_{B(0,1)} \abs{e(u_n)}^2 \dd{x} \leq \int_{B(0,1)} \abs{e(v)}^2 \dd{x} + C e_n^{-1} (\varepsilon_n m_n + h_n(1)).
        \end{equation}

        We observe that for $x \in B(0,1) \setminus \left(G \cup S_1 \cup S_2\right)$, we have either $v = 0$ or $v = u_n$ in a neighborhood of $x$ (check this for each domain in the definition of $v$) and in both cases
        \begin{equation}\label{eq_energy_simplify}
            \abs{e(v)} \leq \abs{e(u_n)}.
        \end{equation}
        The inequality (\ref{eq_energy_simplify}) also holds true for points $x \in B(0,1) \setminus (G \cup B(0,\rho))$ since either $x \notin S_1 \cup S_2$ and we are back to the previous case, or $x \in (S_1 \cup S_2) \setminus B(0,\rho) \subset \ZZ \subset \bigcup_i 10 B_i$, where $v = 0$.
        In summary, (\ref{eq_energy_simplify}) holds true for
        \begin{equation*}
            x \in B(0,1) \setminus \left[G \cup \left(B(0,\rho) \cap (S_1 \cup S_2)\right)\right]
        \end{equation*}
        so the energy comparison (\ref{eq_energy_comparison}) simplifies to
        \begin{equation*}
            \int_{B(0,\rho) \cap (S_1 \cup S_2)} \abs{e(u_n)}^2 \leq \int_{B(0,\rho) \cap (S_1 \cup S_2)} \abs{e(v)}^2 + C e_n^{-1} (\varepsilon_n m_n + h_n(1)).
        \end{equation*}
        For the points $x \in B(0,\rho) \cap (S_1 \cup S_2) \setminus G$, we observe furthermore that either $x \in \bigcup_i 10 B_i$, or $x \in B(0,\rho) \setminus (K \cup \Omega_1 \cup \Omega_2 \cup \bigcup_i 10 \overline{B}_i)$ or $x \in \bigcup_{h = 1,2} \Omega_h \cap S_h \setminus \bigcup_i 10 B_i$.
        Here, we have used the fact that $\Omega_1 \cap \Omega_2 = \emptyset$ and $W \subset S_h \subset \Omega_h \cup W$ to deduce $(S_1 \cup S_2) \cap \Omega_1 \subset S_1 \cap \Omega_1$ (and the same with the indices reversed).
        In the first two cases, we have $v(x) = 0$ in a neighborhood of $x$ so we can also simplify the right-hand side as
        \begin{equation*}
            \int_{B(0,\rho) \cap (S_1 \cup S_2)} \abs{e(u_n)}^2 \leq \sum_{h=1,2} \int_{\Omega_h \cap S_h \setminus \bigcup 10 B_i} \abs{e(v_h)}^2 + C e_n^{-1} (\varepsilon_n m_n + h_n(1)).
        \end{equation*}
        To simplify the notations in the integral sign, we have written $\bigcup 10 B_i$ for $\bigcup_{i \in I(\rho)} 10 B_i$.
        We use the inclusion $\ZZ \subset \bigcup_{i \in I(\rho)} 10 B_i$ and (\ref{eq_rmk_vh}) to bound
        \begin{equation*}
            \sum_{h=1,2} \int_{\Omega_h \cap S_h \setminus \bigcup 10 B_i}  \abs{e(v_h)}^2 \leq \sum_{h=1,2} \int_{S_h \setminus \ZZ} \abs{e(v_h)}^2 \leq C \int_{D} \abs{e(u_n)}^2
        \end{equation*}
        so we arrive at
        \begin{equation*}
            \int_{B(0,\rho) \cap (S_1 \cup S_2)} \abs{e(u_n)}^2 \leq C \int_{D} \abs{e(u_n)}^2 + C e_n^{-1} (\varepsilon_n m_n + h_n(1)).
        \end{equation*}

        Then we come back to the notation $\rho = \rho_n$ and we set
        \begin{align*}
            C_n(\eeta) &:= \set{y \in \R^N | \mathrm{dist}(y, P_0 \cap \overline{B}(0,\rho_n - 2 \eeta)) < \eeta/(4U)}\\
            C_\infty(\eeta) &:= \set{y \in \R^N | \mathrm{dist}(y, P_0 \cap \overline{B}(0,\rho_\infty - 2 \eeta)) < \eeta/(4U)}.
        \end{align*}
        We rely on the inclusions (\ref{eq_rmk_vh1}) and (\ref{eq_rmk_vh2}) to simplify the domains of integration;
        \begin{equation*}
            \int_{C_n(\eeta)} \abs{e(u_n)}^2 \dd{x} \leq C \int_{\overline{B}(0,\rho_n) \cap \set{\mathrm{dist}(\cdot,P_0) \leq \eeta} \setminus C_n(\eeta)} \abs{e(u_n)}^2 \dd{x} + C e_n^{-1} (\varepsilon_n m_n + h_n(1)).
        \end{equation*}
        We let $n \to +\infty$ and get
        \begin{equation*}
            \mu\left(C_\infty(\eeta)\right) \leq C \mu\left(\overline{B}(0,\rho_\infty) \cap \set{\mathrm{dist}(\cdot, P_0) \leq \eeta} \setminus C_\infty(\eeta)\right).
        \end{equation*}
        This yields in particular
        \begin{equation*}
            \mu\left(P_0 \cap B(0,\rho_\infty - 2 \eeta)\right) \leq C \mu\left(\overline{B}(0,\rho_\infty) \cap \set{\mathrm{dist}(\cdot,P_0) \leq \eeta} \setminus P_0 \cap B(0,\rho_\infty - 2 \eeta)\right)
        \end{equation*}
        and we finally make $\eeta \to 0$ to get
        \begin{equation*}
            \mu\left(P_0 \cap B(0,\rho_\infty)\right) \leq C \mu\left(P_0 \cap \partial B(0,\rho_\infty)\right).
        \end{equation*}
        According to the way we built $\rho_\infty$ (see (\ref{eq_partial_rho})), this implies $\mu\left(P_0 \cap B(0,\rho_\infty)\right) \leq \varepsilon$.
        Here, $\rho_\infty$ might depend on $\varepsilon$ but since $\rho_\infty \geq 1/2$ anyway, we conclude that
        \begin{equation*}
            \mu\left(P_0 \cap B\left(0,\tfrac{1}{2}\right)\right) = 0.
        \end{equation*}
    \end{proof}


    \section{Joint decay and conclusion}


    {\color{black}
        We show in Lemma \ref{lem_decay} that when all the quantities are small, they stay small at smaller scales. We deduce in Proposition \ref{prop_main} that the bad mass is actually zero and that $K$ is an almost-minimal set with a gauge of optimal exponent $\alpha$. Then, our $\varepsilon$-regularity theorem will follow from the regularity theory of almost-minimal sets.
    }

    \begin{lemma}\label{lem_decay}
        Let $(u,K)$ be a Griffith almost-minimizer with gauge $h$ in $\Omega$.
        Let $x_0 \in K$, $r_0 > 0$ be such that $B(x_0,r_0) \subset \Omega$ and $K$ separates $B(x_0,r_0)$.
        For all $\varepsilon_0 > 0$, there exists $\varepsilon \in (0,\varepsilon_0)$ (that depends on $N$, $\varepsilon_0$) such that if
        \begin{equation*}
            \beta(x_0,r_0) + \omega(x_0,r_0) + m(x_0,r_0) + h(r_0) \leq \varepsilon,
        \end{equation*}
        then for all $0 < r \leq r_0$, $K$ separates $B(x_0,r)$ and
        \begin{equation*}
            \beta(x_0,r) + \omega(x_0,r_0) + m(x_0,r_0) + h(r_0) \leq \varepsilon_0.
        \end{equation*}
    \end{lemma}

    \begin{proof} 
        We use the letter $C$ as a generic constant $\geq 1$ that depends on $N$.
        We let $b \in (0,1)$, this constant will be fixed later and will depend only on $N$.
        We let $\varepsilon_0 \in ]0,1/2]$ and $\varepsilon_1$, $\varepsilon_2$, $\varepsilon_3 \in ]0,\varepsilon_0]$.
        They will will be chosen small enough depending on $N$ and $b$ but each $\varepsilon_i$ will also depend on the $\varepsilon_j$ for $j < i$.

        Our goal is to prove that the conditions
        \begin{equation}\label{eq_iterate0}
            \begin{gathered}
                \text{$K$ separates $B(x_0,r_0)$ and}\\
                \beta(x_0,r_0) \leq \varepsilon_0, \quad m(x_0,r_0) \leq \varepsilon_1,\quad  \omega(x_0,r_0) \leq \varepsilon_2, \quad h(r_0) \leq \varepsilon_3
            \end{gathered}
        \end{equation}
        imply that for all $0 < r \leq r_0$,
        \begin{equation*}
            \begin{gathered}
                \text{$K$ separates $B(x_0,r)$ and}\\
                \beta(x_0,r) \leq C \varepsilon_0, \quad m(x_0,r) \leq C \varepsilon_1, \quad \omega(x_0, r) \leq C \varepsilon_2, \quad h(r) \leq \varepsilon_3.
            \end{gathered}
        \end{equation*}
        One can deduce the Lemma's statement.

        We assume (\ref{eq_iterate0}) and we start by showing that
        \begin{equation*}
            \begin{gathered}
                \text{$K$ separates $B(x_0,b r_0)$ and}\\
                \beta(x_0,b r_0) \leq \varepsilon_0, \quad m(x_0,b r_0) \leq \varepsilon_1,\quad \omega(x_0,b r_0) \leq \varepsilon_2, \quad h(b r_0) \leq \varepsilon.
            \end{gathered}
        \end{equation*}
        It is clear that $h(b r_0) \leq \varepsilon_3$ since $h$ is non-decreasing.
        It is also readily seen with Lemma \ref{lem_separation} that $K$ separates $B(x_0, br_0)$ if $\varepsilon_0$ is small enough.
        We assume that (\ref{eq_varepsilon0}) holds so that Propositions \ref{prop_badmass_decay}, \ref{prop_flatness_decay}, and \ref{prop_energy_decay} can be applied.
        We also assume $\varepsilon_0$ small enough so that $m(x_0,t)$ can be estimated by scaling (Remark \ref{rmk_badmass}) for $t \in [b r_0,r_0]$.

        We deal with the decay of the bad mass.
        According to Proposition \ref{prop_badmass_decay} and the scaling property of $m$, there exists a universal constant $C_0 \geq 1$ such that
        \begin{equation*}
            m(x_0,b r_0) \leq C_0 (\varepsilon_2 + \varepsilon_0 \varepsilon_1 + \varepsilon_3).
        \end{equation*}
        We choose $\varepsilon_1$ in such a way that
        \begin{equation*}
            C_0 \varepsilon_2 = \varepsilon_1/4.
        \end{equation*}
        Then we choose $\varepsilon_0$ small enough so that $C_0 \varepsilon_0 \leq 1/4$ and finally we just choose $\varepsilon_3$ very small so that
        \begin{equation*}
            m(x_0,br_0) \leq \varepsilon_1.
        \end{equation*}

        We pass to the energy.
        We assume $\varepsilon_0$ less than the constant $\varepsilon_e$ given by Proposition \ref{prop_energy_decay} for our choice of $b$ (which will be fixed at some point).
        According to Proposition \ref{prop_energy_decay}, and more specifically Remark \ref{rmk_energy_decay}, there exists a constant $\CB \geq 1$ (depending on $N$, $b$) such that
        \begin{align*}
            \omega(x_0,br_0) &\leq C b \varepsilon_2 + \CB (\varepsilon_0 \varepsilon_1 + \varepsilon_3)\\
                             &\leq C b \varepsilon_2 + \CB (4 C_0 \varepsilon_0 \varepsilon_2 + \varepsilon_3).
        \end{align*}
        We take $b$ small enough so that $C b \leq 1/2$. We then assume $\varepsilon_0$ small enough so that $4 \CB C_0 \varepsilon_0 \leq 1/4$ and finally $\varepsilon_3$ small enough so that $\omega(x_0,br_0) \leq \varepsilon_2$.

        We are left to deal with the flatness.
        We recall that $\EG \in (0,1)$ and $\CG \geq 1$ are the universal constants of Theorem \ref{thm_flatness} for $\gamma = 1/2$.
        We recall that $\cc = 10^{-6}$ is the constant of Proposition \ref{prop_flatness_decay}.
        We apply Lemma \ref{lem_flatness} in the ball $B(x_0,\cc r_0)$ with the constants
        \begin{equation*}
            \varepsilon := \cc^{-1} \varepsilon_0 \quad \text{and} \quad a := \cc^{-1} b.
        \end{equation*}
        Here, we need to assume $\varepsilon_0$ small enough so that $\varepsilon \leq \EG$ and $b$ small enough so that $a \leq 1/2$.
        Since
        \begin{equation*}
            \beta(x_0, \cc r_0) \leq \cc^{-1} \beta(x_0,r_0) \leq \varepsilon,
        \end{equation*}
        Lemma \ref{lem_flatness} gives a constant $\lambda > 0$ (depending on $N$, $\varepsilon_0$) such that if for all deformation competitor $L$ of $K$ in $B(x_0, \cc r_0)$,
        \begin{equation*}
            \HH^{N-1}(K \cap B(x_0, \cc r_0)) - \HH^{N-1}(L \cap B(x_0, \cc r_0)) \geq \lambda (\cc r_0)^{N-1},
        \end{equation*}
        then we have
        \begin{equation*}
            \beta(x_0,b r_0) = \beta(x_0, a \cc r_0) \leq 4 \CG \sqrt{a} \lambda.
        \end{equation*}
        Moreover, Proposition \ref{prop_flatness_decay} estimates
        \begin{align}
            \HH^{N-1}(K \cap B(x_0, \cc r_0)) - \HH^{N-1}(L \cap B(x_0, \cc r_0)) &\leq C (\varepsilon_0\varepsilon_1 + \varepsilon_2 + \varepsilon_3 ) r_0^{N-1}\notag\\
                                                                                  &\leq 3 C \varepsilon_1 r_0^{N-1}.\label{eq_minimality_decay}
        \end{align}
        We assume $\varepsilon_1$ small enough (depending on $N$, $\alpha$, $\varepsilon_0$) so that the left-hand side of (\ref{eq_minimality_decay}) is $\leq \lambda \cc^{N-1}$ and this implies
        \begin{equation*}
            \beta(x_0, b r_0) \leq 4 \CG \sqrt{a} \varepsilon.
        \end{equation*}
        Substituting $\varepsilon = \cc^{-1} \varepsilon_0$ and $a = \cc^{-1} b$, we find
        \begin{equation*}
            \beta(x_0, b r_0) \leq 4 \CG \cc^{-3/2} \sqrt{b} \varepsilon_0
        \end{equation*}
        and we can finally fix $b$ small enough so that $\beta(x_0,br_0) \leq \varepsilon_0$.

        We now iterate our intermediate result and we obtain that for all $n \geq 0$, $K$ separates $B(x_0, b^n r_0)$ and
        \begin{equation*}
            \beta(x_0,b^n r_0) \leq \varepsilon_0, \quad m(x_0,b^n r_0) \leq \varepsilon_1, \quad \omega(x_0,b^n r_0) \leq \varepsilon_2, \quad h(b^n r_0) \leq \varepsilon_3.
        \end{equation*}
        We deduce that  for all $0 < r \leq r_0$, $K$ separates $B(x_0,r)$ and 
        \begin{equation*}
            \beta(x_0,r) \leq C \varepsilon_0, \quad m(x_0, r) \leq C \varepsilon_1, \quad \omega(x_0, r) \leq C \varepsilon_2, \quad h(r) \leq \varepsilon_3.
        \end{equation*}
    \end{proof}

    \begin{proposition}\label{prop_main}
        There exists a universal $c \in (0,1)$ and for each choice of $\alpha \in (0,1)$, $\varepsilon_0 \in ]0,1/2]$, there exists $\varepsilon \in (0,\varepsilon_0)$ (depending on $N$, $\alpha$, $\varepsilon_0$) such that the following holds.
        Let $(u,K)$ be a Griffith almost-minimizer with gauge $h(t) = h(1) t^{\alpha}$ in $\Omega$.
        Let $x_0 \in K$, $r_0 > 0$ be such that $B(x_0,r_0) \subset \Omega$, $K$ separates $B(x_0,r_0)$ and 
        \begin{equation*}
            \beta(x_0,r_0) + \omega(x_0,r_0) + m(x_0,r_0) + h(r_0) \leq \varepsilon,
        \end{equation*}
        then for all $x \in K \cap B(x_0,c r_0)$ and for all $0 < r \leq c r_0$, we have
        \begin{equation*}
            \beta(x,r) \leq \varepsilon_0, \quad m(x,r) = 0, \quad \omega(x,r) \leq \varepsilon_0 \left(\frac{r}{r_0}\right)^\alpha
        \end{equation*}
        and $K$ is an almost-minimal set in $B(x_0,c r_0)$ with gauge $\tilde{h}(r) = \varepsilon_0 (r/r_0)^\alpha$.
    \end{proposition}
    \begin{proof}
        The letter $C$ as a generic constant $\geq 1$ that depends on $N$ and $\alpha$.
        Let $\omega_0 > 0$ be any constant.
        Let $\varepsilon_0 > 0$ (it will be chosen later, depending on $N$, $\alpha$ and $\omega_0$) and let $\varepsilon_1 \in (0,\varepsilon_0)$ be the associated constant in Lemma \ref{lem_decay} (depending on $N$, $\alpha$, $\varepsilon_0$).
        According to the scaling properties of our different quantities, there exists $\varepsilon > 0$ (which depends on $N$, $\alpha$ and $\varepsilon_0$) such that if
        \begin{equation}\label{eq_main_varepsilon}
            \beta(x_0,r_0) + \omega(x_0,r_0) + m(x_0,r_0) + h(r_0) \leq \varepsilon,
        \end{equation}
        then for all $x \in B(x_0,r_0/2)$, we have
        \begin{equation*}
            \beta(x,r_0/2) + \omega(x,r_0/2) + m(x,r_0/2) + h(r_0/2) \leq \varepsilon_1
        \end{equation*}
        which implies in turn by Lemma \ref{lem_decay} that for all $x \in B(x_0,r_0/2)$ and all $0 < r \leq r_0/2$,
        \begin{equation}\label{eq_thm_main_varepsilon0}
            \beta(x,r) + \omega(x,r) + m(x,r) + h(r) \leq \varepsilon_0.
        \end{equation}

        We assume $\varepsilon_0 \leq \tau$ (where $\tau$ is the universal constant defined in (\ref{eq_constantes}). It then follows from the definition of the bad mass, that for all $x \in B(x_0,r_0/2)$ and $0 <r \leq r_0/2$,
        \begin{equation*}
            m(x,t) = 0.
        \end{equation*}

        Next, we prove that $\omega$ decays with exponent $\alpha$.
        Let us fix $x \in B(x_0,r_0/2)$ and $r_1 = r_0/2$.
        Let $b \in (0,1)$ (which will be chosen small enough, depending on $N$, $\alpha$).
        We assume $\varepsilon_0$ small enough in (\ref{eq_thm_main_varepsilon0}) so that Proposition \ref{prop_energy_decay} applies and give a constant $\CB \geq 1$ (depending on $N$, $\alpha$, $b$) such that for all $0 < r \leq r_1$,
        \begin{equation}\label{eq_omega_decay}
            \omega(x,b r) \leq C b \omega(x,r) + \CB h(r).
        \end{equation}
        We fix $b$ such that $C b \leq b^\alpha/2$.
        We assume $\varepsilon_0$ small enough so that $\omega(x,r_1) \leq \omega_0$ and $\CB h(r_1) \leq \omega_0 b^\alpha / 2$.
        Since $h$ decays as a power $\alpha$, we see that for all $k \in \N$,
        \begin{equation}\label{eq_h_decay}
            \CB h(b^k r_1) = \CB h(r_1) b^{k\alpha} \leq \omega_0 b^{(k+1)\alpha}/2.
        \end{equation}
        We are going to deduce by induction on (\ref{eq_omega_decay}) that for all $k \in \N$,
        \begin{equation}\label{eq_omega_induction}
            \omega(x,b^k r_1) \leq \omega_0 b^{k \alpha}.
        \end{equation}
        The case $k = 0$ is straightforward and if (\ref{eq_omega_induction}) holds at rank $k$, it also holds at rank $k+1$ because by (\ref{eq_omega_decay}) and (\ref{eq_h_decay}),
        \begin{align*}
            \omega(x,b^{k+1} r_1) &\leq b^{\alpha} \omega(x,b^k r_1)/2 + \CB h(b^k r_1)\\
                                  &\leq b^{\alpha} \left(\omega_0 b^{k\alpha}\right)/2 + \omega_0 b^{(k+1) \alpha}/2\\
                                  &\leq \omega_0 b^{(k+1) \alpha}.
        \end{align*}
        This proves our claim.
        We then deduce from (\ref{eq_omega_induction}) and the scaling property of $\omega$ that for all $0 < r \leq r_1$,
        \begin{equation*}
            \omega(x,r) \leq C \omega_0 \left(\frac{r}{r_1}\right)^\alpha.
        \end{equation*}
        Note that since $h(r_1) \leq \omega_0$, we also have $h(r) = h(r_1) (r/r_1)^{\alpha} \leq \omega_0 (r/r_1)^\alpha$ for $0 < r \leq r_1$.

        We finally show that the set $K$ is almost-minimal in $B(x_0, \cc r_1)$.
        For all $x \in K \cap B(x_0,r_1)$ and $0 < r \leq r_1$, we apply Proposition \ref{prop_flatness_decay} in $B(x,r)$ and we obtain that for all deformation competitor $L$ in $B(x,\cc r)$, we have
        \begin{align*}
            \HH^{N-1}(K \cap B(x, \cc r)) - \HH^{N-1}(L \cap B(x, \cc r))   &\leq C \left[\omega(x,r) + h(r)\right] r^{N-1}\\
                                                                            &\leq C \omega_0 \left(\frac{\cc r}{r_0}\right)^\alpha.
        \end{align*}
        We conclude that $K$ is almost-minimal in $B(x_0, \cc r_1)$ with gauge $\tilde{h}(t) = C \omega_0 (t/r_0)^\alpha$.
        One can choose $\omega_0$ arbitrary small by taking $\varepsilon_0$ small enough and then $\varepsilon$ accordingly small in (\ref{eq_main_varepsilon}).
    \end{proof}

    From the conclusion of Proposition \ref{prop_main}, the set $K$ is almost-minimal in $B(x_0, c r_0)$ with a small flatness and a small gauge.
    Then, according to the regularity theory of almost-minimal sets, $K$ is a Hölder differentiable surface in a smaller ball.
    The exact statement can be found in \cite[Theorem 12.25]{d5} and says that if a coral almost-minimal set $E$ is close enough to ``full lenght minimal cone $X$'', then $E$ is a $C^{1,\gamma}$ version of $X$.
    The theorem is written for two dimensional sets but it hold in higher dimensions as well in the special case where $X$ is an hyperplane.
    We state a simplified version for hyperplanes and we justify below how it can be deduced from \cite[Theorem 12.25]{d5}.
    \begin{theorem}\label{thm_david}
        For each choice of $\alpha \in (0,1)$, there exists constants $\varepsilon_0 > 0$, and $\gamma \in (0,\alpha)$ (depending on $N$, $\alpha$) such that the following holds.
        Let $E$ be a coral almost-minimal set with gauge $h(t) = h(1) t^{\alpha}$ in some open set $\Omega$.
        Let $x_0 \in E$, $r_0 > 0$ be such that $B(x_0,100 r_0) \subset \Omega$ and
        \begin{equation}\label{eq_david_beta}
            \beta_E(x_0,100 r_0) + h(100 r_0) \leq \varepsilon_0,
        \end{equation}
        then there is a $C^{1,\gamma}$ diffeomorphism
        \begin{equation*}
            \phi \colon B(x_0,2 r_0) \longrightarrow \phi(B(x_0,2 r_0))  
        \end{equation*}
        such that $\phi(x_0) = x_0$, $D\phi(x_0)$ is the identity mapping, $\abs{\phi(y) - y} \leq 10^{-2} r_0$ for all $y \in B(x_0,2r_0)$ and
        \begin{equation*}
            E \cap B(x_0,r_0) = \phi(P \cap B(x_0,2r_0)) \cap B(x_0,r_0),
        \end{equation*}
        where $P$ is a given hyperplane which achieves the minimum in the definition of $\beta_E(x_0,100)$.
    \end{theorem}
    \begin{proof}
        David's theorem \cite[Theorem 12.25]{d5} requires that
        \begin{equation}\label{eq_david_density}
            f(100r_0) + d_{x_0,100r_0}(E,X) + h(100r_0) \leq \varepsilon_0,
        \end{equation}
        where
        \begin{equation}\label{eq_defi_f}
            f(r) = \frac{\HH^{N-1}(E \cap B(x_0,r))}{\omega_{N-1} r^{N-1}} - \lim_{s \to 0} \left(\frac{\HH^{N-1}(E \cap B(x_0,s))}{\omega_{N-1} s^{N-1}}\right)
        \end{equation}
        is the density excess, $\omega_{N-1}$ is the measure of a $(N-1)$-dimensional unit disk and $d_{x_0,100r_0}(E,X)$ is the normalized local Hausdorff distance between $E$ and a ``full lenght minimal cone $X$'' centred at $x_0$. 
        The right-hand side limit in (\ref{eq_defi_f}) exists because for a coral almost-minimal set $E$ containing $x_0$, the density $s \mapsto s^{1-N} \HH^{N-1}(E \cap B(x_0,s))$ is nearly non-decreasing (\cite[Proposition 5.24]{d3}).

        The goal of this proof is to justify that the condition (\ref{eq_david_beta}) is a simplification of (\ref{eq_david_density}) in the special case where $X$ is an hyperplane passing through $x_0$. First, the bilateral flatness $\beta_E(x_0,100 r_0)$ is nothing else than the normalized local Hausdorff distance between $E$ and an hyperplane passing through $x_0$. Next, we justify that $f(100 r_0)$ is controlled by $\beta_E(x_0,100 r_0)$.
        We start by showing
        \begin{equation}\label{eq_lower_bound_density}
            \lim_{s \to 0} \frac{\HH^{N-1}(E \cap B(x_0,s))}{s^{N-1}} \geq \omega_{N-1}.
        \end{equation}
        Let $\theta$ denotes the value of the limit in (\ref{eq_lower_bound_density}).
        We have necessarily $\theta > 0$ because almost-minimal sets are Ahlfors-regular (\cite[Lemma 2.15]{d3}).
        According to \cite[Proposition 7.21]{d3}, there exists a coral minimal cone $F$ centered at the origin (take any blow-up of $E$ at $x_0$ in local Hausdorff distance) such that $\HH^{N-1}(F \cap B(0,1)) = \theta$.
        The set $F$ is a cone centred at the origin so for all $r > 0$,
        \begin{equation*}
            \HH^{N-1}(F \cap B(0,r)) = r^{N-1} \theta.
        \end{equation*}
        The set $F$ is also minimal in $\R^N$ so for all $x \in F$, the fonction $r \mapsto r^{1-N} \HH^{N-1}(F \cap B(x,r))$ is non-decreasing. One can deduce that for all $x \in F$ and for all $r > 0$,
        \begin{equation}\label{eq_F_density}
            \HH^{N-1}(F \cap B(x,r)) \leq \theta r^{N-1}.
        \end{equation}
        As a minimal set is rectifiable, we should have for $\HH^{N-1}$-a.e. $x \in F$,
        \begin{equation}\label{eq_F_rectifiable}
            \lim_{r \to 0} \frac{\HH^{N-1}(F \cap B(x,r))}{\omega_{N-1} r^{N-1}} = 1.
        \end{equation}
        And since $\HH^{N-1}(F) > 0$, there exists a least one point such that (\ref{eq_F_rectifiable}) holds.
        In combination with (\ref{eq_F_density}), this proves that $\theta \geq \omega_{N-1}$.
        With (\ref{eq_lower_bound_density}) at hand, we see that for $0 < r \leq 100 r_0$
        \begin{equation*}
            f(r) \leq \frac{\HH^{N-1}(E \cap B(x_0,r))}{\omega_{N-1} r^{N-1}} - 1
        \end{equation*}
        and we can see as in (\ref{eq_density_flatness}) that for $0 < r \leq 50 r_0$,
        \begin{equation*}
            \frac{\HH^{N-1}(E \cap B(x_0,r))}{\omega_{N-1} r^{N-1}} - 1 \leq C \beta(x_0,2r).
        \end{equation*}
        It is now clear that (\ref{eq_david_beta}) implies (\ref{eq_david_density}) (at a smaller scale but it does not matter).
    \end{proof}

    It is left to improve the assumptions of Proposition \ref{prop_main} by initializing $m$ and $\omega$. Our main theorem is proved by combining Lemma \ref{lem_initialization}, Proposition \ref{prop_main} and Theorem \ref{thm_david}.
    \begin{lemma}\label{lem_initialization}
        Let $(u,K)$ be a Griffith almost-minimizer with gauge $h$ in $\Omega$.
        Let $x_0 \in K$, $r_0 > 0$ be such that $B(x_0,r_0) \subset \Omega$, $K$ separates $B(x_0,r_0)$ and $h(r_0) \leq \eaf$.
        For all $\varepsilon_0 > 0$, there exists $c \in (0,1)$ and $\varepsilon \leq \varepsilon_0$ (both depending on $N$, $\varepsilon_0$) such that if
        \begin{equation*}
            \beta(x_0,r_0) + h(r_0) \leq \varepsilon,
        \end{equation*}
        then
        \begin{equation*}
            \beta(x_0,c r_0) + \omega(x_0, c r_0) + m(x_0, c r_0) + h(c r_0) \leq \varepsilon_0.
        \end{equation*}
    \end{lemma}
    \begin{proof}
        The letter $C$ denotes a generic universal constant $\geq 1$.
        We assume
        \begin{equation*}
            \beta(x_0,r_0) + h(r_0) \leq \varepsilon
        \end{equation*}
        and we initialize $\omega$ to be less than some constant $\varepsilon_1 > 0$ at a smaller scale.
        Let $b \in (0,1)$ (to be chosen soon, depending on $N$, $\varepsilon_1$) and let $\varepsilon_e$ be the associated constant in Proposition \ref{prop_energy_decay}. We assume $\varepsilon \leq \varepsilon_e$ so that Proposition \ref{prop_energy_decay} applies and give us a constant $\CB \geq 1$ (depending on $N$, $b$) such that
        \begin{align*}
            \omega(x_0, b r_0) &\leq C b \omega(x_0, br_0) + \CB \left(\beta(x_0,r_0) m(x_0,r_0) + h(r_0)\right)\\
                               &\leq C b + \CB (\beta(x_0,r_0) + h(r_0)).
        \end{align*}
        Now, we can choose $b$ small enough so that $C b \leq \varepsilon_1$ and then $\varepsilon$ small enough so that $\omega(x_0,br_0) \leq 2 \varepsilon_1$.
        We can also take $\varepsilon$ small enough so that
        \begin{equation}\label{eq_start2}
            \beta(x_0, br_0) + \omega(x_0,br_0) + h(b r_0) \leq 3 \varepsilon_1
        \end{equation}
        by the scaling property of $\beta$ and the monotonicity of $h$, 

        Next, we start from (\ref{eq_start2}) and we initialize $m$ to be less than some constant $\varepsilon_0 > 0$.
        We assume $\varepsilon_1$ small enough so that Proposition \ref{prop_badmass_decay} applies and we get,
        \begin{align*}
            m(x_0, b r_0/4) &\leq C(\omega(x_0,br_0) + \beta(x_0,br_0)m(x_0,br_0) + h(br_0))\\
                            &\leq C(\omega(x_0,br_0) + \beta(x_0,br_0) + h(br_0))\\
                            &\leq C \varepsilon_1
        \end{align*}
        which is $\leq \varepsilon_0$ if $\varepsilon_1$ is small enough.
    \end{proof}

    \section*{Acknowledgements}

    C. Labourie was funded by an internal research grant from the University of Cyprus and the DFG project FR 4083/3-1. He gratefully acknowledges their support.

    \bibliographystyle{plain}
    \bibliography{biblio_griffith}

\begin{thebibliography}{10}

\bibitem{AMR}
Giovanni Alessandrini, Antonino Morassi, and Edi Rosset.
\newblock The linear constraints in {P}oincar\'{e} and {K}orn type
  inequalities.
\newblock {\em Forum Math.}, 20(3):557--569, 2008.

\bibitem{AFH}
Luigi Ambrosio, Nicola Fusco, and John~E. Hutchinson.
\newblock Higher integrability of the gradient and dimension of the singular
  set for minimisers of the {M}umford-{S}hah functional.
\newblock {\em Calc. Var. Partial Differential Equations}, 16(2):187--215,
  2003.

\bibitem{afp}
Luigi Ambrosio, Nicola Fusco, and Diego Pallara.
\newblock {\em Functions of bounded variation and free discontinuity problems}.
\newblock Oxford Mathematical Monographs. The Clarendon Press, Oxford
  University Press, New York, 2000.

\bibitem{bil}
Jean-Fran\c{c}ois Babadjian, Flaviana Iurlano, and Antoine Lemenant.
\newblock Partial regularity for the crack set minimizing the two-dimensional
  {G}riffith energy.
\newblock {\em J. Eur. Math. Soc. (JEMS)}, 24(7):2443--2492, 2022.

\bibitem{b}
A.~Bonnet.
\newblock On the regularity of edges in image segmentation.
\newblock {\em Ann. Inst. H. Poincar\'e Anal. Non Lin\'eaire}, 13(4):485--528,
  1996.

\bibitem{CCI}
Antonin Chambolle, Sergio Conti, and Flaviana Iurlano.
\newblock Approximation of functions with small jump sets and existence of
  strong minimizers of {G}riffith's energy.
\newblock {\em J. Math. Pures Appl. (9)}, 128:119--139, 2019.

\bibitem{vito}
Antonin Chambolle and Vito Crismale.
\newblock Existence of strong solutions to the {D}irichlet problem for the
  {G}riffith energy.
\newblock {\em Calc. Var. Partial Differential Equations}, 58(4):Paper No. 136,
  27, 2019.

\bibitem{CFI}
Sergio Conti, Matteo Focardi, and Flaviana Iurlano.
\newblock Existence of strong minimizers for the {G}riffith static fracture
  model in dimension two.
\newblock {\em Ann. Inst. H. Poincar\'{e} Anal. Non Lin\'{e}aire},
  36(2):455--474, 2019.

\bibitem{dalM}
Gianni Dal~Maso.
\newblock Generalised functions of bounded deformation.
\newblock {\em J. Eur. Math. Soc. (JEMS)}, 15(5):1943--1997, 2013.

\bibitem{d3}
Guy David.
\newblock H\"older regularity of two-dimensional almost-minimal sets in
  {$\mathbb R^n$}.
\newblock {\em Ann. Fac. Sci. Toulouse Math. (6)}, 18(1):65--246, 2009.

\bibitem{d5}
Guy David.
\newblock {$C^{1+\alpha}$}-regularity for two-dimensional almost-minimal sets
  in {$\mathbb R^n$}.
\newblock {\em J. Geom. Anal.}, 20(4):837--954, 2010.

\bibitem{DLF1}
C.~De~Lellis and M.~Focardi.
\newblock Higher integrability of the gradient for minimizers of the {$2d$}
  {M}umford-{S}hah energy.
\newblock {\em J. Math. Pures Appl. (9)}, 100(3):391--409, 2013.

\bibitem{I1}
C.~De~Lellis, F.~Ghiraldin, and F.~Maggi.
\newblock A direct approach to {P}lateau's problem.
\newblock {\em J. Eur. Math. Soc. (JEMS)}, 19(8):2219--2240, 2017.

\bibitem{I4}
Camillo De~Lellis, Antonio De~Rosa, and Francesco Ghiraldin.
\newblock A direct approach to the anisotropic {P}lateau problem.
\newblock {\em Adv. Calc. Var.}, 12(2):211--223, 2019.

\bibitem{I2}
G.~De~Philippis, A.~De~Rosa, and F.~Ghiraldin.
\newblock A direct approach to {P}lateau's problem in any codimension.
\newblock {\em Adv. Math.}, 288:59--80, 2016.

\bibitem{I3}
Guido De~Philippis, Antonio De~Rosa, and Francesco Ghiraldin.
\newblock Rectifiability of varifolds with locally bounded first variation with
  respect to anisotropic surface energies.
\newblock {\em Comm. Pure Appl. Math.}, 71(6):1123--1148, 2018.

\bibitem{I5}
Guido De~Philippis, Antonio De~Rosa, and Francesco Ghiraldin.
\newblock Existence results for minimizers of parametric elliptic functionals.
\newblock {\em J. Geom. Anal.}, 30(2):1450--1465, 2020.

\bibitem{DPF}
Guido De~Philippis and Alessio Figalli.
\newblock Higher integrability for minimizers of the {M}umford-{S}hah
  functional.
\newblock {\em Arch. Ration. Mech. Anal.}, 213(2):491--502, 2014.

\bibitem{dugundji}
James Dugundji.
\newblock {\em Topology}.
\newblock Allyn and Bacon, Inc., Boston, Mass., 1966.

\bibitem{FM}
G.~A. Francfort and J.-J. Marigo.
\newblock Revisiting brittle fracture as an energy minimization problem.
\newblock {\em J. Mech. Phys. Solids}, 46(8):1319--1342, 1998.

\bibitem{Lab}
Camille Labourie.
\newblock Weak limits of quasiminimizing sequences.
\newblock {\em J. Geom. Anal.}, 31(10):10024--10135, 2021.

\bibitem{ll}
Camille Labourie and Antoine Lemenant.
\newblock Regularity improvement for the minimizers of the two-dimensional
  griffith energy.
\newblock {\em preprint}, 2021.

\bibitem{lempak}
A.~Lemenant and R.~Pakzad.
\newblock The bi-harmonic optimal support problem.
\newblock {\em in preparation}, 2023.

\bibitem{l3}
Antoine Lemenant.
\newblock Energy improvement for energy minimizing functions in the complement
  of generalized {R}eifenberg-flat sets.
\newblock {\em Ann. Sc. Norm. Super. Pisa Cl. Sci. (5)}, 9(2):351--384, 2010.

\bibitem{l2}
Antoine Lemenant.
\newblock Regularity of the singular set for {M}umford-{S}hah minimizers in
  {$\mathbb R^3$} near a minimal cone.
\newblock {\em Ann. Sc. Norm. Super. Pisa Cl. Sci. (5)}, 10(3):561--609, 2011.

\bibitem{lreview2}
Antoine Lemenant.
\newblock A selective review on {M}umford-{S}hah minimizers.
\newblock {\em Boll. Unione Mat. Ital.}, 9(1):69--113, 2016.

\bibitem{rigot}
S\'{e}verine Rigot.
\newblock Big pieces of {$C^{1,\alpha}$}-graphs for minimizers of the
  {M}umford-{S}hah functional.
\newblock {\em Ann. Scuola Norm. Sup. Pisa Cl. Sci. (4)}, 29(2):329--349, 2000.

\bibitem{simon}
Leon Simon.
\newblock {\em Lectures on geometric measure theory}, volume~3 of {\em
  Proceedings of the Centre for Mathematical Analysis, Australian National
  University}.
\newblock Australian National University, Centre for Mathematical Analysis,
  Canberra, 1983.

\end{thebibliography}

    \end{document}